\documentclass[12pt]{amsart}
\usepackage{amssymb}
\usepackage{amsmath}
\usepackage{setspace}
\usepackage{mathtools}
\usepackage{verbatim}
\usepackage{lipsum}
\usepackage{amsthm}
\usepackage{tikz-cd}
\usepackage{hyperref}
\newtheorem{theorem}{Theorem}[section]
\newtheorem{lemma}[theorem]{Lemma}
\newtheorem{corollary}[theorem]{Corollary}

\newtheorem{remark}[theorem]{Remark}

\newtheorem{proposition}[theorem]{Proposition}
\newtheorem{defn}[theorem]{Definition}

\newtheorem{example}[theorem]{Example}

\newcommand{\bb}[1]{\mathbb{#1}}

\newcommand{\cal}[1]{\mathcal{#1}}
\newcommand{\ff}[1]{\mathfrak{#1}}
\newcommand{\Spec}[1]{\operatorname{Spec}\, #1}

\newcommand{\sing}[1]{\operatorname{Sing} \, #1}
\newcommand{\supp}[1]{\operatorname{Supp} \, #1}
\newcommand{\exc}[1]{\operatorname{Exc} \, #1}
\newcommand{\codim}[1]{\operatorname{codim} \, #1}

\newcommand{\loc}[1]{\operatorname{loc}( #1)}

\title{MMP for  co-rank one foliations on  threefolds}
\author{Paolo Cascini and Calum Spicer} 

\address{Department of Mathematics, Imperial College London, 180 Queen's Gate, 
London SW7 2AZ, UK} 
\email{p.cascini@imperial.ac.uk}
\email{calum.spicer@imperial.ac.uk}

\begin{document}

\begin{abstract}
We prove existence of flips, special termination, the base point free theorem and, in the case of log general type, the existence of minimal models for F-dlt foliated  pairs of co-rank one on a $\mathbb Q$-factorial projective threefold. 

As applications, we show the existence of  F-dlt modifications and  F-terminalisations for foliated pairs and we show that foliations with canonical or F-dlt singularities admit non-dicritical singularities. Finally, we show abundance in the case of numerically trivial foliated  pairs.  
\end{abstract}

\maketitle
\tableofcontents

\section{Introduction}

The (classical) Minimal Model Program predicts that a complex projective manifold is either uniruled or it admits a minimal model, i.e. it is birational to a (possibly singular) projective variety with nef canonical divisor. 
Although this is still an open problem, many important cases of the program have been carried out successfully, e.g. in the case of varieties of dimension at most three and for varieties of general type. 
After the work of McQuillan \cite{McQuillan08} 
we expect a similar picture to hold  in the  theory of birational geometry of foliations
(see also \cite{Brunella00} and \cite{Mendes00}). 
More specifically, assuming that $X$ is a normal complex projective variety and $\cal F$ is a foliation with mild singularities, it is conjectured  that either $\cal F$ is uniruled, i.e. $X$ is covered by rational curves which are tangent to $\cal F$, or 
$\cal F$ admits a minimal model, i.e. $X$ is birational to a projective variety $Y$ such that the transformed foliation on $Y$ (cf. Section \ref{subs_transforms}) admits a nef canonical divisor (cf. Section \ref{s_definition}). 
Many of the main goals of the program were carried out successfully in the case of rank one foliations (cf. \cite{McQuillan08, McQ05} ) and, in any rank, it is expected to follow the main steps of Mori's program. 

\medskip 

The goal of this paper is to show the existence of flips (cf. Section \ref{s_ample_model}) for foliations of co-rank one on a complex projective threefold and present several applications, under some natural assumptions on the singularities. 

\subsection{Statement of main results}

In \cite{Spicer17} it was shown that given a foliated pair $(\cal F,\Delta)$, with some mild assumption on the singularities,  and given a $(K_{\cal F}+\Delta)$-negative extremal ray $R$, there is a morphism $\phi_R\colon X \rightarrow Y$, in the category of algebraic spaces, 
contracting only those curves $C$ such that $[C] \in R$.  Projectivity of $Y$ and the existence of flips were shown in some special cases, but not
in the generality needed to run the MMP.

Our first main result is to show in greater generality 
that if $\phi_R$ is a flipping contraction then the flip exists:

\begin{theorem}[= Theorem \ref{flipsexist} + Theorem \ref{t_canimpliesnondicritical}]
\label{T_flips_exist} 
Let $\mathcal F$ be a co-rank one foliation on a  $\bb Q$-factorial
 projective threefold $X$ and let $\Delta\ge 0$ such that $(\mathcal F,\Delta)$ 
is F-dlt.  Let $\phi\colon X \rightarrow Y$ be a $(K_{\cal F}+\Delta)$-flipping contraction. 

Then the $(K_{\cal F}+\Delta)$-flip exists.
\end{theorem}

Notice that F-dlt foliated pairs play the same role  as dlt log pairs in the classical MMP (see Definition \ref{d_fdlt} for a precise definition). 
\medskip 

Next, we turn to the question of constructing a minimal model of a foliated pair $(\cal F, \Delta)$. As in Mori's program, existence of minimal models would follow if one could show that any sequence of flips terminates.  We are unable to
show termination in complete generality, but we are able to show a weaker version of termination, i.e., termination
of flips with scaling, which suffices to show that minimal models exist in several cases of interest: 

\begin{theorem}[= Theorem \ref{t_existence-minmod} + Theorem \ref{t_canimpliesnondicritical}]
\label{T_min_model1}
Let $\mathcal F$ be a co-rank one foliation on a  $\bb Q$-factorial projective threefold $X$. 
Let $\Delta=A+B$ be a $\bb Q$-divisor such that $(\mathcal F,\Delta)$ is a F-dlt pair, 
$A\ge 0$ is an ample $\bb Q$-divisor and $B\ge 0$. Assume that there exists a $\bb Q$-divisor $D\ge 0$ such that 
$K_{\mathcal F}+\Delta\sim_{\mathbb Q}D$. 

Then $(\mathcal F,\Delta)$ admits a minimal model. 
\end{theorem}

%
%
%
See Section \ref{S_mmp_with_scaling} for a precise definition of a minimal model.

It is important to observe that Theorems \ref{T_flips_exist} and \ref{T_min_model1} 
make no assumptions on the singularities of $X$ other than $\bb Q$-factoriality. 
However, as we will see by Theorem \ref{t_canimpliesnondicritical} below, the output of the MMP (and the intermediary steps
of the MMP more generally) will be $\mathbb Q$-factorial varieties with klt singularities. Notice also that we first prove Theorems \ref{T_flips_exist} and \ref{T_min_model1} under the assumption that the foliation has non-dicritical singularities  (cf. Definition \ref{defn_non-dicrit}) but we later prove that, if $(\cal F,\Delta)$ is an F-dlt foliated  pair then $\cal F$ admits non-dicritical singularities (cf. Theorem \ref{t_canimpliesnondicritical}).

\medskip

Along the way to proving the termination of flips with scaling, we prove the following basepoint free theorem
for foliations which we expect will be of interest.  Observe that if $\cal F$ is a rank one surface foliation with $K_{\cal F}$ nef and big
then $K_{\cal F}$ is in general not semi-ample, see for instance \cite[Corollary IV.2.3]{McQuillan08}. On the other hand, we prove the following:

\begin{theorem}[= Theorem \ref{t_bpf}]
Let $\mathcal F$ be a co-rank one foliation on a $\bb Q$-factorial  projective variety $X$ of dimension at most three.  
Let $\Delta$ be a $\bb Q$-divisor such that $(\mathcal F,\Delta)$ is a F-dlt pair. Let $A\ge 0$ and 
$B\ge 0$ be $\bb Q$-divisors such that $\Delta= A+B$ and $A$ is ample. Assume that $K_{\mathcal F}+\Delta$ is nef. 

Then $K_{\mathcal F}+\Delta$ is semi-ample.  
\end{theorem}
Note that our result is in some sense optimal. Indeed, contrary to the case of varieties, if $\cal F$ is a foliation such that $K_{\cal F}$ is big and nef, then we cannot choose in general a divisor $\Delta\ge 0$ such that $\Delta\sim_{\mathbb Q}\epsilon K_{\cal F}$ for some $\epsilon>0$  and $(\cal F,\Delta)$ is F-dlt, as some of the components of $\Delta$ might be $\cal F$-invariant.

\subsection{Application to F-dlt modifications and F-terminalisations}

In the study of the birational geometry of varieties, dlt modifications and terminalisations have proven to be very useful tools.
The existence of these modifications follows from the MMP for varieties.   We prove foliated analogues of these modifications 
as a consequence
of our results on the foliated MMP (see Section \ref{s_invariant} for the definition of the number $\epsilon$ and Section \ref{subs_transforms} for the notion of the transform of a foliation
under a birational map):

\begin{theorem}[Existence of F-dlt modifications, = Theorem \ref{t_existencefdlt}]
Let $\cal F$ be a co-rank one foliation on a 
normal projective variety $X$ of dimension at most three.
Let $(\cal F, \Delta)$ be a foliated pair.

Then there exists a birational morphism $\pi\colon Y \rightarrow X$
 such that if $\cal G$ is the transformed foliation on $Y$ then
$(\cal G,\pi_*^{-1} \Delta+\sum \epsilon(E_i)E_i)$ is
F-dlt where the sum is taken over all the $\pi$-exceptional divisors
and  
\[(K_{\cal G}+\pi_*^{-1} \Delta +\sum \epsilon(E_i)E_i)+G= \pi^*(K_{\cal F}+\Delta)\]
where $G \geq 0$.

In particular, if $(\cal F, \Delta)$ is lc then $\pi$ only extracts
divisors of discrepancy $=-\epsilon(E_i)$.


Furthermore, we may choose $(Y, \cal G)$ so that 
\begin{enumerate}

\item \label{i_lccentres1} if $Z$ is an lc centre of 
$(\cal G, \pi_*^{-1} \Delta+\sum \epsilon(E_i)E_i)$ then $Z$ is contained in a codimension one lc centre of $(\cal G, \pi_*^{-1} \Delta+\sum \epsilon(E_i)E_i)$,

\item $Y$ is $\bb Q$-factorial and

\item $Y$ is klt.
\end{enumerate}
\end{theorem}

\begin{theorem}[Existence of F-terminalisations, =Theorem \ref{t_existenceFterminalization}]
Let $\cal F$ be a co-rank one
foliation on a normal projective variety $X$ of dimension at most three.

Then there exists a birational morphism $\pi\colon Y \rightarrow X$
such that 
\begin{enumerate}
\item  if $\cal G$ is the transformed foliation on $Y$, then $\cal G$
is F-dlt, canonical and terminal along $\sing Y$,

\item $Y$ is klt and $\bb Q$-factorial and

\item $K_{\cal G}+E = \pi^*K_{\cal F}$ where $E \geq 0$.
\end{enumerate}
\end{theorem}

Using similar ideas we are also able to prove the following:

\begin{theorem}[=Theorem \ref{t_canimpliesnondicritical}]
Let $(\cal F, \Delta)$ be a foliated pair on a normal projective threefold $X$.
Suppose that $(\cal F, \Delta)$ is canonical. 

Then $\cal F$ has non-dicritical singularities (cf. Definition \ref{defn_non-dicrit}).
\end{theorem}

Observe that we do not require the smoothness of $X$.
We expect that this result will be useful in the study of foliation singularities.

\subsection{Application to foliation abundance}
It is a direct consequence of \cite[Theorem 2]{LPT11} that if $X$ is a projective manifold with $\cal F$ a co-rank one foliation with canonical singularities 
and $c_1(K_{\cal F}) =0$
then $K_{\cal F}$ is torsion.
When $X$ is a threefold we extend this result to 
the log situation where we consider $\cal F$ together with a boundary $\Delta$, as well as weakening the hypotheses
on the singularities. 

\begin{theorem}[=Theorem \ref{threefoldlogabundancezero}]
Let $\cal F$ be a co-rank one foliation
on a projective threefold $X$.
Let $(\cal F, \Delta)$ be a foliated pair with log canonical 
foliation singularities.
Suppose that $c_1(K_{\cal F}+\Delta) = 0$.

Then $\kappa(K_{\cal F}+\Delta) = 0$.
\end{theorem}

\subsection{Sketch of the Proof}

We first give a rough outline of our approach to the existence of flips.
Let us focus on a special case first: suppose $X$ is a smooth threefold and
$\cal F$ is the foliation induced by a fibration $f\colon X \rightarrow B$ onto a curve $B$ and with simple normal crossing fibres.
It is easy to compute that $K_{\cal F} = K_{X/B}-\sum (r_i-1)F_i$
 where the sum runs over the components of fibres with multiplicity $r_i$.
 
Let $C \subset X$ be a $K_{\cal F}$-flipping curve (cf. Section \ref{s_ample_model}).  It was shown in \cite{Spicer17} that
$C$ is tangent to $\cal F$ (cf. Definition \ref{tangtransdef}).  Let $T = f^*(f(C))_{red}$ and let $S$ be a component of 
$T$ containing $C$.  Note that $T$ is the largest connected $\cal F$-invariant divisor containing $C$ (cf. Section \ref{s_invariant}). 
A direct computation shows that $K_{\cal F}\vert_S = (K_X+T)\vert_S$ and, 
in particular, it follows that $(K_X+T)\cdot C = K_{\cal F}\cdot C<0$ and since $(X, T)$ is log canonical 
it follows that the $K_{\cal F}$-flip
can be realised as a log flip in the classical MMP.

In general, if $X$ is a normal threefold and $\cal F$ is a co-rank one foliation on $X$ with mild singularities and which admits a flipping contraction (cf. Section \ref{s_ample_model}), 
 we want to realise the $K_{\cal F}$-flip as a $(K_X+\sum S_i)$-flip where the $S_i$ are all the 
$\cal F$-invariant divisors meeting $C$.  There are two technical challenges here.  The first is that the
$S_i$ are not necessarily algebraic divisors, and if $C \subset \text{sing}(\cal F)$, they might not even be defined
analytically locally around $C$: instead they might only exist as formal divisors in the formal completion of $X$ along $C$.
The second challenge is to control the singularities of the pair $(X, \sum S_i)$.

To handle the first challenge we develop an extension of the classical MMP
to the formal setting.  We adapt some approximation results of Artin/Elkik to show that
we can approximate the $S_i$ by algebraic divisors $S'_i$ on an \'etale neighbourhood $U$ of $C$
in the sense that $S_i = S'_i$ on some infinitesimal neighbourhood of $C$.
By choosing a sufficiently close approximation, it follows that the $(K_U+\sum S'_i)$-flip coincides with 
the $(K_{X}+\sum S_i)$-flip.  We emphasise that we are only approximating the $S_i$ and we 
are not approximating the foliation $\cal F$, i.e., the $S'_i$ are not necessarily invariant divisors
of some other foliation $\cal F'$. Indeed, it is well known that it is not in general possible
to approximate a foliation with non-convergent separatrices by one with convergent separatrices (see Section \ref{s_approx} and 
\ref{s_approx2}).

Controlling the singularities of $(X, \sum S_i)$ is done by way of several results which provide bounds between
the singularities of $\cal F$ and those of $(X, \sum S_i)$. There is some additional difficulty arising from the fact
that the $S_i$ are not necessarily $\bb Q$-Cartier (even if $X$ is $\bb Q$-factorial).  
This can be handled by somewhat standard arguments
from the classical MMP (see Section \ref{s_flip}).

\medskip 

The remaining results on the MMP for foliations of co-rank one on a threefold (i.e., special termination, basepoint free theorem, termination of flips with scaling,
existence of F-dlt modifications, etc.)
are mostly a consequence of direct translations of standard ideas and strategies from the classical MMP
to the foliated setting.  Nevertheless there are some intriguing issues which arise and which are unique to the foliated setting
(e.g., foliations may admit infinitely many lc centres).

\medskip 

Finally, our abundance type result for foliations with $c_1(K_{\cal F}) = 0$ is a consequence of a rather lengthy and delicate
case by case analysis.  Some central ingredients are
Touzet's results on foliations with pseudo-effective conormal bundle and a careful
analysis of families of surfaces foliations with trivial canonical class (see Section \ref{s_abundance}).

\subsection{Acknowledgements}
Both  the authors were funded by EPSRC. We would like to thank J. M\textsuperscript cKernan, M. McQuillan, J. V. Pereira and R. Svaldi 
for many useful discussions. 
We are particularly grateful to the
referees for carefully reading the paper  and to help us  to improve the presentation of the paper considerably.
  The first author would like
to thank the National Center for Theoretical Sciences in Taipei and Professor J. A. Chen for their generous hospitality, where some of the
work for this paper was completed.

\section{Preliminary results}
We work over the field of complex numbers $\bb C$. Throughout the paper, a variety is a complex analytic space. 

Let $X$ be a normal variety and let $V\subset X$ be a closed subvariety.  Let $\hat X$ be the formal completion of $X$ along $V$. A {\bf formal divisor } $D$ on $\hat X$ is a union of distinct integral 
formal subschemes of pure codimension one. 
Let $I_D$ be the ideal sheaf of $D$. We say that  
$D$ is {\bf $\mathbb Q$-Cartier} if
$(I_D^{\otimes n})^{**}$ is locally of the form $f\cdot \cal O_{\widehat{X}}$ where $f$ is a local section of $\cal O_{\widehat{X}}$ and $n$ is a positive integer.
We say that $D$ does not contain $V$ in its support if $f$ does not vanish along $V$. 
Note that, in this case, if we denote by $\nu\colon V^{\nu}\to V$ the normalisaton of $V$, we can define the 
pull-back of $D$ to $V^{\nu}$ as a $\mathbb Q$-Cartier divisor on $V^{\nu}$.

Let $X$ be a smooth variety of dimension $n$ and let $D$ be a reduced divisor, we say that $D$ is {\bf normal crossing}, or that $(X,D)$ is a normal crossing pair,  if, for each closed point $x\in D$, there exist local formal coordinates $x_1,\dots,x_n$ such that $D$ is defined by $\{x_1\cdot\dots\cdot x_r=0\}$ for some $1\le r\le n$.
Note that this definition works equally well even if $\hat{X}$ is the formal completion of
a smooth variety $X$ along a closed subvariety  and $D \subset \hat{X}$ is a formal divisor.
 We say that a divisor $D$ on a smooth variety $X$ is {\bf simple normal crossing} if it is normal crossing and every irreducible component of $D$ is smooth. 

Given a normal variety $X$, we denote by $\Omega^1_X$ its sheaf of K\"ahler differentials and,
by $T_X:=(\Omega^1_X)^*$ its tangent sheaf. 
For any positive integer $p$, we denote $\Omega_X^{[p]}\coloneqq (\Omega_X^p)^{**}$.

Given a  $\bb Q$-divisor $\Delta$ on a normal variety $X$, we write $ \lfloor \Delta \rfloor$ for the round-down of $\Delta$ and $\{\Delta\}$ for the
fractional part of $\Delta$, i.e. $\{\Delta\}=\Delta - \lfloor \Delta \rfloor$.
 A $\bb Q$-factorial variety is a normal variety $X$ on which every divisor is $\bb Q$-Cartier. A  birational map $f\colon X\dashrightarrow Y$ between normal varieties is a {\bf birational contraction} if $f^{-1}$ does not contract any divisor.

We refer to \cite[Section 2.3]{KM98} for the classical definitions of singularities 
(e.g. klt, log canonical, etc.. ) appearing in the minimal model program. 
In particular, a {\bf log pair} $(X,\Delta)$ 
 consists of a normal variety $X$ and a $\mathbb Q$-divisor $\Delta$ with coefficients in $(0,1]$ and such that $K_X+\Delta$ is $\mathbb R$-Cartier. Note that if the coefficients of $\Delta$ are rational, then $K_X+\Delta$ is $\mathbb Q$-Cartier. If $S$ is an irreducible component of $\lfloor \Delta \rfloor$, and $\nu\colon S^{\nu}\to S$ is its normalisation, then we may write
 \[
 (K_X+\Delta)|_{S^{\nu}}=K_{S^{\nu}}+\Theta
 \]
where $\Theta$ is an effective $\mathbb R$-divisor on $S^{\nu}$ called the {\bf different} of $(X,\Delta)$ with respect to $S$ (cf. 
\cite[(4.2.9)]{Kollar13}).

Let $(X,\Delta)$ be a log pair,  let $S\subset X$ be a prime divisor contained in the support of $\lfloor \Delta \rfloor$
and let $\widehat X$ be the formal completion of $X$ along $S$. Assume that $T$ is a $\mathbb Q$-Cartier formal divisor on $\widehat X$ which does not contain $S$ in its support. Then we define the different of $(\widehat X,\Delta+T)$ with respect to $S$, as $\Theta+\nu^*T$, where $\Theta$ is the different of $(X,\Delta)$ with respect to $S$.  
 
We say that a normal variety $X$ is {\bf potentially klt} if there exists a $\bb Q$-divisor $\Delta \geq 0$ such that $(X, \Delta)$ is klt. We say that a normal variety $X$ is {\bf\'etale locally potentially klt} if for all $x \in X$ there is an \'etale
neighborhood $U$ of $x$ such that $U$ is potentially klt.

\subsection{Basic definitions}\label{s_definition}

A {\bf foliation} on a normal variety $X$ is a coherent subsheaf $\cal F \subset T_X$ such that
\begin{enumerate}
\item $\cal F$ is saturated, i.e. $T_X/\cal F$ is torsion free, and

\item $\cal F$ is closed under Lie bracket.
\end{enumerate}

The {\bf rank} of $\cal F$ is its rank as a sheaf.  Its {\bf co-rank} is its co-rank as a subsheaf
of $T_X$.

Let $X$ be a normal variety and let $\cal F$ be a rank $r$ foliation
on $X$.  We can associate to $\cal F$ a morphism
\[\phi\colon \Omega^{[r]}_X \rightarrow \cal O_X(K_{\cal F})\]
defined by taking the double dual of the  $r$-wedge product of the map $\Omega^{[1]}_X\to \cal F^*$, induced by the inclusion
$\cal F \subset T_X$.  This yields a map
\[
\phi'\colon (\Omega^{[r]}_X \otimes \cal O_X(-K_{\cal F}))^{**} \rightarrow \cal O_X\]
and we define the {\bf singular locus} of $\cal F$,   denoted by $\sing \cal F$, to be the cosupport of the image of $\phi'$.
A {\bf canonical divisor} of $\cal F$ is a divisor $K_{\cal F}$
such that $\cal O_X(-K_{\cal F}) \cong \det (\cal F)$.
We define the {\bf normal sheaf} of $\cal F$ as $\mathcal N_{\cal F}\coloneqq (T_X/\cal F)^{**}$. The {\bf conormal sheaf} $\mathcal N_{\cal F}^*$ of $\cal F$ is the dual of $\cal N_{\cal F}$.
If $\cal F$ is a foliation of co-rank one then, by abuse of notation, we denote by $N^*_{\cal F}$ a divisor associated to $\cal N^*_{\cal F}$. 


\subsection{Invariant subvarieties}\label{s_invariant}

Let $X$ be a normal variety and let $\cal F$ be a rank $r$ foliation on $X$. 
Let $S\subset X$ be a subvariety. Then  $S$ is said to be  {\bf $\cal F$-invariant}, or {\bf invariant by }$\cal F$, if for any open subset $U\subset X$ and any section $\partial \in H^0(U,\cal F)$, we have that 
\[ \partial (\mathcal I_{S\cap U})\subset \mathcal I_{S\cap U},
\]
where $\mathcal I_{S\cap U}$ denotes the ideal sheaf of $S\cap U$. 
If  $D\subset X$ is a prime divisor then we define $\epsilon(D) = 1$ if $D$ is not $\cal F$-invariant and $\epsilon(D) = 0$ if it is $\cal F$-invariant.


Let $X$ be a normal variety and let $\cal F$ be a foliation on $X$.  Let $\widehat{X}$ be the formal completion of $X$
along a proper closed subvariety $V \subset X$.  We say that a formal subvariety $S \subset \widehat{X}$ is a {\bf formal $\cal F$-invariant divisor} if it is a formal divisor which is $\widehat {\cal F}$-invariant, where $\widehat{\cal F}$ is
the restriction of $\cal F$ to $\widehat{X}$.

\subsection{Transform of a foliation under a rational map}
\label{subs_transforms}

Let $X$ be a normal variety and let $\cal F$ be a foliation on $X$.
Let $\phi\colon Y \dashrightarrow X$ be a dominant map and assume that there exist smooth open subsets
$U \subset X$ and $V \subset Y$ such
that the restriction $\phi|_V\colon V\to U$ is a   morphism.

Let $\cal F_U$ denote the restriction of $\cal F$ to $U$. Then the morphism $\mathcal N^*_{\cal F_U}\to \Omega^1_U$ induces a morphism $(\phi|_V)^*\mathcal N_{\cal F_U}\to \Omega^1_V$ and, therefore, a foliation $\cal G_V$ on $V$. 
We may extend $\cal G_V$ to a foliation on all of $Y$.
Indeed, we may take $\cal G$ to be the saturated subsheaf of $T_X$ whose restriction to $V$ is $\cal G_V$. 
It is easy to see that $\cal G$ is closed under Lie bracket, since it is closed under Lie bracket over a dense open subset. We will refer to $\cal G$ as the {\bf induced foliation} on $Y$ by $\phi$.   
If $\phi\colon Y\to X$ is a morphism, then the induced foliation is called the  {\bf pulled back foliation} and we denote it by $\phi^{-1}\cal G$. 
If $f\colon X\dashrightarrow X'$ is a birational map, then the induced foliation on $X'$ by $f^{-1}$ is called the {\bf transformed foliation} of $\cal F$ by $f$ and we will denote it by $f_*\cal F$.

\subsection{Foliation singularities}
Frequently, in birational geometry, it is useful to consider
pairs $(X, \Delta)$ where $X$ is a normal variety and $\Delta$ is a $\bb Q$-Weil divisor
such that $K_X+\Delta$ is $\bb Q$-Cartier.  By analogy, we define:

\begin{defn} Let $X$ be a normal variety. 
A {\bf foliated pair} $(\cal F, \Delta)$ on $X$ consists of  a foliation $\cal F$ on $X$ and a $\bb R$-divisor 
 divisor $\Delta$
such that $K_{\cal F}+\Delta$ is  $\bb R$-Cartier.
\end{defn}
Note that if $(\cal F,\Delta)$ is foliated pair and $\Delta$ is a $\mathbb Q$-divisor, then $K_{\cal F}+\Delta$ is $\mathbb Q$-Cartier. 
Note also that we are typically interested only in the case when $\Delta \geq 0$,
although it simplifies some computations to allow $\Delta$ to have negative coefficients.

Given a birational morphism $\pi\colon  \widetilde{X} \rightarrow X$ 
and a foliated pair $(\cal F, \Delta)$ on $X$,
let $\widetilde{\cal F}$ be the pulled back foliation on $\tilde{X}$
and $\pi_*^{-1}\Delta$ be the strict transform of $\Delta$ in $\widetilde X$. 
We may write
\[
K_{\widetilde{\cal F}}+\pi_*^{-1}\Delta=
\pi^*(K_{\cal F}+\Delta)+ \sum a(E, \cal F, \Delta)E.\]
where $\pi_*K_{\widetilde{\cal F}}=K_{\cal F}$,  the sum runs over all the prime $\pi$-exceptional divisors on $\widetilde X$ and  the rational number $a(E,\cal F,\Delta)$ is called the {\bf discrepancy} of $(\cal F,\Delta)$ with respect to $E$. If $\Delta=0$, then we will simply denote $a(E,\mathcal F, 0)$ by $a(E,\mathcal F)$.  Building on the work of McQuillan (e.g. see \cite[Definition I.1.5]{McQuillan08}), we define: 
\begin{defn}\label{d_canonical} Let $X$ be a normal variety and let $(\cal F,\Delta)$ be a foliated pair on $X$. 
We say that  $(\cal F, \Delta)$ is {\bf terminal} (resp. {\bf canonical},  {\bf log terminal}, {\bf log canonical}) if
$a(E, \cal F, \Delta) >0$ (resp. $\geq 0$, $> -\epsilon(E)$, $\geq -\epsilon(E)$),  for any birational morphism  $\pi\colon \widetilde X\to X$ and for any prime $\pi$-exceptional divisor $E$ on  $\widetilde X$.

We say that a foliation $\cal F$ is terminal (resp. canonical, log canonical) if the foliated pair $(\cal F,0)$ is such. 

We say that the foliated pair $(\cal F,\Delta)$ is {\bf klt} if $\lfloor \Delta\rfloor=0$ and $a(E,\cal F,\Delta)> -\epsilon(E)$ for any birational morphism  $\pi\colon \tilde X\to X$ and for any $\pi$-exceptional prime divisor $E$ on  $\tilde X$.

Let $P \in X$ be a, not necessarily closed, point of $X$.
We  say  that the foliated pair $(\cal F, \Delta)$ is {\bf terminal} (resp. {\bf canonical, log canonical})
{\bf at $P$} if
for any birational morphism $\pi\colon \widetilde X\to X$ and for any  $\pi$-exceptional  divisor $E$ on $\widetilde X$ whose centre in $X$ is the Zariski closure $\overline P$ of $P$,
we have that the discrepancy of $E$ is $>0$ (resp. $\geq 0$, $\geq -\epsilon(E)$).
\end{defn}

Notice that these notions are well defined, i.e., $\epsilon(E)$ and $a(E, \cal F, \Delta)$
are independent of $\pi$.

Observe that in the case where $\cal F = T_X$, no exceptional divisor
is invariant and so this definition recovers the usual
definitions of (log) terminal, (log) canonical.

\begin{remark} 
\label{remark_inv_not_lc}
It follows from Definition \ref{d_canonical} that 
if some component of $\supp \Delta$ is $\cal F$-invariant,
then $(\cal F, \Delta)$ is not log canonical. 
Indeed, let $D$ be an $\cal F$-invariant component of $\Delta$ with 
coefficient $a>0$.  Let $p \in D$ be a general point so that $X, D$ and $\cal F$ are all smooth at $p$.
Let $b\colon X' \rightarrow X$ be the blow up at $p$ and let $\cal F'=b^{-1}\cal F$. Then $b$ extracts a single $\cal F'$-invariant divisor $E$ of discrepancy 
\[
a(E,\cal F,\Delta)=n-1-a
\] where $n$ is the dimension of $X$. 
Let $D'$ be the strict transform of $D$ and observe that $Z \coloneqq D'\cap E$ is contained in $\sing \cal F'$.
Let $b'\colon X'' \rightarrow X'$ be the blow up of $X'$ in $Z$.  Then $b'$ extracts a divisor $E'$ of discrepancy 
\[
a(E',\cal F,\Delta)=n-2a-1.
\]
We may now blow up  $E'\cap D''$, where  $D''$ is the strict transform of  $D'$ in $X''$, and, continuing this way, we produce a sequence of extractions of divisors with discrepancy $n-ka-1$ where $k =1 , 2, \dots$.  Thus, if $k$ is sufficiently large, then we extract a divisor $F$ such that 
\[ a(F,\cal F,\Delta)<0=\epsilon(F)
\] and, in particular, $(\cal F, \Delta)$ is not log canonical.
\end{remark}

\begin{defn}
Given a foliated pair $(\cal F, \Delta)$, we say that $W \subset X$ is a {\bf log canonical
centre} (in short, lc centre) of $(\cal F, \Delta)$ provided $(\cal F,\Delta)$ is log canonical at the generic
point of $W$ and there is some divisor $E$ of discrepancy $-\epsilon(E)$ on
some model of $X$ dominating $W$.
\end{defn}

Notice
that in the case that $\epsilon(E) = 0$ for all exceptional divisors $E$
over a centre the notions of log canonical and canonical coincide.  In this case,
we will still refer to canonical centres as log canonical centres.

We also remark that any $\cal F$-invariant divisor is an lc centre
of $(\cal F, \Delta)$.

\medskip 

We have the following nice characterisation due to \cite[Corollary I.2.2.]{McQuillan08}:

\begin{proposition}
Let $0 \in X$ be a normal surface germ with a terminal foliation $\cal F$ of rank one.  

Then there exists a cyclic cover
$\sigma\colon  Y \rightarrow X$ such that $Y$ is a smooth surface and  $\sigma^{-1}\cal F$ is a smooth foliation.
\end{proposition}

We also make note of the following easy fact:

\begin{lemma}
\label{lem_crep_discrep}
Let $\pi\colon Y\to X$ be a proper birational morphism between normal varieties. Let $(\cal F, \Delta)$ be a foliated pair on $X$ and let $\cal G$ be the pulled back foliation of $\cal F$ on $Y$. 
Write $\pi^*(K_{\cal F}+\Delta) = K_{\cal G}+\Gamma$. 

Then
$a(E, \cal F, \Delta) = a(E, \cal G, \Gamma)$ for all $E$.
\end{lemma}

We will make frequent use of the following consequence of the negativity lemma: 

\begin{lemma}\label{l_negativity}
Let $\phi\colon X\dashrightarrow X'$ be a birational map between normal varieties and let 
\begin{center}
\begin{tikzcd}
X \arrow[rd,"f" ' ]\arrow[dashrightarrow]{rr}{\phi} & &X' \arrow{dl}{f'}\\
&Y&  \\
\end{tikzcd}
\end{center}
be a commutative diagram, where $Y$ is a normal variety and $f$ and $f'$ are proper birational morphisms. 
Let $(\cal F,\Delta)$ be a  foliated pair on $X$. Let $\cal F'=\phi_*\cal F$ and let $\Delta'$ be a 
$\mathbb Q$-divisor on $X'$ such that $f_*\Delta=f'_*\Delta'$. 
Assume that $-(K_{\cal F}+\Delta)$ is $f$-ample and $K_{\cal F'}+\Delta'$ is $f'$-ample. 

Then, for any valuation $E$ on $X$, we have
\[
a(E,\cal F,\Delta)\le a(E,\cal F',\Delta').
\]
Moreover, the inequality holds if $f$ or $f'$ is not an isomorphism above the generic point of the centre of $E$ in $Y$. 
\end{lemma}
\begin{proof}
The proof is the same as \cite[Lemma 3.38]{KM98}.
\end{proof}

\medskip

We now recall some facts from \cite{Cano} on simple  singularities.  We say that the numbers  $\lambda_1,\dots,\lambda_l\in \mathbb C^*$ satisfy the {\bf non-resonant condition} if 
for any non-negative integers $a_1,\dots,a_l$ such that 
 $\sum a_i\lambda_i = 0$ we have that $a_i = 0$ for all $i=1,\dots,l$.

\begin{defn}
\label{defn_simple}
Let $\cal F$ be a co-rank one foliation on a smooth variety $X$ of dimension $n$. 
We say that $p \in X$ is a {\bf simple singularity} for $\cal F$
provided that, in formal coordinates $x_1,\dots,x_n$ around $p,$ $N^*_{\cal F}$ is generated by a $1$-form which is 
in one of the following two forms, for some $1 \leq r \leq n$:

\begin{enumerate}
\item There are $\lambda_1,\dots,\lambda_r \in \bb C^*$, which satisfy the non-resonant condition and such that
$$\omega = x_1\cdots x_r\cdot\sum_{i = 1}^r \lambda_i \frac{dx_i}{x_i}.$$

\item There is an integer $k \leq r$ such that
$$\omega = x_1\cdots x_r\cdot \left(\sum_{i = 1}^kp_i\frac{dx_i}{x_i} + 
\psi(x_1^{p_1}\cdots x_k^{p_k})\sum_{i = 2}^r \lambda_i\frac{dx_i}{x_i}\right)$$
where $p_1,\dots,p_k$ are positive integers without a common factor, $\psi(s)$
is a formal power series which is not a unit, and the numbers $\lambda_2,\dots,\lambda_r \in \bb C^*$ satisfy the non-resonant condition.

\end{enumerate}
The integer $r$ is called the {\bf dimension-type} of the singularity. The $r$-uple $(\lambda_1,\dots,\lambda_r)$ in (1) (resp. the $(r-1)$-uple $(\lambda_2,\dots,\lambda_r)$ in (2)) is called the {\bf residual spectrum} of the singularity (cf. \cite[Remark 20]{Cano}). 

If $(X,D)$ is a normal crossing pair and $\cal F$ is a co-rank one foliation on $X$ then we say that 
 $\cal F$ has {\bf simple singularities adapted} to $D$ if $\cal F$ has simple singularities and, for every $p\in X$, we may choose formal coordinates around $p$ as above and such that the divisor $D\cup \{x_1\cdots x_r=0\}$ is also normal crossing at $p$
 (cf. \cite[Definition 3, Definition 13 and Definition 14]{Cano}).

A {\bf stratum} of $\sing \cal F$ is a closed subvariety $Z\subset \sing\cal F$ such that for all $p \in Z$
and coordinates $x_1, \dots, x_n$ as above, in the formal neighbourhood of $X$ at $p$, 
we have that $Z$ is a stratum of $\{x_1\cdots x_r = 0\}$.

\end{defn}

By Cano \cite[Main Theorem]{Cano}, every co-rank one foliation $\cal F$ on a smooth threefold $X$ admits a resolution
$\pi\colon X'\to X$ by blow ups centred in the singular locus of the foliation,
such that the transformed
foliation has  simple singularities.
By allowing blow ups in centres tangent to the foliation (but perhaps not contained in $\sing \cal F$)
we may get that the transformed
foliation has  simple singularities adapted to $\exc \pi$.  More generally, we may perform
a sequence of blow ups so that the transform of $\cal F$ has simple singularities adapted
to the transform of any divisor $D$ on $X$ (cf. \cite[Theorem 3 and Section 4.5]{Cano}).

We remark that if $\cal F$ is an algebraic foliation defined on an algebraic threefold
and $D$ is an algebraic divisor
then we only need to blow up in algebraic centres.  Indeed, for centres
contained in $\sing \cal F$ this is obvious.  If $Z$ is a centre which needs to be blown up and
which is not contained in $\sing \cal F$ then $Z$ is either contained in the singular locus of $D$ or is
contained in the tangency
locus of $\cal F$ and $D$.  Since $\cal F$ and $D$ are algebraic, these loci are also algebraic. Hence $Z$
is algebraic.



\begin{lemma}\label{l_simplecanonical}
Let $X$ be a smooth variety and let $\cal F$ be a co-rank one foliation with simple singularities on $X$. 

Then $\cal F$ is canonical. 
\end{lemma}
\begin{proof}
Let $\pi\colon X' \rightarrow X$ be a birational morphism, let $E \subset X'$ be a divisor
and let $Z = \pi(E)$.  If $Z$ is not contained in $\sing \cal F$, then by shrinking about the generic
point of $Z$ we may apply \cite[Lemma 3.10]{AD13} to conclude that $a(E, \cal F) \geq 0$.

Assume now that $Z \subset  \sing \cal F$. By shrinking about the generic
point of $Z$, we may assume that $Z$ is smooth and, 
by Zariski's Lemma (cf.  \cite[Lemma 2.45]{KM98}),
after possibly replacing $X'$ by a higher model, we may assume that $\pi$ is a composition of blow-ups of subvarieties centred on $Z$.  By induction on the number of blow ups, it suffices to show  that if $b\colon Y \rightarrow X$
is the blow up of $Z$ then 
\begin{enumerate}
\item[(i)] if $E_0$ is the $b$-exceptional divisor  then $a(E_0, \cal F) \geq 0$; and 

\item[(ii)] $b^{-1}\cal F$ has simple singularities in a neighbourhood of 
$b^{-1}(z)$, where $z \in Z$ is a general point. 
\end{enumerate} 

To prove (i), observe that if $W$ is a minimal stratum of $\sing \cal F$ containing $Z$ and
$\omega$ is a $1$-form defining $\cal F$ at a general point of $Z$,
then $b^*\omega$ vanishes to order $\codim W-1$ at the generic point of $E_0$.
Thus, 
\[
a(E_0, \cal F) = a(E_0, X) - (\codim W-1) = \codim Z-\codim W \geq 0
\]
and (i) follows.

We now prove (ii).
%
%
Since we are only concerned about the behaviour at the general point of $Z$, we may assume without loss 
of generality that $Z$ is contained in a stratum of $\sing \cal F$ which meets no lower dimensional strata.
Thus, If $p\in Z$ is a closed point then we may find formal coordinates $x_1, ..., x_r, y_1, ..., y_{n-r}$ around $p$, where $n$ is the dimension of $X$, $r$ is the dimension type of the singularity and if $\omega$ is a $1$-form defining the foliation near $Z$ then
 $\omega$ is in one of the two forms in  Definition \ref{defn_simple} and
\[
Z =\{x_1 = ... = x_r = y_1 = ... = y_l = 0\}
\]
for some $l \leq n-r$.
Let $\omega' \coloneqq \frac{1}{x_1\cdots x_r} \omega$.

Fix $i=1,\dots,r$. Consider the chart for the blow up given by 
\[
x_1 = x'_ix'_1 \, \dots\,  x_i = x_i' \, \dots\, x_r=x'_ix'_r\quad  y_1 = x'_iy'_1\,\dots\, y_s = x'_iy'_s.
\]
If $\omega$ has a simple singularity of type (1) then 
\[\omega'  = \sum_{j = 1}^r \lambda_j \frac{dx_j}{x_j}\] and, in this chart, \[b^*\omega' = \sum_{j = 1}^r \lambda'_j \frac{dx'_j}{x'_j}\]
where $\lambda'_j = \lambda_j$ for $j \neq i$ and $\lambda'_i = \sum_{j}\lambda_j$.
In particular, notice that $\lambda'_1,\dots,\lambda'_r\in \mathbb C^*$ satisfy the non-resonant condition, since it is a positive integral linear combination of a non-resonant spectrum, and so $b^*\omega'$ is still a simple singularity of type (1).
A similar computation holds in all other charts. Thus $b^*\omega'$ defines a simple singularity of type (1).

So suppose that we have a simple singularity of type (2). Then
\[\omega' = \sum_{j = 1}^kp_j\frac{dx_j}{x_j} + 
\psi(x_1^{p_1}\cdots x_k^{p_k})\sum_{i = 2}^r \lambda_j\frac{dx_j}{x_j}.\]
Again, using coordinates for the blow up as above we see that if $1 \leq i \leq k$ then
\[b^*\omega' = \sum_{j = 1}^k p'_j \frac{dx'_j}{x'_j} + \psi((x'_1)^{p'_1}\cdots (x'_k)^{p'_k})\sum_{i = 2}^r \lambda'_j \frac{dx'_j}{x'_j}\]
and if $i >k$ we have
\[b^*\omega' = \sum_{j = 1}^k p'_j \frac{dx'_j}{x'_j} +p'_i \frac{dx'_i}{x'_i}+ \psi((x'_1)^{p'_1}\cdots (x'_k)^{p'_k}(x'_i)^{p'_i})\sum_{i = 2}^r \lambda'_j \frac{dx'_j}{x'_j}\]
where $p'_j =  p_j$ if $j \neq i$ and $p'_i = p_1+\dots+p_k$ and $\lambda'_j$ is defined as above.
Again, it follows that $\lambda'_2,\dots,\lambda'_r$ satisfy the non-resonant condition and, similarly as above, $b^*\omega$ defines a simple singularity of type (2).

We remark that in both these cases the exceptional divisor of the blow up is invariant.

Thus, $b^{-1}\cal F$ has simple singularities in a neighbourhood 
of $b^{-1}(z)$ and (ii) follows. 
\end{proof}

The converse of this statement is false (e.g. see \cite[Example 2.16]{Spicer17}).

%
%
%

\medskip


\begin{defn}
\label{defn_non-dicrit}
Given a normal variety $X$ and a foliation $\cal F$ on $X$, we say that 
$\cal F$ has {\bf non-dicritical} singularities if for any closed point $q \in X$ and any proper birational morphism $\pi\colon X'\to X$
such that $\pi^{-1}(q)$ is a divisor we have that each component of $\pi^{-1}(q)$ is invariant by $\pi^{-1}\cal F$.
\end{defn}

\begin{example}
Let $\lambda \in \bb R$.  Consider the rank one foliation $\cal F_\lambda$ on $\bb C^2$
generated by $x\partial_x+\lambda y\partial_y$.
For $\lambda \in \bb Q_{> 0}$ we can see that $\cal F_\lambda$ is dicritical,
and otherwise is non-dicrtical.
See \cite[pg. 7]{Brunella00} for an explicit resolution of $\cal F_{\lambda}$ when
$\lambda \in \bb Q_{>0}$.
\end{example}

\begin{defn}
\label{tangtransdef}
Let  $X$ be a normal variety and let $\cal F$ be a co-rank one foliation with non-dicritical singularities.

We say that a subvariety $W \subset X$ (possibly contained in $\sing X$ or $\sing \cal F$)
is {\bf tangent} to $\cal F$ if for any birational morphism 
$\pi\colon X'\to X$  and any 
 divisor $E$ on $X$ such that $E$
dominates $W$, we have that $E$ is $\cal F'$-invariant, where $\cal F'$ is the pulled back foliation on $X'$.
Otherwise, we say that  $W \subset X$ is {\bf transverse} to $\cal F$. 
\end{defn}

Note that the  definitions above differ slightly from the usual ones, but, in our opinion, they are
more flexible when working with singular varieties. In particular, if $W\subset X$ is a subvariety which is not contained 
in $\sing X \cup \sing\cal F$, then our definition of tangency coincides with the classic one. 
Note also that for divisors the notions of tangency and invariance coincide. 

Assume now that $X$ is a normal variety, $\mathcal F$ is a co-rank one foliation on $X$ and  $W\subset X$ is an irreducible  subvariety which is not contained in $\sing X \cup \sing\cal F$ and which is transverse to $\cal F$. Let $\nu\colon W^{\nu}\to W$ be its normalisation. 
Let $U\subset X$ be a proper open subset which intersect $W$ and which is contained in the smooth locus of $X$ and let $V=\nu^{-1}(U\cap W)$.  Then the composition of the natural maps
\[
\mathcal N_{\cal F}^*|_{V}\to \Omega^{1}_X|_{V}\to \Omega^{1}_{V}
\]
induces a co-rank one foliation $\cal G_V$ on $V$ which extends naturally to a foliation $\cal G$ on $W^{\nu}$, called the {\bf restricted foliation} of $\cal F$ on $W^{\nu}$.

\medskip


If $X$ is a smooth variety and $\cal F$ is a co-rank one foliation on $X$ then
we say that $\cal F$ is {\bf strongly non-dicritical}  if for any sequence of blow ups 
 \[
 X_n \xrightarrow{p_n} \cdots X_1 \xrightarrow{p_1} X\]
 in smooth centres tangent to $\cal F_i$
or smooth centres contained in $\sing \cal F_i$, where $\cal F_i$ is the transformed foliation on $X_i$,
we have that the exceptional locus of $X_n \rightarrow X$ is $\cal F_n$-invariant (cf. \cite{CM92}).
We remark that we allow these maps to be blow ups along analytic subvarieties.

\begin{remark}\label{r_simple=ndc}
Simple singularities are strongly non-dicritical.  Indeed, by \cite[Th\'eor\`eme 4]{CM92} non-dicriticality can be checked by blowing up
in permissible centres (cf. \cite[Definition 1]{Cano}).  However, as we saw in the proof of Lemma \ref{l_simplecanonical}
the exceptional divisor for such a blow up is always  invariant by the induced foliation.
\end{remark}

In \cite[Th\'eor\`eme 4]{CM92}, the following  characterisation of strong non-dicriticality is given.
A germ of a co-rank one foliation $\cal F$ on $0 \in \bb C^n$ is strongly non-dicritical if and only if
there does not exists a germ of a surface $0 \in Z\subset \mathbb C^n$ such that $Z$ is transverse to 
$\cal F$ and such that the restricted foliation to $Z$ admits
infinitely many invariant curves passing through $\sing \cal F$.

We now show that, in the case of smooth threefolds, all these notions of non-dicriticality coincide:

\begin{lemma}
\label{lem_equiv_dicrit}
Let $X$ be a smooth threefold and let $\cal F$ be a co-rank one foliation on $X$.

Then $\cal F$ is non-dicritical if and only if it is strongly non-dicritical.
\end{lemma}
\begin{proof}
Suppose that $\cal F$ is strongly non-dicritical.  Let $\pi\colon X' \rightarrow X$ be a proper birational morphism, 
let $q \in X$ be a closed point
and suppose that $E \subset \pi^{-1}(q)$ is a prime divisor.  
Suppose for sake of contradiction that $E$ is not $\pi^{-1}\cal F$-invariant.  
Let $H \subset X'$ be a general hypersurface
such that if $\cal G$ is the foliation restricted to $H$ then $E\cap H$ is transverse to $\cal G$.
Then, through a general point $P \in E \cap H$ there is a germ of a $\cal G$-invariant curve passing through
$P$, call it $\Sigma_P$.  It follows that if $Z = \pi(H)$ then $Z$ is transverse to $\cal F$ and the restricted foliation
admits infinitely many invariant curves passing through $q = \pi(E \cap H)$, namely, $\pi(\Sigma_P)$
as we let $P$ vary over points of $E\cap H$, a contradiction. 
In particular, $q\in \sing \cal F$. 
Thus, $\cal F$ is non-dicritical. 

\medskip

Now suppose that $\cal F$ is non-dicritical.
Let $\pi\colon X' \rightarrow X$ be a resolution of singularities so that
$(X', \exc \pi)$ is log smooth and $\cal F' \coloneqq \pi^{-1}\cal F$ has 
simple singularities adapted to $\exc \pi$ (cf. Definition \ref{defn_simple}). In particular, if $E$ is a component of $\exc \pi$ which is not
$\cal F'$-invariant then $\sing \cal F' \cap E$ has no one-dimensional components.
We may also take $\pi$ to be a sequence of blow ups centred either
in $\sing \cal F$ or centres tangent to the foliation.  
Since, by Remark \ref{r_simple=ndc}, simple singularities are strongly non-dicritical,
  it suffices to show that $\exc\pi$ is $\cal F'$-invariant.

If $E$ is a component of $\exc \pi$ such that $\pi(E)$ is zero-dimensional or $\pi(E)$ is not contained in $\sing \cal F$ 
then it follows immediately that $E$ is $\cal F'$-invariant.
So suppose that $\pi(E)$ is a curve contained in $\sing \cal F$.  Since $\cal F$ is non-dicritical it follows that the fibres
of $E \rightarrow \pi(E)$ are tangent to $\cal F'$.  
Suppose for sake of contradiction that $E$ is not $\cal F'$-invariant.

If $\cal G$ is the foliation restricted to $E$ then $\cal G$ is the foliation induced by $E \rightarrow \pi(E)$ and is therefore
smooth over the generic point of $\pi(E)$.
Since $\sing \cal F' \cap E$ has no one-dimensional components, it follows that $\cal F'$ is smooth in a neighbourhood
of a general fibre of $E \rightarrow \pi(E)$.
Next, we claim that if $E'$ is any other $\pi$-exceptional divisor dominating $\pi(E)$ then $E'$ is not $\cal F'$-invariant.
Indeed, suppose otherwise.  Without loss of generality we may assume that $E'\cap E \neq \emptyset$
and $E \cap E'$ dominates $\pi(E)$.  However $E \cap E'$ is $\cal G$-invariant since $E'$ is $\cal F'$-invariant and so $E\cap E'$
is contained in a fibre of $E \rightarrow \pi(E)$, a contradiction.

Shrinking about a neighborhood
of the general point of $\pi(E)$ we may therefore assume that every $\pi$-exceptional divisor is transverse to
$\cal F'$ and that $\cal F'$ is smooth.

Now, let $\omega$ be a $1$-form defining $\cal F$ in a neighborhood of a general point
of $\pi(E)$.  Since $\pi(E)$ is contained in $\sing \cal F$ it follows that $\pi^*\omega$ must vanish
along $\exc \pi$ and, in particular, it follows that 
that  $N^*_{\cal F'} = F$ where $F\geq 0$ is $\pi$-excpetional and $F \neq 0$.

On the other hand, it is easy to verify that $K_{\cal F'}$ and $K_{X'}$ are numerically equivalent over $X$.
Indeed, one can check (as in the proof of \cite[Lemma 3.11]{Spicer17})
that if $E'$ is any $\pi$-exceptional divisor dominating $\pi(E)$ and if $C$ is a general fibre
of $E' \rightarrow \pi(E)$ then $(K_{\cal F'}+E)\cdot = (K_{X'}+E')\cdot C = -2$ and so $K_{\cal F'}\cdot C = K_{X'}\cdot C$.
This, together with the equality $K_{X'} = K_{\cal F'}+N^*_{\cal F'}$ implies that $N^*_{\cal F'} = F$ is numerically trivial
over $X$. Since $F \neq 0$, the negativity lemma (e.g. see \cite[Lemma 3.39]{KM98})  gives us a contradiction. 
\end{proof}

\begin{remark}
The above proof shows that  if $X$ is a smooth threefold, $\cal F$ is a non-dicritical  co-rank one foliation on $X$, $\pi\colon X'\to X$ is a birational morphism
and $W$ is $\cal F$-invariant, then
$\pi^{-1}(W)$ 
is $\cal F'$-invariant where $\cal F'$ is the transformed foliation on $X'$.
\end{remark}

\begin{remark} Let  $X$ be a normal threefold and let $\cal F$ be a co-rank one foliation with non-dicritical singularities.
Let $\pi\colon X'\to X$ be a  birational morphism and let $\cal F'$ be the transformed foliation on $X'$. Assume that there 
exists a prime $\pi$-exceptional divisor $E$, which is $\cal F'$-invariant and  whose centre in $X$ is $W$. Then every
$\pi$-exceptional divisor $E$ whose centre in $X$ is $W$ is $\cal F'$-invariant. 
Indeed, the proof of this fact follows exactly as in the proof of Lemma \ref{lem_equiv_dicrit}.
%
%
\end{remark}

%
%
%

\begin{defn}
Given a normal germ $0 \in X$ (resp. the formal completion $0 \in \widehat{X}$ of a normal variety 
$X$ at the point $0\in X$) with a  co-rank one foliation $\cal F$
such that $0$ is a singular point for $\cal F$, 
we call an irreducible hypersurface germ (resp. a formal hypersurface) $0 \in S$ a {\bf separatrix (resp. formal separatrix)}
if it is $\cal F$-invariant.
\end{defn}

Let $0 \in X$ be a smooth germ and let $\cal F$ be a co-rank one foliation $\cal F$ on $X$ defined by a $1$-form $\omega$. 
Sometimes, in the literature, a formal separatrix is defined to be an irreducible and reduced formal
power series $f$ such that  $f$ divides $df\wedge \omega$.
We claim that, under this assumption, $\{f=0\}$ defines a formal hypersurface which is $\cal F$-invariant.  Indeed, let $v$ be a vector field
such that $v(\omega)= 0$.  On one hand, $v(df \wedge \omega)$ is necessarily divisble by $f$ since 
$df\wedge\omega$ is. On the other hand, we can compute $v(df\wedge \omega)= v(df) \omega -df v(\omega) = v(df)\omega$
which implies that $v(df)$ is divisible by $f$, i.e., the ideal $(f)$ is invariant by $v$, as required.
Conversely, if $\{f = 0\}$ is $\cal F$-invariant then $df \wedge \omega$ is divisible by $f$.  Indeed, after replacing $(X,\{f=0\})$ by its log resolution whose existence is guaranteed by \cite[Theorem 1.1.9 and Theorem 1.1.13]{temkin18}, 
we may assume that  the pair $(X, \{f = 0\})$ is log smooth. Thus, we may assume that in some formal coordinates $x_1, ..., x_n$
we may write $f = x_1$.  Now write $\omega = \sum a_idx_i$ for some functions $a_1,\dots,a_n$ and suppose for sake of contradiction that there exists
$j \neq 1$ such that $a_j \notin (x_1)$.  In this case 
$\partial \coloneqq a_j\frac{\partial}{\partial x_1} - a_1\frac{\partial}{\partial x_j}$ is a local vector field in $\cal F$.
However, $\partial(x_1)  =a_j \notin (x_1)$ and so $\{x_1 =0\}$ is not invariant, a contradiction.

\begin{example}
Let $\cal F$ be a co-rank one foliation on a smooth variety $X$ with a simple singularity at $p\in X$ of dimension type $r$.  Let 
$x_1,\dots,x_n$ be formal coordinates as in Definition \ref{defn_simple}. Then $\{x_i=0\}$ is a formal separatrix for each $i=1,\dots,r$
and moreover these are the only formal separatrices at $p$, see \cite[Appendix: About simple singularities]{Cano}.

\end{example}

\medskip

Note that away from the singular locus of $\cal F$
a separatrix is in fact a leaf.  Furthermore being non-dicritical 
implies that there are only finitely
many separatrices through a singular point. The converse of this statement is false.

\begin{lemma}\label{l_simplegeneric}
Let $X$ be a smooth variety of dimension $n$ and let
$\cal F$ be a co-rank one foliation on $X$ with simple singularities.
Suppose that $S \subset \sing \cal F$ is a subvariety such that  $\dim S \geq 1$.
Let $H \subset X$ be a general element of a base point free linear system and let $\cal G$ be the restricted foliation.

Then $\cal G$ has simple singularities at a general point of $S \cap H$.
\end{lemma}
\begin{proof}
Remark \ref{r_simple=ndc} implies that $\cal F$ admits non-dicritical singularities. 
Let $P \in S$ be a general closed point and let $r \leq n-1$ be the dimension type of 
$\cal F$ at $P$. Note that this condition is generic in $\sing \cal F$.

Let $D_1, ..., D_{r}$ be the separatrices
of $\cal F$ at $P$, including the formal ones.  
Let $H$ be a hyperplane passing through $P$ such that $(\widehat{X}, D_1+...+D_{r}+H)$
is a 
normal crossings pair where $\widehat{X}$ is the formal completion of $X$ at $P$.  
If $\omega$ is a $1$-form defining $\cal F$
it is easy to check that $\omega$ restricted to $H$ is still a $1$-form which is of one of the types listed in 
Definition \ref{defn_simple}.
Thus, the restricted foliation will have simple singularities near $P$.
It follows from \cite[Proposition 14]{Cano} that having pre-simple singularities
is an open condition in $\sing \cal F$, and the proof there works just as well to imply
that having simple singularities is an open condition in $\sing \cal F$.
\end{proof}

Even for simple foliation singularities it is possible that there
are separatrices which do not converge.  However, as the following
definition/result shows there is always at least one convergent
separatrix along a simple foliation singularity of codimension two.

\begin{lemma}
 \label{lem_strong_sep_exist} 
For a simple singularity of type (1), all separatrices are convergent.

For a simple singularity of type (2), around a general point
of a codimension two component of the the singular locus we can write 
$\omega = pydx+qxdy+ x\psi(x^py^q)\lambda dy$. The hypersurface $\{x = 0\}$ is a convergent
separatrix,  called the {\bf strong separatrix}.
\end{lemma}
\begin{proof}
This is proven in \cite[Part II \S 5]{CC92}.
\end{proof} 

\subsection{Foliation with $K_{\cal F}$ not pseudo-effective}

We make note of the following easy consequence of \cite{CampanaPaun19}

\begin{lemma}
\label{lem_KF_not_psef_uniruled}
Let $X$ be a normal variety and let $\cal F$ be a foliation on $X$
such that $K_{\cal F}$ is $\bb Q$-Cartier and $K_{\cal F}$ is not pseudo-effective.  
Then $\cal F$ is uniruled, i.e.,
there exists a family of rational curves covering $X$ and tangent to $\cal F$.
\end{lemma}
\begin{proof}
Let $\pi\colon X' \rightarrow X$ be a resolution of singularities and let $\cal F'$ be the pulled back foliation.
Then observe that $K_{\cal F'}$ is not pseudo-effective and so we may freely assume that $X$ is smooth.

Let $\alpha$ be a movable class so that $K_{\cal F}\cdot \alpha <0$.  Let $\cal E \subset \cal F$ be a maximal destabilizing subsheaf
of $\cal F$ with respect to $\alpha$.  It follows that $\cal E$ defines a foliation with $\mu_{\alpha, min}(\cal E)>0$
and so by \cite[Theorem 1.1]{CampanaPaun19} $\cal E$ is a foliation with rationally connected leaves, from which our claim follows.
\end{proof}

\subsection{Steps of the Minimal Model Program}
\label{s_ample_model}

We recall some of the main definitions commonly used in the Minimal Model Program. 
Let $X$ be a normal  projective variety. We denote by $N_1(X)$ the $\mathbb R$-vector space spanned by $1$-cycles on $X$ modulo numerical equivalence (e.g. see \cite[Definition 1.16]{KM98}). We denote by $NE(X)\subset N_1(X)$ the subset of effective $1$-cycles $[\sum_{i=1}^k a_i C_i]$ where $a_1,\dots,a_k$ are positive real numbers and $C_1,\dots,C_k$ are curves in $X$, and we denote by $\overline {NE(X)}$ its closure (e.g. see \cite[Definition 1.17]{KM98}).
A {\bf ray} is a one-dimensional subcone $R$ of $\overline {NE(X)}$ and it is called {\bf extremal} if for any $u,v\in \overline {NE(X)}$ such that $u+v\in R$, we have that $u,v\in R$. 
If $D$ is a $\mathbb Q$-Cartier $\mathbb Q$-divisor on $X$ then the extremal ray $R$ is said to be $D$-{\bf negative} if $D\cdot C<0$ for any curve $C$ such that $[C]\in R$. 
Given an extremal ray $R \subset \overline{NE}(X)$, we define the {\bf locus} of $R$, denoted 
$\loc R$, to be the set of all those points $x\in X$ such that there exists a curve $C$ with $x \in C$ 
and $[C] \in R$.
A projective birational morphism $f\colon X \rightarrow Y$ between normal projective varieties is said to be  a
 {\bf flipping contraction}, or {\bf small contraction}, if its exceptional locus has codimension at least two and it is called a {\bf divisorial contraction} if its exceptional locus is a divisor. Moreover, the birational morphism $f\colon X\to Y$ is said to be an
{\bf extremal contraction} if the relative Picard number $\rho(X/Y)\coloneqq \rho(X)-\rho(Y)$ is equal to one. Given an extremal ray $R \subset \overline{NE}(X)$, an extremal contraction $f\colon X\to Y$ is said to be {\bf associated} to $R$ if the locus of $R$ coincides with the exceptional locus of $f$. 
%

%
%

Let $D$ be a $\bb Q$-Cartier $\mathbb Q$-divisor on a normal projective variety $X$, let $R$ be a $D$-negative extremal ray and let $f\colon X\to Y$  be a flipping contraction associated to $R$. Note, in particular, that $-D$ is $f$-ample. Then 
the {\bf $D$-flip} is a birational map $\phi\colon X \dashrightarrow X^+$ such that
\begin{enumerate}
\item $\phi$ is an isomorphism in codimension one,

\item there exists an extremal small contraction $f^+\colon X^+ \rightarrow Y$, and


\item $D^+\coloneqq \phi_*D$ is $f^+$-ample.
\end{enumerate}

Note that the $D$-flip, if it exists, is unique.

\medskip 

We refer to \cite[Definition 3.6.4]{BCHM06}
for the definition of ample models and we observe that if
$\phi\colon X \dashrightarrow X^+$ is a $D$-flip with flipping contraction $f\colon X\to Y$
then $X^+$ is the ample model of $D$ over $Y$.

\medskip 

Let $\phi\colon X\dashrightarrow Y$ be a  birational contraction between normal projective varieties and let $D$ be a $\mathbb Q$-Cartier $\mathbb Q$-divisor on $X$. Then $\phi$ is said to be a {\bf step of a $D$-MMP} if it is either a $D$-flip or a divisorial contraction associated to a $D$-negative extremal ray. More in general, $\phi$ is said to be a {\bf sequence of steps of a $D$-MMP} if we can decompose $\phi$ as 
\[X=X_0\dashrightarrow X_1\dots\dashrightarrow X_\ell =Y
\]
so that  the birational contraction $X_i\dashrightarrow X_{i+1}$ is a step of a $D_i$-MMP where $D_i$ is the strict transform of $D$ on $X_i$. Note that if $X$ is $\mathbb Q$-factorial, then $Y$ is also $\mathbb Q$-factorial.

If we replace the projective variety $X$ by a projective morphism $\pi\colon X \to U$ between normal varieties, then all the definitions above admit a relative version, by replacing each variety  by one  admitting  a projective morphism to $U$, and each birational map by a birational map over $U$, in a similar fashion as in the classical MMP (e.g. see \cite[Section 3.6]{KM98}).

\subsection{A result from the classical MMP}

We will need to make use of some techniques from the classical MMP.

\begin{defn}
\label{defn_small_q_fact}
Let $X$ be a normal variety.  We say that
a birational morphism $f\colon Y \rightarrow X$ is a {\bf small $\bb Q$-factorialisation} if the following holds:
\begin{enumerate}
\item $f$ is an isomorphism in codimension one,

\item $f$ is a projective morphism, and

\item $Y$ is $\bb Q$-factorial.
\end{enumerate}

\end{defn}

\begin{theorem}
\cite[Corollary 1.4.3]{BCHM06}
\label{thm_existence_q_fact_classical}
Let $X$ be a normal algebraic variety of dimension three and let $D \geq 0$ be such that $(X, D)$ is klt.

Then there exists a small $\bb Q$-factorialisation for $X$.
Moreover, if we write $K_Y+D_Y = f^*(K_X+D)$ then $(Y, D_Y)$ is klt.
\end{theorem}

\begin{remark}\label{r_qfactorialisation} 
Since flips are known to exist over complex analytic varieties (cf. \cite[Main Theorem]{Shokurov93}), the same proof as in \cite{BCHM06} implies the existence of a small $\bb Q$-factorialisation for a klt pair $(X,D)$, where $X$ is a complex analytic space of dimension three. 
\end{remark}


\section{F-dlt foliated pairs and basic adjunction type results}
The goal of this section is to define a new category of foliated log pair singularities, namely F-dlt pairs. These are  analogous of dlt log pairs in the classical MMP and they seem to be the most suitable singularities to run a foliated MMP.  In particular, we prove several properties  satisfied by these pairs, which we  use later on in the paper. 

\subsection{Foliated log smooth pairs}

\begin{defn}
\label{defnlogsmooth}
Given a normal variety $X$, a co-rank one foliation $\cal F$ and a foliated pair $(\cal F, \Delta)$ we say that $(\cal F, \Delta)$
is {\bf foliated log smooth} provided the following hold:
\begin{enumerate}
\item $(X, \Delta)$ is log smooth,

\item $\cal F$ has simple singularities, and 

\item if $S$ is the support of the non $\cal F$-invariant components of $\Delta$,
 $p \in S$ is a closed point and  $\Sigma_1, ..., \Sigma_k$ are the
 (possibly formal) $\cal F$-invariant divisors  passing through $p$, 
then $S \cup \Sigma_1 \cup... \cup \Sigma_k$ is a normal
crossings divisor at $p$. 
\end{enumerate}

Given a normal variety $X$, a co-rank one foliation $\cal F$ and a foliated pair $(\cal F, \Delta)$, a 
{\bf foliated log resolution}, or in short {\bf log resolution},  is a proper birational morphism 
$\pi\colon Y\to X$ so that $\exc \pi$ is a divisor and $(\mathcal G, \pi_*^{-1}\Delta+\sum E)$ is foliated log smooth where 
$\mathcal G$ is the pulled back foliation on $Y$ and 
the sum runs over all the $\pi$-exceptional divisors.

\end{defn}

\begin{remark}\label{r_log-smooth}

\begin{itemize}
\item If $X$ is a surface, then the existence of a foliated log resolution follows from a result of Seidenberg \cite{Seidenberg68}. If $X$ is a threefold, then such a resolution exists by \cite{Cano}.
\item Items (2) and (3) in Definition \ref{defnlogsmooth} imply that 
each component of $S$ is generically transverse to the foliation, no strata
of $S$ is tangent to the foliation and no strata of $\sing \cal F$ 
is contained in $S$.
%
%
%

\item If $\cal F$ is log smooth and if $D$ is a $\cal F$-invariant divisor then
it is not necessarily the case that $D$ is smooth, although it will have at worst
normal crossings singularities.
\end{itemize}
\end{remark}

\begin{lemma}
\label{l_logsmoothlogcanonical}
Let $(\cal F,\Delta)$ be a foliated log smooth pair on a variety $X$, where  $\Delta = \sum a_iD_i$ is a $\mathbb Q$-divisor such that  $0\leq a_i \leq 1$ and  $D_i$ is not $\cal F$-invariant for every $i$.

Then $(\cal F, \Delta)$ is log canonical.
\end{lemma}
\begin{proof} 
By Lemma \ref{l_simplecanonical}, since $\cal F$ has simple singularities, it follows that $\cal F$ is canonical.

Now let $\pi\colon Y \rightarrow X$ be a blow up of subvariety $Z \subset \supp \Delta$
where $Z$ has codimension $k$. Let $E$ be the exceptional divisor.  We compute the discrepancy of this blow up as follows:
\begin{enumerate}
\item If $Z$ is transverse to the foliation then the discrepancy is 
\[
(k-1)-\sum_{\{i \mid Z \subset D_i\}}a_i \geq -1 = -\epsilon(E)
\]
where the inequality holds since $k \geq \#\{i \mid Z \subset D_i\}$ by Item (1) in Definition \ref{defnlogsmooth}.

\item If $Z$ is tangent to the foliation but not contained in
$\sing \cal F$ then the discrepancy is
\[
(k-1)-\sum_{\{i \mid Z \subset D_i\}}a_i  \geq 0 = -\epsilon(E)
\]
where the inequality holds since $k \geq \#\{i \mid Z \subset D_i\}+1$ by Item (3) in Definition \ref{defnlogsmooth}.

\item If $Z \subset \sing \cal F$ then let $m$ be the codimension of the minimal strata of $\sing \cal F$ 
containing $Z$. 
The discrepancy of the blow up is 
\[
(k-1)-(m-1) - \sum_{\{i \mid Z \subset D_i\}}a_i \geq 0 = -\epsilon(E)
\]
where the inequality holds since $k \geq m+\#\{i \mid Z \subset D_i\}$ by Item (3) in Definition \ref{defnlogsmooth}.
\end{enumerate}

As in the proof of Lemma \ref{l_simplecanonical}, it follows that if $\cal G$ is the transformed foliation on $Y$ and $\Delta'$ is the strict transform of $\Delta$ in $Y$, then $(\cal G,\Delta'+E)$ is foliated log smooth. Thus, the result then follows by  induction.
\end{proof}

Note that in contrast to the classical situation, if $(\cal F, \Delta)$ is a foliated log smooth pair
then $(\cal F, \Delta)$ may have infinitely many lc centres:

\begin{example}
Let $(\cal F, D_1+D_2)$ be a foliated log smooth pair on a  threefold $X$, for some prime divisors $D_1$ and $D_2$  which are not
$\cal F$-invariant.  Suppose that  $Z = D_1 \cap D_2$ is non-empty,   disjoint from $\sing \cal F$ and that $Z$ is  transverse to $\cal F$. Then $Z$ is an lc centre of $(\cal F, D_1+D_2)$.  Moreover, if $p \in Z$ and if
$\pi\colon Y \rightarrow X$ is the blow up at $p$ with exceptional divisor $E$ then the discrepancy with respect
to $E$ is $0 = \epsilon(E)$ and so $p$ is an lc centre of $(\cal F, D_1+D_2)$. In particular, $(\mathcal F,D_1+D_2)$ admits infinitely many lc centres. 

Note also that if $\mathcal F$ is a foliation on a smooth projective variety $X$ which is induced by a fibration onto a curve then any smooth vertical fibre is an lc centre. 
\end{example}

\subsection{Extending separatrices}

We provide a general result on the existence of formal separatrices
which is a slight generalisation of the results in \cite[\S IV]{CC92}.

\begin{lemma}\label{l_formalseparatrix}
Let $X$ be a normal quasi-projective threefold.  Let $\cal F$ be a co-rank one foliation on $X$
with non-dicritical singularities. Let $V \subset X$ be a subvariety tangent to $\cal F$,
let $q \in V$ be any point
and let $\widehat{S}_q$ be a, possibly formal, $\cal F$-invariant divisor containing $q$. Let  $\widehat X$ be the formal completion
of $X$ along $V$.

Then there exists an $\widehat{\cal F}$-invariant formal subscheme $\widehat{S}$ on $\widehat X$
which contains $\widehat{S}_q$ near $q$. 

Moreover, if $\widehat{S}_q$ is in fact convergent, then we may take $\widehat{S}$ to be convergent.
\end{lemma}

\begin{proof}
Let $\pi\colon W \rightarrow X$ be a high enough foliated log resolution
so that $\pi^{-1}(V) = E$ is a divisor. By definition, we see that $E$ is $\pi^{-1}
\cal F$-invariant.  
Let $\widehat{W}$
be the completion of $W$ along $E$ and let
$\hat{\pi}\colon \widehat{W} \rightarrow \widehat{X}$ be the induced morphism.

Let $\widehat{X}_{/q}$ denote the formal completion of $X$ along $q$ and let $\widehat{W}_{/\pi^{-1}(q)}$ denote
the formal completion of $W$ along $\pi^{-1}(q)$.  Note that we have morphisms
$\widehat{X}_{/q} \rightarrow \widehat{X}$ and $\widehat{W}_{/\pi^{-1}(q)} \rightarrow \widehat{W}$
which commute with the induced morphism $\widehat{W}_{/\pi^{-1}(q)} \rightarrow \widehat{X}_{/q}$ and $\widehat W\to \widehat X$.
Since $\widehat{S}_q$ is a formal subscheme of $\widehat{X}_{/q}$ we may take
its strict transform on $\widehat{W}_{/\pi^{-1}(q)}$, call it $\widehat{S}_q'$.
Recall that we can construct the strict transform as follows: let $\widehat{X}_{/q} = \text{Spf } A$,
$\widehat{S}_q = \text{Spf } B$ and let $\pi$ be given by the blow up along an ideal $I$.
Let $\tilde{W}$ be the blow up of $\Spec A$ along $I \otimes A$.  If $\tilde{S} = \Spec B$ 
we may define the strict transform $\tilde{S}'$ as the blow up of $\tilde{S}$ along the ideal
$I \otimes_A B$, \cite[Corollary II.7.15]{Hartshorne77}.  We may take $\widehat{S}'_q$ to be the formal completion of $\tilde{S}'$
along $\tilde{S}'\cap \pi^{-1}(q)$.

The arguments in \cite[\S IV]{CC92}  
and their slight extension in \cite[Lemma 5.3]{Spicer17} show that if $\widehat{S}'_q$ is convergent then we can extend
$\widehat{S}'_{q}$ to a $\pi^{-1}\cal F$-invariant formal subscheme $\widehat{S}'$ of $\widehat{W}$.
In fact, the arguments given in \cite{Spicer17} work even if $\widehat{S}'_q$ is not convergent
as in the proof of \cite[Theorem IV.2.1]{CC92}.  

For the reader's convenience we briefly indicate  some of the important ideas of the proof of \cite[Theorem IV.2.1]{CC92}.
Let $\omega$ be a 1-form defining  a simple singularity on $0 \in U \subset \bb C^3$ with coordinates $(x, y, z)$.  
Suppose that the dimension type
of $\omega$ is $3$ (the case where the dimension type is $2$ can be handled in a similar manner).  
By the work of \cite[\S 2]{CC92} we know
that two of the separatrices at $0$ are convergent.  So, after performing a holomorphic change of coordinates
we may assume that they are given by $\{x=0\}$ and $\{y=0\}$.  It follows that the formal separatrix at $0$
may be defined by an equation $f = z+\phi(x, y)$ where $\phi(x, y) \in (xy)\bb C[[x, y]]$.   By \cite[Proposition II.5.4]{CC92}
it follows that there exists a bidisc $V\subset \bb C^2$ such that in fact $\phi(x, y) = \sum \phi_i(x, y)(xy)^i$ where 
$\phi_i(x, y) \in \cal O_{\bb C^2}(V)$.  In particular, up to shrinking $U$, if $\widehat{U}$ denotes the formal completion of
$U$ along $\{xy=0\}$ we see that in fact $z +\phi(x, y) \in H^0(\widehat{U}, \cal O_{\widehat{U}})$.  In particular, 
we have extended the formal separatrix, a priori only defined on the formal completion of $U$ at $0$, to the
formal completion of $U$ along the union of the convergent separatrices.  Using this local extension result we see
that the arguments given for extending convergent separatrices work for extending formal separatrices.

\medskip

Let $\cal I_{\widehat{S}'} \subset \cal O_{\widehat{W}}$ be the ideal sheaf corresponding to $\widehat{S}'$.
By the proper mapping theorem for formal schemes \cite[Th\'eor\`eme 3.4.2]{EGAIII},
$\hat{\pi}_* I_{\widehat{S}'}$ is a coherent sheaf, and since
$\hat{\pi}_*\cal O_{\widehat{W}} = \cal O_{\widehat{X}}$ we see that it is in fact
an ideal sheaf corresponding to a formal subscheme $\widehat{S} \subset \widehat{X}$.

Since being an invariant divisor can be checked locally, it suffices
to check $\widehat{S}$ is a formal invariant divisor in the case where $X$ is affine.

If $X$ is affine, then let $\widetilde{X} = \Spec {\cal O_{\widehat{X}}}$
and let $\widetilde{W} = W \times_X \widetilde{X}$ and let $\tilde{\pi}\colon \widetilde{W} \rightarrow \widetilde{X}$ be induced map.  
By the Grothendieck existence theorem, 
$\hat{S}'$ corresponds to a closed subscheme of $\widetilde{W}$ denoted $\widetilde{S}'$
and $\hat{S}$ correspond to a closed subscheme of $\widetilde{X}$ denoted $\widetilde{S}$.
The above construction gives us $\tilde{\pi}_*\widetilde{S}' = \widetilde{S}$.  Observe 
that $\widetilde{S}$ is a divisor on $\widetilde{X}$. 
Let $\widetilde{U} = \widetilde{X} \setminus \tilde{\pi}(\exc \tilde{\pi})$ and note that $\widetilde{U}$ is a Zariski
open subset and $\widetilde{S}\cap\widetilde{U} \neq \emptyset$.
It follows immediately that $\widetilde{S}\cap \widetilde{U}$ is $\cal F\vert_{\widetilde{U}}$-invariant
(since $\tilde{\pi}$ is an isomorphism above $\widetilde{U}$).  Since $\widetilde{S}$ admits an invariant Zariski dense
subset we see that in fact $\widetilde{S}$ is invariant.
The theorem on formal functions tells us that the completion of $\widetilde{S}$ along $V$ is exactly $\widehat{S}$,
and our result follows.
\end{proof}

\subsection{F-dlt foliated pairs}

\begin{defn}\label{d_fdlt}
Let $X$ be a normal variety and let $\cal F$ be a co-rank one foliation on $X$.
Suppose that $K_{\cal F}+\Delta$ is $\bb Q$-Cartier.

We say $(\cal F, \Delta)$ is {\bf foliated divisorial log terminal (F-dlt)}
if 
\begin{enumerate}
\item each irreducible component of $\Delta$ is generically transverse to $\mathcal F$ and has coefficient at most one, and
\item there exists a foliated log resolution $\pi\colon Y\to X$ of $(\cal F, \Delta)$ which only
extracts divisors $E$ of discrepancy $>-\epsilon(E)$.
\end{enumerate}
\end{defn}

\begin{remark}\label{r_fdlt=lc}
As we show in Remark \ref{r_canonical-vs-fdlt} below, canonical singularities are not in general F-dlt. On the other hand,  if $(\cal F,\Delta)$ is a F-dlt pair, then it is log canonical. Indeed, by assumption, there exists a foliated log resolution $\pi\colon Y\to X$ which only extract divisors $E$ of discrepancy $>-\epsilon(E)$. Thus, if $\cal G$ is the transformed foliation on $Y$ and 
$\Gamma$ is the strict transform of $\Delta$ in $Y$, then we may write
\[
K_{\cal G}+\Gamma+F = \pi^*(K_{\cal F}+\Delta)+G
\]
where $F,G\ge 0$ are $\pi$-exceptional $\mathbb Q$-divisors without any common component. Note, in particular, that no component of $\Gamma+F$ is $\cal G$-invariant and, therefore,  
Lemma \ref{l_logsmoothlogcanonical} implies that $(\cal G,\Gamma+F)$ is log canonical. Thus, for any valuation $S$, we have
\[
a(S,\cal F,\Delta)\ge a(S,\cal G,\Gamma+F)\ge -\epsilon(S)
\]
and our claim follows. 
\end{remark}

\begin{lemma}\label{l_fdlt-logsmooth}
Let $X$ be a normal variety and let $\cal F$ be a 
co-rank one foliation on $X$.
Suppose that $(\cal F, \Delta)$ is a F-dlt  pair on $X$ and 
let $\pi\colon Y\to X$ be a foliated log resolution such that $a(E,\cal F,\Delta)>-\epsilon(E)$ for any $\pi$-exceptional divisor $E$.

Then $\pi$ is  an isomorphism at the general point of $\pi^{-1}(W)$ for any lc centre $W\subset X$. In particular,  $(\cal F,\Delta)$ is foliated log smooth at the generic point of $W$. 
\end{lemma}

\begin{proof} 
Suppose by contradiction that $\pi$ is not an isomorphism at the general point of $\pi^{-1}(W)$. 
We may write  
\[
K_{\mathcal G}+\Gamma = \pi^*(K_\cal F+\Delta) + F
\]
where $\Gamma, F\ge 0$ are $\bb Q$-divisors without any common component and $\cal G$ is the pulled back foliation on $Y$. 
Note that $(\mathcal G,\Gamma)$ is log smooth.  

By assumption, there exists a valuation $T$, whose centre in $X$ is $W$ and  such that $a(T,\mathcal F, \Delta)=-\epsilon(T)$. 
Since $\pi$ is not an isomorphism at the general point of $\pi^{-1}(W)$ and since $\exc \pi$ is a divisor, there exists a $\pi$-exceptional  prime divisor $E$ which contains the centre of $T$ in $Y$. 

If $E$ is $\mathcal G$-invariant, then $E$ is contained in the support of $F$ and we have 

\[
a(T,\cal G,\Gamma) < a(T,\cal F,\Delta)=-\epsilon(T),
\]
which contradicts Lemma \ref{l_logsmoothlogcanonical}. 

Similarly, if $E$ is not $\mathcal G$-invariant, then $E$ is not contained in the support of $\lfloor \Gamma \rfloor$ and there exists
$\delta >0$ such that if $\Gamma'\coloneqq \Gamma+\delta E$ then, $(\mathcal G,\Gamma')$ is log smooth,  the coefficients of $\Gamma'$ are not greater than one and

\[
a(T,\cal G,\Gamma') < a(T,\cal G,\Gamma) \le a(T,\cal F,\Delta)=-\epsilon(T),
\]
which contradicts again Lemma \ref{l_logsmoothlogcanonical}.
\end{proof}

\begin{proposition}
\label{p_finmanylccenters}
Let $X$ be a  normal variety and let $\cal F$ be a  co-rank one foliation on $X$.
Let $(\cal F, \Delta)$ be a F-dlt pair on $X$. 

Then $(\cal F, \Delta)$
has only finitely many lc centres of codimension at least two, which are not contained in the support of $\lfloor \Delta\rfloor$.\end{proposition}
\begin{proof}
By Lemma \ref{l_fdlt-logsmooth}, we have that $(\cal F, \Delta)$ is foliated log smooth at the generic point of every lc centre.
Let $Z \subset X$ be a subvariety of codimension at least two which is not contained in the support of $\lfloor \Delta\rfloor$ 
and such that $(\cal F,\Delta)$ is foliated log smooth at the generic point of $Z$. Let $\pi\colon Y \rightarrow X$ be the blow up at 
$Z$ with exceptional divisor $E$ and 
suppose that $Z \subset \supp \Delta$. Computing as in Lemma \ref{l_logsmoothlogcanonical}, we
see that the discrepancy of this blow up is $>-\epsilon(E)$.  Computing inductively we see that every
divisor dominating $Z$ has discrepancy $>-\epsilon(E)$.

Thus, every lc centre of $(\cal F, \Delta)$ not contained in $\supp\lfloor \Delta \rfloor$ must also be an lc centre of $(\cal F, 0)$.  
Keeping in mind that $\cal F$ has simple singularities at the general point of $Z$,
a straightforward computation
shows that the lc centres of $(\cal F, 0)$ are  strata of $\sing \cal F$ and therefore there are only finitely many such centres.
\end{proof}

\begin{remark}\label{r_finmanylccenters}
As in the proof of  Proposition \ref{p_finmanylccenters}, it follows that if $\cal F$ is a  co-rank one foliation on a normal variety 
$X$ and $(\cal F,\Delta)$ is a $F$-dlt pair, then there are only finitely many lc centres which are transverse to $\cal F$. Indeed, as in the proof of Lemma \ref{l_logsmoothlogcanonical}, these centres are  strata of $\lfloor \Delta\rfloor$.
\end{remark}

\begin{lemma}\label{l_mmp-fdlt}
Let $X$ be a normal threefold and let $\cal F$ be a co-rank one foliation
on $X$. Suppose that $(\cal F, \Delta)$ is a  F-dlt pair on $X$ and that $\phi\colon X\dashrightarrow X'$ is a sequence of steps of a $(K_{\mathcal F}+\Delta)$-MMP. Let $(\mathcal F',\Delta')$ be the transformed foliated pair on $X'$.

Then $(\mathcal F',\Delta')$ is also F-dlt. 
\end{lemma}

\begin{proof} We may assume that $\phi\colon X\dashrightarrow X'$ is either a $(K_{\cal F} + \Delta)$-flip or a divisorial contraction. 
We denote by $\Sigma$ the flipped locus if $\phi$ is a flip and  $\Sigma\coloneqq \phi(\exc \phi)$ if $\phi$ is a divisorial contraction.

By assumption, there exists a foliated log resolution $\pi\colon Y\to X$ of $(\cal F,\Delta)$ which only extracts divisors $E$ of discrepancy $>-\epsilon(E)$. It is enough to show that $(\cal F',\Delta')$ also admits such a foliated log resolution. 
Let $\overline Y\subset  Y\times X'$ be the closure of the graph of $\phi\circ \pi$ and let 
$p\colon \overline Y\to Y$ 
be the induced  proper birational morphism.  Let $\cal G$ be the pulled back foliation on $\overline Y$ and let $\overline \Delta$ be the strict transform of $\Delta$ in $\overline Y$. Let $f=\pi\circ p\colon \overline Y\to X$ be the induced morphism and let $F\coloneqq\sum F_i$ where the sum runs over all the $f$-exceptional divisors. Let $g\colon Y'\to \overline Y$ be a foliated log resolution of $(\cal  G, \overline \Delta+F)$. We may assume that $g$ is an isomorphism in the locus where $(\cal  G, \overline \Delta+F)$ is log smooth. 

Let $\pi'\colon Y'\to X'$ be the induced morphism, let $E'$ be a prime $\pi'$-exceptional divisor and let $W$ be the centre of $E'$ in $X'$.
We claim that $a(E',\mathcal F',\Delta')>-\epsilon(E')$. 
 If $W$ is contained in $\Sigma$, then 
Lemma \ref{l_negativity} and Remark \ref{r_fdlt=lc} imply that
\[-\epsilon(E')\le a(E',\mathcal F,\Delta)<a(E',\mathcal F',\Delta'),
\]
as claimed. 
If $W$ is not contained in $\Sigma$ then, by construction, the morphism $Y'\to Y$ is an isomorphism at the general point of $E'$ and the strict transform of $E'$ in $Y$ is $\pi$-exceptional.
Thus, 
\[-\epsilon(E')<a(E',\mathcal F,\Delta)=a(E',\mathcal F',\Delta'),
\]
and, again,  our claim follows. 
\end{proof}

\begin{lemma}
\label{basicpropertyfdlt}
Let $X$ be a normal threefold and let $\cal F$ be a co-rank one foliation
on $X$ with non-dicritical singularities.
Let $C \subset X$ be a curve tangent to $\mathcal F$ and suppose that $(\cal F, \Delta)$ is a F-dlt pair on $X$.

Then
\begin{enumerate}
\item $(\cal F, \Delta)$ is canonical at the generic point of $C$.
\item If in addition $C \subset \sing X$ then $(\cal F, \Delta)$ is terminal at the generic point of $C$.
\end{enumerate}
\end{lemma}
\begin{proof}
Item (1) follows from the observation that every divisor $E$ dominating $C$ on some log
resolution must be foliation invariant.

If $(\cal F,\Delta)$ is not terminal along $C$ then $C$ is an lc centre. Thus,  Lemma \ref{l_fdlt-logsmooth} implies that $X$ is smooth at the generic
point of $C$, i.e. $C$ is not contained in $\sing X$, and (2) follows.
\end{proof}

\begin{remark}
In fact, in Case (1) above there is an open set $U\subset X$ intersecting $C$ on which $(\cal F, \Delta)$ has canonical singularities. Indeed, let $\pi\colon X' \rightarrow X$ be a foliated log resolution and let $E_1, ..., E_k$ be the $\pi$-exceptional divisors of discrepancy $<0$ with respect to $(\cal F,\Delta)$.  Let $W = \pi(\sum E_i)$.  By assumption $C$ is not contained in $W$
and observe that $(\cal F, \Delta)$ has canonical singularities on $U\coloneqq X\setminus W$.

\end{remark}

\begin{lemma}\label{l_fdlt-terminal}
Let $X$ be a normal quasi-projective threefold and let $\cal F$ be a co-rank one foliation with non-dicritical singularities
on $X$.
Suppose that $(\cal F, \Delta)$ is F-dlt
and let $C \subset X$ be a curve tangent to $\cal F$.

Let $\widehat{X}$ denote the formal completion of $X$ along $C$.

\begin{enumerate}
\item If $(\cal F, \Delta)$ is terminal at the generic point of $C$ there exists a single $\cal F$-invariant divisor
$S\subset \widehat{X}$, and we may take $S$ to be convergent.

\item If $(\cal F, \Delta)$ is not terminal at the generic point of $C$ then
$X$ is smooth at the generic point of $C$ and at a general point $P \in C$ there are
2 (formal) separatrices at $P$ containing $C$, each of which may be extended to a divisor $S \subset \widehat{X}$.
Moreover, at least one of these separatrices may be extended to a convergent one.
\end{enumerate}
\end{lemma}
\begin{proof}
Note that Lemma \ref{basicpropertyfdlt} implies that $(\mathcal F, \Delta)$ is canonical at the generic point of $C$.
Suppose that $(\cal F, \Delta)$ is terminal at the generic point of $C$.  We first show the existence of a unique separatrix containing $C$
at a general point $P \in C$, and so we are free to shrink to an analytic neighborhood of $P$. 
Let $H$ be a germ of a general hypersurface at $P$ and let $\cal F_H$ be the restricted foliation. 
We claim that  $K_{\cal F_H}$ is $\bb Q$-Cartier and $\cal F_H$ has terminal singularities (cf. \cite[Lemma 8.7]{Spicer17}). 
Indeed, let $\pi\colon X' \rightarrow X$ be a log resolution of $\cal F$ and set $\cal F' \coloneqq \pi^{-1}\cal F$ and $\Delta' = \pi_*^{-1}\Delta$.
Perhaps shrinking about $P$ we may assume that $\exc \pi$ is $\cal F'$-invariant and the following hold:
\begin{enumerate}
\item[(i)] For a general choice of $H$ set $H' \coloneqq \pi^*H$.  
Then the log pair $(X', \Delta'+\exc \pi+H')$ is simple normal crossing and so $\pi$ is in fact a log resolution
of $(\cal F, \Delta+H)$.  

\item[(ii)] We may write $K_{\cal F'}+\Delta'+H' = \pi^*(K_{\cal F}+\Delta+H)+\sum_{i=1}^k a_iE_i$ where  $E_1,\dots,E_k$ are the $\pi$-exceptional divisors
and $a_i>0$ for each $i$.

\item[(iii)] $(K_{\cal F'}+H')\vert_{H'} = K_{\cal F_{H'}}$ where $\cal F_{H'}$ is the restricted foliation, 
see \cite[Lemma 3.1 and Corollary 3.3]{Spicer17} (note that we will prove a more general adjunction statement in Lemma \ref{adjunction} below).  
Moreover,  since $(\cal F', \Delta'+H')$ is foliated log smooth it follows that $(\cal F_{H'}, \Delta'\vert_H)$ is log smooth
(c.f. the proof of Lemma \ref{l_simplegeneric}).
\end{enumerate}
Set $D_i = E_i \cap H$.
By items (ii) and (iii) we see that $K_{\cal F_{H'}}+\Delta'\vert_{H'}\sim_{\bb Q} \sum_{i = 1}^k a_iD_i$.
Let $r\colon H' \rightarrow H''/H$ be a run of the $K_{\cal F_{H'}}+\Delta'\vert_{H'}$-MMP over $H$, see \cite[Corollary 2.26]{Spicer17}.
We know that $H''$ has at worst quotient singularities and 
$(\cal F_{H''}, r_*\Delta'\vert_{H'})$, and hence $\cal F_{H''}$, has terminal singularities where $\cal F_{H''}$ is the transform of $\cal F_H$.
By the negativity lemma we see that $r_*(\sum_{i = 1}^k b_iD_i)= 0$, i.e., $H'' \cong H$.
Thus $\cal F_{H}$ has terminal singularities.

%

By \cite[Corollary I.2.2]{McQuillan08} there is a unique germ of a convergent curve $P \in \gamma$
which is $\cal F_H$-invariant.  In turn, $\gamma$ is contained in a unique convergent separatrix of $\cal F$ at $P$.
Indeed, the unicity follows from the fact that $\gamma$ is the unique $\cal F_H$-invariant curve through $P$. 
The existence of this separatrix is guaranteed by \cite[Main Theorem]{CC92} 
when $\cal F$ has non-dicritical singularities and $X$ is smooth, the case where $\cal F$ has non-dicritical singularities
and $X$ is singular follows by essentially the same arguments, see 
\cite[Corollary 5.4]{Spicer17} (note that the hypothesis on the compactness of the curve $C$ is not needed here).  
By Lemma \ref{l_formalseparatrix} we may extend
this separatrix to a convergent invariant divisor on $\widehat{X}$.

If $(\cal F, \Delta)$ is not terminal at the generic point of $C$, then $C$ is an lc centre and Lemma \ref{l_fdlt-logsmooth} implies that $(\mathcal F,\Delta)$ is log smooth at the general point of $C$. 
This implies then that $\cal F$ has
 simple singularities at a general point $P \in C$, and hence has two
(possibly formal) separatrices at $P$.  By Lemma \ref{l_formalseparatrix} we may extend each of them to formal $\cal F$-invariant
divisors on $\widehat{X}$.  By Lemma \ref{lem_strong_sep_exist} one of these separatrices is convergent and so may
be extended to a convergent invariant divisor.
\end{proof}

\begin{remark}\label{r_canonical-vs-fdlt}
\begin{enumerate}
\item Even if at a general point $P \in C$ there are two distinct separatrices, the invariant divisor
we produce by extending them might be the same for each separatrix.

\item In general, a canonical non-terminal singularity may only admit
a single separatrix formal or otherwise.  Thus, canonical does not imply F-dlt.

\item However, log terminal  does imply F-dlt (keeping in mind
that in general canonical does not imply log terminal for foliations).

\item Recall that, by Remark \ref{r_fdlt=lc},  F-dlt implies log canonical. 
\end{enumerate}
\end{remark}

\subsection{Adjunction}

Recall the following result:
\begin{lemma}
\label{singcomparison}
Let $X$ be a normal threefold, let $\cal F$ be a co-rank one foliation with non-dicritical singularities
and let $(\cal F, \Delta)$ be a foliated pair on $X$. 
Let $D = \sum D_i$ be an $\cal F$-invariant divisor on $X$ 
and suppose that $D$ is $\bb Q$-Cartier.

Suppose that $(\cal F,\Delta)$ is log terminal (resp. log canonical). Then $(X, \Delta+D)$
is log terminal (resp. log canonical).

Suppose that $(\cal F, \Delta)$ is F-dlt and $\lfloor \Delta \rfloor = 0$. Then $(X, \Delta+(1-\epsilon)D)$ is klt for all $\epsilon >0$. 

Suppose that $(\cal F, \Delta)$ is F-dlt and each $D_i$ is smooth in codimension one (but perhaps $\lfloor \Delta \rfloor \neq 0$).
Then $(X, \Delta+D)$ is dlt.

\end{lemma}
\begin{proof}
The first claim is \cite[Lemma 8.14]{Spicer17}, see also 
\cite[Proposition 3.11]{LPT11}.

The second claim is a direct consequence of Lemma \ref{l_fdlt-logsmooth}.


The final claim follows by recalling the third point in Remark \ref{r_log-smooth} and thereby 
noting that if each $D_i$ is smooth in codimension one then a foliated log resolution
of $(\cal F, \Delta)$ is also a log resolution of $(X, \Delta+D)$.
\end{proof}

In fact, Lemma \ref{singcomparison} remains true even 
in the following slightly more general set up.
Let $X$ be a normal threefold with a foliated pair $(\cal F, \Delta)$ and let $V \subset X$ be a closed subvariety and let
$\widehat{X}$ be the formal completion of $X$ along $V$ and let
$\widehat{\cal F}$ be the restriction of $\cal F$ to $\widehat{X}$. 
Then under the hypotheses of Lemma \ref{singcomparison}
the conclusions of the Lemma still hold even if 
$D \subset \widehat{X}$ is a formal
divisor which is $\widehat{\cal F}$-invariant.
We remark that the third point of Remark \ref{r_log-smooth} implies that a log resolution of $(\cal F, \Delta)$, perhaps followed
by some further blow ups in $\sing \cal F$, is a log resolution of $(\widehat{X}, \Delta+D)$, and in particular
$(\widehat{X}, \Delta+D)$ admits a resolution by blow ups in algebraic centres.

\begin{remark}
Let $X$ be a normal threefold, let $\cal F$ be a co-rank one foliation on $X$, and let $P \in X$ be a point.
Let $U$ be an analytic neighbourhood of $P$ and let $\cal F_U$ be the foliation restricted to $U$.
Then, thanks to the existence of a log resolution and since simple singularities are canonical by Lemma \ref{l_simplecanonical}, we know that if $\cal F$ is (log) canonical then
$\cal F_U$ is (log) canonical.  Conversely, if $\cal F_U$ is (log) canonical at $P$ then the same
is true for $\cal F$.
\end{remark}


\begin{lemma}[Adjunction]
\label{adjunction}
Let $X$ be a normal threefold, let $\cal F$ be a co-rank one foliation on $X$.
Let $(\cal F, \epsilon(S)S+\Delta)$ be a foliated pair
where $S$ is a prime  divisor and $\Delta\ge 0$ is a $\bb Q$-divisor on $X$ which does not contain $S$ in its support.  
Let $\nu\colon S^{\nu}\to S$ be the normalisation and let $\cal G$ be the restricted foliation to $S^\nu$
if $S$ is not $\cal F$-invariant and let $\cal G = T_{S^\nu}$ if $S$ is $\cal F$-invariant.

Then, there exists $\Theta\ge 0$ such that  
\[
(K_{\cal F}+\epsilon(S)S+\Delta)\vert_{S^\nu} = K_{\cal G}+\Theta.
\]

Now suppose that $K_X, K_X+\Delta$ and $S$ are $\bb Q$-Cartier and that 
$\cal F$ admits non-dicritical singularities. 
Then the following hold:
\begin{enumerate}
\item Suppose $\epsilon(S) =1$. Suppose moreover that $(\cal F, S+\Delta)$ is lc (resp. lt, resp. F-dlt). Then 
$(\cal G, \Theta)$ is lc (resp. lt, resp. F-dlt).

\item Suppose  $\epsilon(S) = 0$.  Suppose that $\cal F$ is F-dlt. Then $(S^{\nu}, \Theta':= \lfloor \Theta \rfloor_{red}+\{\Theta\})$ is log canonical and if, in addition,  $S$ and $\sing \cal F\cap S$ are normal then $(S, \Theta')$
is dlt.
Moreover, if  $(\cal F, \Delta)$ is log terminal, then $(S^{\nu},\Theta')$ is log terminal. 
\end{enumerate}
\end{lemma}
\begin{proof}
The first claim follows from \cite[Proposition 3.4]{Spicer17} if $S$ is not $\cal F$-invariant
and \cite[Definition 3.11]{AD14} if $S$ is $\cal F$-invariant.  Note that \cite[Definition 3.11]{AD14}
assumes that $\cal F$ has algebraic leaves, however the arguments work equally well in this situation.

We now prove the second claim. 
Let $\pi\colon Y \rightarrow X$ be a foliated log resolution of $(\cal F, \epsilon(S)S+\Delta)$, let $\cal F_Y$ be the pulled back
foliation on $Y$  and
let $T$ be the strict transform of $S$ in $Y$. If $(\mathcal F, \epsilon (S)S+\Delta)$ is F-dlt, then we choose $\pi$ so that $a(E,\mathcal F+\epsilon(S)S,\Delta)> -\epsilon(E)$ for any $\pi$-exceptional divisor $E$.

We may write 
\[
K_{\cal F_Y}+\epsilon(T)T+\Delta' + \sum a_iE_i = \pi^*(K_\cal F+\epsilon(S)S+\Delta)
\]
where $\Delta'$ is the strict transform of $\Delta$ in $Y$, $a_i\in \bb Q$ and the sum is taken over all the  $\pi$-exceptional  divisors. 
Suppose first that $\epsilon(S)=1$ and  note that, by the second point of Remark \ref{r_log-smooth} we have that $T$ is 
generically transverse to $\cal F_Y$ and is not tangent to $\cal F_Y$ along any curves contained in $T$ and contains no one-dimensional components of $\sing \cal F_Y$. Let $\nu'\colon T^\nu\to T$ be the normalisation.  By 
\cite[Corollary 3.3]{Spicer17} we get $(K_{\cal F_Y}+T)\vert_{T^{\nu}}= K_{\cal G_T}+\Theta$ 
where $\cal G_T$ is the restricted foliation to $T^{\nu}$ and $\Theta\ge 0$ is a $\mathbb Q$-divisor.  Note that the image of the  support of $\Theta$ in $T$ is contained in curves where $T$ is tangent to $\cal F_Y$ and on one-dimensional components $\sing \cal F_Y \cap T$. By our previous observation, we see that  $\Theta=0$.
Thus, we get
\[
K_{\cal G_T}+\Delta'\vert_{T^{\nu}} +\sum a_iE_i\vert_{T^{\nu}} = \phi^*(K_{\cal G}+\Theta)
\]
where $\phi\colon T^{\nu}\to S^{\nu}$ is the induced morphism.
By assumption $a_i \leq \epsilon(E_i)$ (resp. $a_i<\epsilon(E_i)$). To prove our result we need
to show that $a_i \leq \epsilon(E_i\vert_{T^{\nu}})$ (resp. $a_i < \epsilon(E_i\vert_{T^{\nu}})$). 

Suppose for sake of contradiction that for some $i$, we have 
$\epsilon(E_i) = 1$, $\epsilon(E_i\vert_{T^{\nu}}) = 0$ and $a_i>0$ (resp. $a_i \ge 0$).
In this case,  consider the blow up of $Y$ at $E_i \cap T$ and let $F$ be the exceptional divisor.
Notice that $\epsilon(F) = 0$.  However, this is
a blow up of discrepancy $ \leq -a_i < \epsilon(F)$ (resp. $\leq -a_i \le \epsilon(F)$, resp. $ \leq -a_i < \epsilon(F)$), and so we see that
$(\cal F_Y, T+\Delta'+\sum a_iE_i)$ is not lc (resp. lt, resp. F-dlt), hence $(\cal F, S+\Delta)$ is not lc
(resp. lt, resp. F-dlt). Thus, (1) follows. 

\medskip 

Suppose now that $\epsilon(S)=0$ and let $\widehat{X}$ be the formal completion of $X$ along $S$. Let $T$ be the sum of all the (formal) 
invariant divisors meeting $S$, and which are not equal to $S$.  Suppose for the moment that $T$ is $\bb Q$-Cartier.
Observe that if $\sing \cal F\cap S$ is normal then each component of $T$ is smooth in codimension one.
Then by Lemma \ref{singcomparison} we know that $(\widehat{X}, S+T+\Delta)$ is dlt,
furthermore, as in the proof of \cite[Lemma 8.9]{Spicer17}, 
it follows that the different of $(\widehat{X}, S+T+\Delta)$ with respect to $S$ is exactly
$\Theta'$.  We then apply adjunction for varieties to conclude.

We now handle the case where $T$ is not necessarily $\bb Q$-Cartier.
By Lemma \ref{singcomparison} and  since $(\cal F, \Delta)$ is F-dlt, it follows that  $X$ is klt.  Suppose for sake of contradiction 
that there exists a point $P \in \widehat{X}$
such that $T$ is not $\bb Q$-Cartier at $P$.  Then, since $X$ is klt, if $U$ is a small analytic neighbourhood of
$P \in X$ we may find a small $\bb Q$-factorialisation 
$f\colon Y \rightarrow U$ such that $Y$ is analytically $\bb Q$-factorial by Remark \ref{r_qfactorialisation}.  
Let $\Sigma = f^{-1}(P)$, let
$\widehat{U}$ be the formal completion of $U$ at $P$ and let $\widehat{Y}$ be the formal completion
along $\Sigma$ so that we have a morphism $\hat{f}\colon \widehat{Y} \rightarrow \widehat{U}$.

Observe that $K_Y = f^*K_U$ and if we let $S' = f^*S$ then $\Sigma \subset S'$.
Let $\widehat{S}'$ denote the restriction of $S'$ to $\widehat{Y}$.
Let $\Sigma_0$ be a component of $\Sigma$ which meets $\hat{f}_*^{-1}T$.  Note that $K_Y\cdot \Sigma_0 = 0$
and that, after possibly shrinking $U$, we may assume that $\Sigma_0$ spans an extremal ray in $\overline{NE}(Y/U)$.
Thus, there exists a flop of $\Sigma_0$ over $U$, call it $\psi\colon Y \dashrightarrow Y'/U$ and let $\Sigma'_0$
be the flopped curve.  Observe that  $\Sigma_0 \subset \hat{f}_*^{-1}T$ or $\Sigma'_0 \subset \psi_*\hat{f}_*^{-1}T$.
So up to replacing $Y$ by $Y'$, we may assume that $\Sigma_0 \subset \hat{f}_*^{-1}T$.
Next, note that $\Sigma_0 \subset \widehat{S}' \cap \hat{f}_*^{-1}(T) \subset \sing \cal H$ where $\cal H \coloneqq f^{-1}\cal F$
and the latter inclusion holds because $S'$ and $\hat{f}_*^{-1}(T)$ are both $\cal H$-invariant.

By Remark \ref{r_fdlt=lc}, the foliation $\cal F$ has log canonical singularities. Thus, since $f$ is small,   $\cal H$ has also log canonical singularities. Note also by Lemma \ref{l_fdlt-terminal} 
that $\cal H$ has canonical, and not terminal, singularities at the generic point of $\Sigma_0$, since
$\Sigma_0$ is contained in the intersection of two (formal) invariant divisors.
Thus we may find a birational morphism $b\colon Y' \rightarrow Y$ 
which extracts a divisor $E$ dominating $\Sigma_0$ such that $0 = a(E, \cal H) = a(E, \cal F)$.
Next, observe that $a(E, \cal F) \geq a(E, \cal F, \Delta) \geq -\epsilon(E) = 0$ and so 
all these inequalities must be equalities and it follows that $P$ is a lc centre
of $(\cal F, \Delta)$.  By Lemma \ref{l_fdlt-logsmooth}, $X$ is smooth at $P$, and so $T$ is necessarily Cartier at $P$, 
a contradiction. Thus, (2) follows. 
\end{proof}

\begin{remark}
Note that, using the same proof,  the first part of Lemma \ref{adjunction} can be generalised to the formal setting. More specifically, 
let $X$ be a normal threefold and let $\cal F$ be a co-rank one foliation with non-dicritical singularities
on $X$.
Suppose that $(\cal F, \Delta)$ is a F-dlt pair on $X$, let $S$ be an $\cal F$-invariant divisor,  let $\widehat{X}$ be the formal completion 
of $X$ along $S$
and let $\nu\colon S^\nu \rightarrow S$ be the normalisation.
Then, we may write  
\[
(K_{\cal F}+\Delta)\vert_{S^\nu} = K_{S^\nu}+\Theta
\]
where $\Theta\ge 0$ is a $\mathbb Q$-divisor on $S^{\nu}$. 

\end{remark}

\begin{corollary}
\label{basicpropertyfdlt2}
Let $X$ be a normal threefold and let $\cal F$ be a co-rank one foliation with non-dicritical singularities
on $X$.
Suppose that $(\cal F, \Delta)$ is a F-dlt pair on $X$, let $S$ be an $\cal F$-invariant divisor such that $(X,\Delta+S)$ is a log pair, let $\widehat{X}$ be the formal completion 
of $X$ along $S$
and let $\nu\colon S^\nu \rightarrow S$ be the normalisation.
Let $T_1, ..., T_k$ be the collection of $\cal F$-invariant divisors on $\widehat{X}$
not equal to $S$ and suppose that $T =\sum T_i$ is $\bb Q$-Cartier.
Write 
\[
(K_{\cal F}+\Delta)\vert_{S^\nu} = K_{S^\nu}+\Theta
\]
and let $\Theta'$ be the different of $(\widehat X,\Delta+S+T)$ with respect to $S$.


Then $\Theta' \leq \Theta$
with equality along the curves contained in $\sing X$. Moreover, for any irreducible curve $C$ in $S$, at the general point of which $\cal F$ has simple singularities and along which $S$ is a strong separatrix, the coefficients of $\Theta$ and $\Theta'$ coincide. 
\end{corollary}
\begin{proof}
If $C$ is a curve contained in the singular locus of $S$,  then  $C\subset \sing \cal F \cup \sing X$. Since $T_1,\dots,T_k$ are $\cal F$-invariant, it follows that $\Theta$ and $\Theta'$
are supported on $\nu^{-1}(\sing \cal F \cup \sing X)$.

Let $C$ be an irreducible curve in $S^{\nu}$. By Lemma \ref{basicpropertyfdlt}, we see that if $\nu(C) \subset \sing X \cap S$
then $\cal F$ is terminal at the generic point of $\nu(C)$.  
In this case, by \cite[Lemma 8.9]{Spicer17}, it follows that the coefficient of $C$ in
$\Theta$ and $\Theta'$ is the same.   

So we may assume that $X$ is smooth at the generic point of $C$.
If $\nu(C) \subset \sing \cal F\cap S$ then $C$ is an lc center of $\cal F$
and so by Lemma \ref{l_fdlt-logsmooth}, $\cal F$ has simple singularities at a general point of $C$.
It follows that $S+T$ is a normal crossings divisor at the generic point of $C$ and
so the coefficient of $C$ in $\Theta'$ is equal to $1$.

Let $\omega$ be a $1$-form defining $\cal F$ at a general point of $C$
and let $\omega_S$ be the restriction of $\omega$ to $S$.  The coefficient
of $C$ in $\Theta$ is the order of the zero of $\omega_S$ along $C$, 
in particular, it is an integer $\geq 1$.  When $S$ is the strong separatrix along $C$
then the order of the zero of $\omega_S$ along $C$ is equal to $1$.
%
%
\end{proof}

\begin{remark}
We emphasise that in Corollary \ref{basicpropertyfdlt2} we are not assuming that $X$ is quasi-projective or that $S$ is an algebraic variety.
Later on we will apply this Lemma in the case where $S$ is an analytic divisor on an analytic neighborhood of a proper curve
$C$ contained in a quasi-projective threefold.
\end{remark}

\begin{lemma}
\label{l_tangentlccentre}
Let $X$ be a $\bb Q$-factorial quasi-projective threefold and let $\cal F$ be a co-rank 1 foliation with non-dicritical singularities on $X$.
Suppose that $(\cal F, \Delta)$ is F-dlt. 
Let $C$ be a one-dimensional lc centre of $(\cal F, \Delta)$
tangent to $\cal F$.

Then there exists  $\Theta \geq 0$ such that 
\[
(K_{\cal F}+\Delta)\vert_{C^\nu} = K_{C^\nu}+\Theta
\]
where $\nu\colon C^\nu \rightarrow C$ is the normalisation.
Moreover, if $P$ is contained in the support of $\lfloor \Theta \rfloor$ then $\nu(P)$ is an lc centre of  $(\cal F, \Delta)$.
\end{lemma}
\begin{proof}
By Lemma \ref{l_fdlt-logsmooth}, it follows that $(\cal F,\Delta)$ is log smooth at the general point of $C$. 
By Lemma \ref{l_fdlt-terminal}, 
there exists a germ of an invariant divisor $S$ containing $C$ and if $C\subset \sing \cal F$, then we may choose $S$ so that it is the strong 
separatrix at a general point of $C$. 
Let $\nu'\colon S^\nu \rightarrow S$ be the normalisation of $S$ and set $C' = \nu'^{-1}(C)$.
Then it is easy to check that  
\[
(K_{\cal F}+\Delta)\vert_{S^\nu} = K_{S^\nu}+C'+\Delta_{S^\nu}
\]
 for some $\bb Q$-divisor $\Delta_{S^\nu}\ge 0$, whose support does not contain $C$. 
 
 Note that we have a morphism $C^\nu \rightarrow C'$ and so we may write
 $(K_{S^\nu}+C'+\Delta_{S^\nu})|_{C^\nu}=K_{C^\nu}+\Theta$. 
Thus, by adjunction for varieties, the image of the non-klt locus of $({C^\nu}, \Theta)$ is supported
on the non-klt locus of $(S^\nu, C'+\Delta_{S^\nu})$.
We claim these centres are contained in 
lc centres of $(\cal F, \Delta)$. 
Let $P\in S$, by Lemma \ref{adjunction} we see that
if $(\cal F, \Delta)$ is log terminal at $P$ then $(S^\nu, C'+\Delta_{S^\nu})$ must also be log terminal at $P$, and so
every point in the non-klt locus of $(S^\nu, C'+\Delta_{S^\nu})$ is an lc centre of $(\cal F, \Delta)$. 
\end{proof}

\begin{lemma}
\label{lem_trans_diff_comp}
Let $X$ be a normal threefold and let $\cal F$ be a corank one foliation with non-dicritical singularities.
Let $S$ be an irreducible divisor transverse to $\cal F$ and suppose that $K_X+\Delta+S$ and $K_{\cal F}+\Delta+S$
are $\bb Q$-Cartier where $\Delta \geq 0$ and $S$ is not contained in the support of $\Delta$.
Write $n^*(K_X+\Delta+S) = K_{S^n}+\Theta$ and $n^*(K_{\cal F}+\Delta+S)= K_{\cal F_S}+\Theta'$
where $n \colon S^n \rightarrow S$ is the normalisation.  Let $C \subset S$ be a curve
such that $n(C)$ is transverse to $\cal F$.  

Then the coefficient of $C$ in $\Theta$ is
equal to the coefficient of $C$ in $\Theta'$.
\end{lemma}
\begin{proof}
By \cite[Lemma 3.11]{Spicer17} we have that the usual discrepancy of any divisor
dominating $n(C)$ is equal to the foliated discrepancy.  
By construction, the coefficient of the different along $C$ is only a function of the discrepancies of divisors dominating $C$ 
with respect to the pairs
$(X, \Delta+S)$ and $(\cal F, \Delta+S)$ and our claimed equality follows.
\end{proof}

\subsection{Bertini type results}

\begin{lemma}\label{l_perturbation}
Let $X$ be a normal projective threefold and let $\cal F$ be a co-rank one foliation on $X$ with non-dicritical singularities.
Let $(\cal F,\Delta)$ be a F-dlt pair such that $\lfloor \Delta\rfloor=0$ and let $A$ be an ample $\mathbb Q$-divisor on $X$.

Then there exists an effective $\mathbb Q$-divisor $A' \sim_{\mathbb Q}A$ such that 
\begin{enumerate}
\item $(\cal F,\Delta+A')$ is also F-dlt; 
\item $\lfloor \Delta+A'\rfloor=0$; and 
\item  the support of  $A'$ does not contain any lc centre of $(\cal F,\Delta)$.
\end{enumerate}
\end{lemma}
\begin{proof}
Lemma \ref{l_fdlt-logsmooth} implies that $(\cal F,\Delta)$ admits a foliated log resolution $\pi\colon Y\to X$  which only extracts divisors $E$ of discrepancy $>-\epsilon(E)$ and which is an isomorphism at the general point of $\pi^{-1}(W)$ for any  $W\subset X$ which is a lc centre of $(\cal F,\Delta)$. We claim that there exists a sufficiently divisible positive integer $m$ such that if $H\in |mA|$ is a general element and  $A'\coloneqq \frac 1m H$ then
\begin{enumerate}
\item [(i)] the support of  $A'$ does not contain any lc centre of $(\cal F,\Delta)$;
\item [(ii)] each irreducible component of $A'$ is generically transverse to $\mathcal F$ and 
$\lfloor \Delta+A'\rfloor=0$; 
 \item[(iii)] if  $x\in X$ is the general point  of a log canonical centre of $(\cal F,\Delta)$ then, $(\cal F,\Delta+A')$ is log smooth in a neighbourhood of $x$; and 
 \item[(iv)] if $(\cal F,\Delta)$ is klt in a neighbourhood of $x\in X$ then  $(\cal F,\Delta+A')$ is also klt in a neighbourhood of $x$. 
 \end{enumerate}

Assuming the claim, let $\cal G\coloneqq\pi^{-1}\cal F$.
We may write
\[K_{\cal G}+\Gamma=\pi^*(K_{\cal F}+\Delta)
\]
for some $\mathbb Q$-divisor $\Gamma$ on $Y$. 
Then there exists a foliated log resolution  $\pi'\colon Y'\to Y$ of $(\cal G,\Gamma+\pi^*A')$ such that if $p\colon Y'\to X$ is the induced morphism, then $\pi'$ is an isomorphism at the general point of  $p^{-1}(W)$ for any lc centre $W$ of $(\cal F,\Delta)$. Thus,  $a(F,\cal F,\Delta+A')>-\epsilon (F)$ for any prime divisor $F$  on $Y'$ which is $p$-exceptional and the Lemma follows.

\medskip 

We now prove the claim. Since $\lfloor \Delta\rfloor=0$, Proposition \ref{p_finmanylccenters} implies that if $m$ is sufficiently divisible then $H$ does not contain any lc centre of $(\cal F,\Delta)$. Thus, (i) holds. It is easy to check that (ii) holds.

Now we check (iii). Let $W$ be a log canonical  centre of $(\cal F,\Delta)$.  By Lemma \ref{l_fdlt-logsmooth}, $(\cal F,\Delta)$ is foliated log smooth at the generic point of $W$. In particular, $\cal F$ admits simple singularities at the general point of $W$. If $x\in \sing \cal F$ is an isolated point, then we may choose $m$ sufficiently divisible so that $x$ is not contained in the support of $H$. Similarly, we may assume that $x$ is the general point of  a one-dimensional component $C$ of $\sing \cal F$. Thus, we may choose $m$ so that $H$ is transverse to $C$ and, in particular, if $\Sigma_1,\Sigma_2$ are, possibly formal, $\cal F$-invariant divisors passing through $x$, then $C$ is contained in $\Sigma_1\cap \Sigma_2$ and $(X,\Sigma_1+\Sigma_2+H)$ is a normal crossing pair at $x$. 
It follows that $(\cal F,\Delta+A')$ is  log smooth at $x$. Therefore (iii) holds.

Finally, we check (iv). 
First, let $U\coloneqq X \setminus \sing X$ and let $x\in U$ be a point such that $(\cal F,\Delta)$ is klt in a neighborhood of 
$x$.  If $x \notin H$ then there is nothing to prove, and if $x \in H$ and 
$H$ is general we may also assume that $x \in U \setminus \sing \cal F$.  
%
%
Moreover, if $H$ is sufficiently general then by \cite[Lemma 2.9]{AD19} it follows that $(\cal F, H)$ is log smooth on $U \setminus \sing \cal F$, 
except perhaps
at finitely many points, $P_1, ..., P_m$.  In fact a dimension count as in \cite[Lemma 2.9]{AD19} shows that a general choice of
$H$ will have at worst an order 2 tangency with $\cal F$ at $P_i$, i.e., in some local analytic coordinates $(x, y, z)$ 
we have $\cal F$ is defined by $dz$ and $H = \{z + x^2+y^2 = 0\}$.  In particular, at the $P_i$ 
we know that $(\cal F, \frac{1}{4} H)$ is klt.  Note
that $\pi$ is a foliated log resolution of $(\cal F, \Delta+H)$ along $U$ except at $P_1, ..., P_m$.  It then follows that 
for $m$ sufficiently divisible and $H$ sufficiently general, we have 
that $(\cal F, \Delta+A')$ is klt on $U\setminus \sing \cal F$.

Next, write $\sing X = Z_0\cup Z_1$ where $Z_0$ are the components of $\sing X$ which are tangent to $\cal F$
and $Z_1$ are the components which are transverse.  Suppose that $x \in Z_0$.  Observe that $\pi^{-1}(Z_0)$
is by definition $\cal G$-invariant and so for a general choice of $H$, we know
that $(\cal G, \Gamma+\pi^*H)$ is foliated log smooth in a neighbourhood of $\pi^{-1}(x)$.  It therefore follows immediately
that, for $m$ sufficiently divisible, $(\cal F, \Delta+A')$ is klt in a neighbourhood of $x$.

Finally suppose that $x \in Z_1$.  Without loss of generality we may assume that $x$ is a general point of $Z_1$.
By non-dicriticality we know that $\pi^{-1}(x)$ is tangent to $\cal G$.  
In fact, arguing as in the proof of Lemma \ref{lem_equiv_dicrit} we see
that $\cal G$ is smooth on a neighbourhood of $\pi^{-1}(x)$ and so by Lemma \ref{l_formalseparatrix} it follows that there exists a $\cal G$-invariant divisor $S$ defined on an analytic neighbourhood of $\pi^{-1}(x)$ such that $\pi^{-1}(x) \subset S$.

We claim that $(\cal G, \Gamma+\pi^*A')$ is klt in a neighbourhood of $\pi^{-1}(x)$ if and only if 
$(Y, \Gamma+S+\pi^*A')$ is log terminal in a neighbourhood of $\pi^{-1}(x)$.  Supposing the claim, then we are done 
by observing that if $(X, \Delta+\pi_*S)$ is log terminal at $x$ then $(X, \Delta+\pi_*S+A')$ is log terminal at $x$
for sufficiently divisible $m$ and sufficiently general $H$.

We now prove the above claim.  After possibly shrinking $x$, we may consider a log resolution
 $\mu\colon \widetilde{Y} \rightarrow Y$ of $(Y, \Gamma+S+\pi^*A')$ which is an isomorphism away from $(\pi\circ \mu)^{-1}(x)$. Let $E$ be the support of  $\exc \mu$. Then 
$K_{\mu^{-1}\cal G}$ and $K_{\widetilde{Y}}+\mu_*^{-1}S+E$ are numerically equivalent over $Y$.  Indeed, if $E_0$ is
any component of $E$ then, since $\mu_*^{-1}S \cap E_0, E' \cap E_0 \subset \sing{\mu^{-1}\cal G}$ where $E'$ is any other component of $E$, 
Corollary \ref{basicpropertyfdlt2} implies that 
\[
K_{\mu^{-1}\cal G}\vert_{E_0} = K_{E_0}+(E-E_0+\mu_*^{-1}S)\vert_{E_0}
\]
We also have
\[
(K_{\widetilde{Y}}+E+\mu_*^{-1}S)|_{E_0} = K_{E_0}+(E-E_0+\mu_*^{-1}S)\vert_{E_0}.
\]  
It follows that the log discrepancies
of $(Y, \Gamma+S+\pi^*A')$ are equal to the discrepancies of $(\cal G, \Gamma+\pi^*A')$ as required. Therefore (iv) holds. 
\end{proof}

\begin{corollary}
\label{c_canonical_perturb}
Set up as in Lemma \ref{l_perturbation}. Suppose in addition that $(\cal F, \Delta)$ is log smooth and has canonical singularities.

Then there exists an effective $\mathbb Q$-divisor $A' \sim_{\mathbb Q}A$ such that 
$(\cal F,\Delta+A')$ has canonical singularities.
\end{corollary}
\begin{proof}
By Lemma \ref{l_perturbation} we may find an effective $\mathbb Q$-divisor $A' \sim_{\mathbb Q}A$
so that $(\cal F, \Delta+A')$ is F-dlt, hence log canonical by Remark \ref{r_fdlt=lc}.
It suffices to show that for any divisor $E$ on some model
of $X$ that $a(E, \cal F, \Delta+A') \geq 0$. 
Clearly this holds for all foliation invariant divisors, so let $E$ be one such divisor which is not foliation
invariant and let $W$ be the centre of $E$ on $X$.

Since $\cal F$ has non-dicritical singularities, by Remark \ref{r_simple=ndc}
we know that $W$ is necessarily transverse to $\cal F$. In particular, by \cite[Lemma 3.11]{Spicer17},
we see that the foliation discrepancies for divisors centred over $W$ are equal to the classical
discrepancies.

If $A'$ is sufficiently general with sufficiently small coefficients then we know that
$(X, \Delta+A')$ has canonical singularities at the generic point of $W$, and so
$(\cal F, \Delta+A')$ has canonical singularities at the generic point of $W$, which implies our claim.
\end{proof}

\begin{lemma}\label{l_exc}
Let $X$ be a $\mathbb Q$-factorial projective threefold and let $\mathcal F$ be a co-rank one foliation on $X$ with non-dicritical singularities. 
Let $\Delta=A+B$ be a $\bb Q$-divisor such that $(\mathcal F,\Delta)$ is a F-dlt pair, $A\ge 0$ is an ample $\bb Q$-divisor and $B\ge 0$. 

Then there exist $\bb Q$-divisors $A',B'\ge 0$ such that $A'$ is ample and if $\Delta'=A'+B'$ then 
\begin{enumerate}
\item $\Delta'\sim_{\bb Q} \Delta$, 
\item $\lfloor \Delta'\rfloor=0$, 
\item $(\mathcal F, \Delta')$ is F-dlt and
\item  if $E$ is a valuation such that $a(E,\mathcal F,\Delta') = -\epsilon(E)$ then $\epsilon(E)=0$ and $a(E,\mathcal F)=0$. 
\end{enumerate}
\end{lemma}
\begin{proof}
Let $\delta >0$ be sufficiently small rational number, so that $A+\delta B$ is ample and let $B'\coloneqq (1-\delta)B$. 
By Lemma \ref{l_perturbation}, there exists a $\mathbb Q$-divisor 
\[0\le A'\sim_{\bb Q} A+\delta B\]
such that, if $\Delta'\coloneqq (1-\delta)B+A'$ then (1),(2) and (3) hold and  the support of $A'$ does not contain any lc centre of $(\cal F,(1-\delta)B)$. Thus, 
if $E$ is a valuation such that $a(E,\mathcal F,\Delta') = -\epsilon(E)$ then $a(E,\mathcal F,(1-\delta)B)=-\epsilon(E)$. Since 
\[a(E,\cal F,B)\ge a(E,\cal F,\Delta)\ge -\epsilon(E),
\] it follows that the centre $W$ of $E$ in $X$ is not contained in the support of $B$ and, in particular, $a(E,\cal F)=-\epsilon(E)$.  By Lemma \ref{l_fdlt-logsmooth}, we have that $\cal F$ has simple singularities along the generic point of $W$. As in the proof of Lemma \ref{l_simplecanonical}, it follows that $\epsilon(E)=0$. Thus, (4) follows.  
\end{proof}

\begin{lemma}\label{l_A+B}
Let $X$ be a  $\bb Q$-factorial projective threefold and let $\mathcal F$ be a co-rank one foliation on $X$ with non-dicritical singularities. 
Let $\Delta=A+B$ be a $\bb Q$-divisor such that $(\mathcal F,\Delta)$ is a F-dlt pair, $A\ge 0$ is an ample $\bb Q$-divisor and $B\ge 0$. Let $\varphi\colon X\dashrightarrow X'$ be a sequence of steps 
of the $(K_{\mathcal F}+\Delta)$-MMP and let $\mathcal F'$ be the transformed foliation on $X'$. 

Then, there exist $\bb Q$-divisors $A'\ge 0$ and $C'\ge 0$ on $X'$ such that 
\begin{enumerate}
\item $\varphi_* A \sim_{\bb Q}  A'+C'$,
\item $A'$ is ample, and 
\item if $\Delta':=A'+C'+\varphi_*B$ then $\Delta'\sim_{\bb Q} \varphi_*\Delta$ and $(\mathcal F',\Delta')$ is F-dlt. 
  \end{enumerate}
\end{lemma}

\begin{proof} By Lemma \ref{l_exc}, we may assume that $\lfloor \Delta\rfloor=0$. 
We may also assume  that $\varphi\colon X\dashrightarrow X'$ is a $(K_{\mathcal F}+\Delta)$-flip (resp. $(K_{\mathcal F}+\Delta)$-divisorial contraction). 
Let $H\ge 0$ be a general ample $\bb Q$-divisor on $X'$. After possibly replacing $H$ by a smaller multiple, we may assume that if $H_X$ is the strict transform of $H$ in $X$ then $A-H_X$ is ample. In particular, by Lemma \ref{l_perturbation}, there exists an effective $\bb Q$-divisor $C\sim_{\mathbb Q} A-H_X$ and $\epsilon>0$ sufficiently small such that $(\mathcal F,\Delta+\epsilon C)$ is F-dlt 
and $\varphi$ is still a  $(K_{\mathcal F}+\Delta+\epsilon C)$-flip (resp. $(K_{\mathcal F}+\Delta+\epsilon C)$-divisorial contraction). 
Thus, if $\mathcal F'$ is the transformed foliation on $X'$, then  Lemma \ref{l_mmp-fdlt} implies that $(\mathcal F',\varphi_*(\Delta+\epsilon C))$ is $F$-dlt.
By Remark \ref{r_fdlt=lc}, we have that  $(\mathcal F',\varphi_*(\Delta+\epsilon C))$ is log canonical and, therefore, 
 $\varphi_*C$ does not contain any lc centre of $(\mathcal F',\varphi_*\Delta)$. 

Moreover, $\varphi_*A\sim_{\bb Q}\varphi_*C+H$. Thus, for a sufficiently small rational number $\delta>0$, we may choose 
$A'=\delta H$ and $C'=(1-\delta)\varphi_* A + \delta \varphi_*C$ and the claim follows.
\end{proof}

\subsection{F-dlt modification}

\begin{defn}\label{d_fdltm}
Let $\cal F$ be a co-rank one foliation on a 
normal  projective variety $X$.
Let $(\cal F, \Delta)$ be a foliated pair.
A {\bf F-dlt modification} for the foliated pair $(\cal F, \Delta)$ is  a birational projective morphism $\pi\colon Y \rightarrow X$
such that if $\cal G$ is the pulled back foliation on $Y$ then 
$(\cal G,\pi_*^{-1} \Delta+\sum \epsilon(E_i)E_i)$ is
F-dlt where we sum over all $\pi$-exceptional divisors
and 
\[
K_{\cal G}+\pi_*^{-1} \Delta +\sum \epsilon(E_i)E_i+F= \pi^*(K_{\cal F}+\Delta)
\]
for some $\pi$-exceptional $\bb Q$-divisor $F \geq 0$ on $Y$.

In particular, if $(\cal F, \Delta)$ is lc then $\pi$ only extracts
divisors $E_i$ of discrepancy $-\epsilon(E_i)$.
\end{defn}

Theorem \ref{t_existencefdlt} below implies the existence of a F-dlt modification for any foliated pair $(\cal F,\Delta)$ on a normal projective variety of dimension at most three. 

\subsection{F-dlt cone and contraction theorem}
In \cite{Spicer17} the cone theorem 
is proved under the hypotheses that $X$ is a $\bb Q$-factorial threefold, $\cal F$ is non-dicritical, $(\cal F, \Delta)$ 
has canonical singularities and the contraction theorem is proved under the additional hypothesis
that $(\cal F, \Delta)$ is terminal at the generic points of $\sing X$.

In fact, it is possible to prove the cone and contraction theorem
under the hypothesis that $(\cal F, \Delta)$ is F-dlt (rather than canonical). Even better, it is possible to prove
the cone theorem in the case that $X$ is not necessarily $\bb Q$-factorial but $X$ is potentially klt. 
We still require that $\cal F$ has non-dicritical singularities.
We explain the required modifications to the cone theorem first.

Since $X$ is potentially klt, Theorem \ref{thm_existence_q_fact_classical} implies that there exists a small $\bb Q$-factorialisation $\pi\colon Y \rightarrow X$.
Write $K_{\cal G}+\Gamma = \pi^*(K_{\cal F}+\Delta)$, where $\cal G$ is the pulled back foliation on $Y$. 

\begin{lemma}
\label{lem_fdlt_small_mod}
Set up as above.

Then  $(\cal G, \Gamma)$ is F-dlt. 
\end{lemma}
\begin{proof}
By Lemma \ref{l_fdlt-logsmooth}, it follows that $(\cal F,\Delta)$ is foliated log smooth at the generic point of every lc centre of $(\cal F,\Delta)$. In particular, $\pi(\exc \pi)$ does not contain any lc centre of $(\cal F,\Delta)$. 
Let $\rho\colon X'\to X$ be a foliated log resolution of $(\cal F,\Delta)$ such that $a(E,\cal F,\Delta)>-\epsilon(E)$ for any $\rho$-exceptional divisor $E$. Let $p\colon \overline Y\to Y$ and  $q\colon \overline Y\to X'$ be proper birational morphisms which resolve the indeterminacy locus of the induced birational contraction $X'\dashrightarrow Y$. Let $\overline {\cal F}$ be the transformed foliation on $\overline Y$ and let $\overline \Delta$ be the strict transform of $\Delta$ in $\overline Y$. 
Let $Y'\to \overline Y$ be a foliated log resolution of $(\overline {\cal F},\overline \Delta+G)$, where $G$ is the support of the exceptional divisor of $q$, and let $\rho'\colon Y'\to Y$ be the induced morphism. We may assume that the morphism $Y'\to X'$ is an isomorphism away from $\phi^{-1}(\exc \pi)$. Suppose by contradiction that there exists a $\rho'$-exceptional divisor $F$ such that 
$a(F,\cal G,\Gamma)=-\epsilon(F)$.
Then the centre of $F$ in $Y$ does not intersect $\exc \pi$ and therefore the image of $F$ in $X'$ is a divisor which is $\rho$-exceptional. Since $a(F,\cal F,\Delta)=a(F,\cal G,\Gamma)=-\epsilon(F)$, we get a contradiction and our claim follows. 
\end{proof}

\medskip

Notice that $(K_{\cal G}+\Gamma)\cdot C = 0$ for every $\pi$-exceptional curve $C$ so 
if $R$ is a $(K_{\cal F}+\Delta)$-negative extremal ray then there exists a $(K_{\cal G}+\Gamma)$-negative extremal 
ray $R'$ such that $\pi_*R' = R$.  Thus, by replacing $(\cal F, \Delta)$ by $(\cal G, \Gamma)$
we may freely assume that $X$ is $\bb Q$-factorial.

Let $R$ be a $(K_{\cal F}+\Delta)$-negative extremal ray and let $H_R$ be a supporting hyperlane to $R$.
We may assume that $H_R-(K_{\cal F}+\Delta)$ is ample and, in particular, $H^3_R > (K_{\cal F}+\Delta)\cdot H^2_R$.
 
If $H^3_R=0$ then \cite[Corollary 2.28]{Spicer17} implies that $X$ is covered by rational curves tangent
to $\cal F$ which span $R$.
Otherwise, if $H^3_R=0$ then $H_R$ is big and there exists an effective divisor $S$ such that
$S$ is negative on $R$.
As in the proof of  \cite[Lemma 4.7]{Spicer17}, it follows that  if $(\cal F, \Delta)$ is log canonical
then $R$ is spanned by a curve $C$.
In either case, every $(K_{\cal F}+\Delta)$-negative extremal ray is spanned by a curve.


\begin{lemma}\label{l_curvesinrayaretangent}
Let $X$ be a normal threefold and let $\cal F$ be a co-rank one foliation with non-dicritical singularities on $X$.
Suppose that $(\cal F, \Delta)$ is a log canonical foliated pair and that $X$ is potentially klt.
Let $R$ be a $(K_{\cal F}+\Delta)$-negative extremal ray and suppose that $\loc R \neq X$.  Suppose that $[C] \in R$. 

 Then $C$ is
tangent to $\cal F$.
\end{lemma}
\begin{proof}
This is proven in \cite[Lemma 8.10]{Spicer17} under the assumption that $X$ is $\bb Q$-factorial and klt.
However, one can observe that the proof does not rely on either of these hypotheses.

For the reader's convenience, we briefly sketch the relevant ideas.  Let $H_R$ be
the supporting hyperplane to $R$ and suppose for sake of contradiction
that $C$ is transverse to $\cal F$.  Since $\loc R \neq X$, it follows that  $H_R$ is big.
Let $A$ be an ample $\bb Q$-Cartier divisor so that $H_R-A$ is still big.
We may then find a $\mathbb Q$-divisor $0 \leq G \sim_{\bb Q} H_R-A$ and notice that $G\cdot C<0$.

Let $\lambda\ge 0$ be the largest real number so that $(\cal F, \Delta+\lambda G)$
is log canonical at the generic point of $C$.  On one hand, we know that $(K_{\cal F}+\Delta+\lambda G)\cdot C<0$.
On the other hand, 
by replacing $X$ by a small $\bb Q$-factorialisation (which is necessarily an isomorphism at the generic point of $C$)
we may apply sub-adjunction (cf. \cite[Theorem 4.5]{Spicer17}) to see that
since $C$ is transverse to $\cal F$ then $(K_{\cal F}+\Delta+\lambda G)\cdot C \geq 0$, a contradiction.
\end{proof}

By Lemma \ref{basicpropertyfdlt}, if $C$ is tangent to $\cal F$ then $(\cal F, \Delta)$
is canonical at the generic point of $C$ and we have reduced to the cone theorem
in \cite[Theorem 7.1]{Spicer17}.

Thus, we have:

\begin{theorem}\label{t_cone-dlt}
Let $X$ be a normal projective threefold and let $\mathcal F$ be a co-rank one foliation with non-dicritical singularities. 
Suppose that $X$ is potentially klt. 
Let  $(\mathcal F,\Delta)$ be a F-dlt pair and let $H$ be an ample $\bb Q$-divisor. 

Then there exist countably many curves $\xi_1,\xi_2,\dots$ such that 
\[
\overline {NE(X)}=\overline{NE(X)}_{K_{\mathcal F}+\Delta\ge 0}+ \sum\bb R_+[\xi_i].
\]

Furthermore, for each $i$,  $\xi_i$ is a rational curve tangent to $\cal F$ such that $(K_{\mathcal F}+\Delta)\cdot \xi_i\ge -6$,
and if $C \subset X$ is a curve such that $[C] \in \bb R_+[\xi_i]$ and $\loc {\bb R_+[\xi_i]} \neq X$ then $C$ is tangent to $\cal F$.
If we assume in addition that $X$ is $\bb Q$-factorial then if
$[C] \in \bb R_+[\xi_i]$ then $C$ is tangent to $\cal F$.

%

In particular, there exist only finitely many  $(K_{\mathcal F}+ \Delta+H)$-negative extremal rays. 
\end{theorem}

\begin{remark}
We will return to the problem of constructing contractions in the $\bb Q$-factorial case in Theorem \ref{contractionsareextremal}
and in the non-$\bb Q$-factorial case in Section \ref{s_nonqfactorial}.
\end{remark}

We now show that the MMP preserves non-dicritical singularities.

\begin{lemma}
\label{l_image_cont_non-dicrti}
Let $X$ be a normal threefold and let $\cal F$ be a non-dicritical co-rank one foliation on $X$.
Suppose that $(\cal F, \Delta)$ is log canonical and that $X$ is potentially klt. 

Let $R$ be a $(K_{\cal F}+\Delta)$-negative extremal ray and let $c\colon X \rightarrow Y$ be the contraction
associated to $R$.  Suppose that $c$ is a birational morphism and let $\cal G \coloneqq c_*\cal F$.
Then $\cal G$ is non-dicritical.  In particular, if $c$ is a flipping contraction and assuming that its flip $\phi\colon X \dashrightarrow X^+$ exists, then $\cal F^+\coloneqq \phi_*\cal F$ is non-dicritical.
\end{lemma}
\begin{proof}
To prove our first claim let $q \in Y$ be a point and let $\pi\colon Y' \rightarrow Y$ be any sequence of blow ups such that $\pi^{-1}(q)$ is a divisor.  
Perhaps passing to a higher model we may
assume we have a factorisation $Y' \xrightarrow{g} X \xrightarrow{c} Y$.  Let $\cal G'\coloneqq\pi^{-1}_* \cal G$.
Lemma \ref{l_curvesinrayaretangent} implies that $c^{-1}(q)$ is tangent to $\cal F$.
Since $g^{-1}(c^{-1}(q))$ is a divisor and since $\cal F$ is non-dicritical
it follows by definition of tangency that $g^{-1}(c^{-1}(q))$ is
an invariant divisor, as required.

We now prove our second claim.  Let $c^+\colon X^+ \rightarrow Y$
be the induced morphism. We have that $c^+_*\cal F=\cal G$ and, by our first claim, it is non-dicritical. This  implies that $\cal F^+$ is non-dicritical, as claimed.
\end{proof}

Note that, in Section \ref{s_flip}, we prove the existence of flips for non-dicritical F-dlt pairs defined on a $\mathbb Q$-factorial projective theefold.

\section{Approximating formal divisors}\label{s_approx}

One of the main difficulties to prove the existence of flips for foliated pairs $(\cal F,\Delta)$, as in Theorem \ref{flipsexist}, is due to the fact that in the singular setting, some of the separatrices through a curve $C$ which is tangent to $\cal F$, are defined only in a formal neighbourhood of the curve $C$. To this end, since the MMP for formal schemes is still unknown, we study some application of Artin and Elkik's approximation theorems. 

We begin by recalling some definitions.
The following definition is \cite[Chapter XI, D\'{e}finition 2]{Raynaud70}.

\begin{defn}
Let $A$ be a ring, $\ff J \subset A$ an ideal and $B$ an \'etale $A$-algebra.
We say that $B$ is an {\bf \'etale neighborhood} of $\ff J$ in $A$
if the morphism
$$A/\ff JA \rightarrow B/\ff JB$$
is an isomorphism.
\end{defn}

The following definition is \cite[Chapter XI, D\'{e}finition 3]{Raynaud70}.

\begin{defn}
A pair $(A, \ff J)$ of a ring and an ideal is called a {\bf henselian couple}
if $\ff J$ is contained in the Jacobson radical of $A$ and for all \'etale neighborhoods $B$ of $\ff J$ in $A$,
there exists an $A$-morphism $B \rightarrow A$.
\end{defn}

Given a pair $(A, \ff J)$ of a ring $A$ and an ideal $\ff J$ contained in the Jacobson radical of $A$, it is possible to define the  {\bf henselization} of $(A,\ff J)$ 
as in \cite[Chapitre XI, Th\'{e}or\`{e}me 2]{Raynaud70}.

The next result is \cite[Theorem 3 (see also the paragraph below Theorem 3)]{Elkik73}: 

\begin{theorem}[Elkik approximation]
\label{elkik}
Let $(A, \ff J)$ be a henselian couple where $A$ is Noetherian and let $\widehat{A}$ be the $\ff J$-adic completion.
Let $\overline{M}$ be a finite type $\widehat{A}$-module, locally free on 
$\Spec {\widehat{A}}\setminus V(\ff J)$.  

Then there exists an $A$-module $M$
such that $M\otimes_A \widehat{A}$ is isomorphic to $\overline{M}$.

Furthermore, for any positive integer $k$, we may choose $M$
so that if 
\[
\widehat{A}^p \xrightarrow{\overline{L}} \widehat{A}^q \rightarrow \overline{M} \rightarrow 0
\]
is a presentation
of $\overline{M}$, then
we have a presentation of $M$
\[
A^p \xrightarrow{L} A^q \rightarrow M \rightarrow 0
\]
such that $L = \overline{L} \mod \ff J^k$
and such that the isomorphism between $\overline{M}$ and $M \otimes \widehat{A}$ is
induced by automorphisms of $\widehat{A}^p$ and $\widehat{A}^q$ congruent to the identity modulo 
$\ff J^k$.
\end{theorem}

\begin{remark}\label{r_reflexive}
Note that, in the Thoerem above, after possibly replacing $M$ by $M^{**}$, we may assume that $M$ is reflexive. Indeed, by \cite[Proposition 1.8]{Hartshorne80} it follows that 
$M^{**}\otimes_A \widehat A$ is reflexive and, therefore, it coincides with $\overline M$. 
\end{remark}

\begin{corollary}
\label{sepapprox}
Let $(A, \ff J)$ be a henselian couple where $A$ is Noetherian and let $\widehat{A}$ be the $\ff J$-adic completion.
Suppose moreover that $A$ is excellent and normal.

Fix any positive integer $k$.
Let $\overline{V} \subset \Spec {\widehat{A}}$ be an effective integral 
divisor and let $m$ be a positive integer such that $m\overline{V}$
is Cartier away from
$W = V(\ff J)$.

Then there exists a divisor $V$ on $\Spec A$ such that
\begin{enumerate}
\item $\cal O(mV) \otimes \widehat{A} \cong \cal O(m\overline{V})$

\item $V = \overline{V}$ mod $\ff J^k$.
\end{enumerate}
\end{corollary}

\begin{proof}
We first assume that $\overline M\coloneqq \cal O(\overline{V})$ is locally free away from $W$.
Note that this case follows from \cite[Corollaire pag. 574]{Elkik73}, but we include a proof for the reader's convenience. 

By Theorem \ref{elkik} and Remark \ref{r_reflexive}, there exists a reflexive sheaf 
$M$ on $\Spec A$ such that $M \otimes \widehat{A} = \cal O(\overline{V})$.
Let $\overline{s} \in \cal O(\overline{V})$ be a section 
whose associated divisor is $\overline{V}$.  We may assume that there is a presentation
\[
\widehat{A}^p \xrightarrow{\overline{L}} \widehat{A}^q \rightarrow \overline{M} \rightarrow 0
\]
such that $\overline{s}$ is the image of $(1, 0, ..., 0)$ in $\overline{M}$.
Approximating this presentation by 
\[
A^p \xrightarrow{L} A^q \rightarrow M \rightarrow 0
\]
we define $s$ to be the image of $(1, 0, ..., 0)$ in $M$.  In particular, we see that 
$s=\overline{s} \mod \ff J^k$ considered as sections of $\overline{M}$.
Let $V$ be the divisor associated
to $s$. Then $M = \cal O(V)$ and
$V = \overline{V}$ mod $\ff J^k$. 

Now, suppose that $m>1$ is the Cartier index of $\cal O(\overline{V})$ away from $W$.
By Lemma \ref{l_existence_of_cover} below, we may find a Galois possibly ramified cover $\sigma\colon \Spec B \rightarrow \Spec A$
such that $(\sigma^*\cal O(\overline{V}))^{**}$ is a line bundle away from $\sigma^{-1}(W)$.

Note that $\ff J' = B \otimes \ff J$ is the ideal corresponding to $\sigma^{-1}(W)$. 
Let $\widehat{B}$ be the $\ff J'$-adic
completion of $B$.  Observe that by \cite[Chapter XI, Proposition 2]{Raynaud70} $(B, \ff J')$ is a henselian couple.
We may find a reflexive sheaf $M'$
on $\Spec B$ such that $M' \otimes \widehat{B} \cong (\sigma^*\cal O(\overline{V}))^{**}$.
Let $\overline{s}$ be as above and let $\overline{t} = \sigma^*\overline{s}$.

As before, for any positive integer $\ell$, we can approximate $\overline{t}$ by a section $t$ of $M'$ such
that $t = \overline{t} \mod \ff J'^\ell$. 
Observe that for $\ell$ sufficiently large we have that $\sum_{g \in G} g\cdot t \neq 0$.
So replacing $t$ by $\frac{1}{\# G} \sum_{g \in G} g\cdot t$
we may assume that $t$ is also $G$-invariant.  
Let $V' $ be the divisor associated to $t$ and so $M'= \cal O(V')$.

By Theorem \ref{elkik} let $L$ be a reflexive sheaf on $\Spec A$ such that $L \otimes \widehat{A} = \cal O(m\overline{V})$.
Then we have $(\sigma^*L)^{**} \cong \cal O(m V')$.
Thus $(\sigma^*L)^{**}$ has a $G$-invariant section $t^{\otimes m}$ which 
descends to a section $\eta$ of $L$.

If we set $V = \sigma(V')$, then notice that $(\eta = 0) = mV$
and so $\cal O(mV) \otimes \widehat{A} \cong \cal O(m\overline{V})$.  
Next, we know that $mV = m\overline{V} \mod \ff J^\ell$ which implies that 
$V = \overline{V} \mod \ff J^{\lfloor \ell/m \rfloor}$.
Choosing $\ell \geq km$ gives our result.
\end{proof}

\begin{lemma}
\label{l_existence_of_cover}
Let $(A, \ff J)$ be a henselian couple where $A$ is Noetherian and let $\widehat{A}$ be the $\ff J$-adic completion.
Suppose moreover that $A$ is excellent.

Let $\overline{V} \subset \Spec {\widehat{A}}$ be a divisor such that $\cal O_{\Spec {\widehat{A}}}(\overline{V})$
is 
$\bb Q$-Cartier away from
$W = V(\ff J)$. 

Then there exists a finite Galois morphism $\sigma\colon \Spec B \rightarrow \Spec A$ such
that if $\widehat{\sigma} \colon \Spec {\widehat{B}} \rightarrow \Spec {\widehat{A}}$
is the completion of this map we have that $(\widehat{\sigma}^*\cal O(\overline{V}))^{**}$
is locally free away from $\widehat{\sigma}^{-1}(W)$.
\end{lemma}
\begin{proof}
Let $X=\Spec A$ and let  $m$ be the Cartier index of $\cal O(\overline{V})$ away from $W$. Let $L$ be a reflexive sheaf
on $X$ such that $L \otimes \widehat{A} \cong \cal O(m\overline{V})$ and whose existence is guaranteed by Theorem \ref{elkik}. 

Let $\{U_i\}_{i = 1, ...n}$ be an open cover of $\Spec A\setminus W$ and let $s_i$ be a global section of 
$L$ such that $s_i\vert_{U_i}$ generates
$L\vert_{U_i}$.
Choose a rational function $\varphi_i$ such that $(\varphi_i)= (s_i\vert_{U_i})$.
 Let $V_i \rightarrow U_i$ be a cover extracting a $m$-th root of $\varphi_i$.  If we let 
$\widehat{U_i} = U_i \times_X \Spec {\widehat{A}}$ and $\widehat{V_i} = V_i \times_X\Spec {\widehat{A}}$
then we see that $\cal O(\overline{V})\vert_{\widehat{V_i}}^{**}$ is locally free.

Let $V = V_1\times_XV_2\times_X ... \times_X V_n$ and let $K(V)$ be the field of functions of $V$ and notice
that $K(V)$ is a finite extension of $K(A)$.
Let $B$ be the integral closure of $A$ in $K(V)$. Note that $B$ is finite over $A$ (e.g. see \cite[Proposition 5.17]{AM69}).  Thus, the natural map  $\sigma \colon \Spec B \rightarrow \Spec A$ 
gives a cover with the desired properties.
\end{proof}

\begin{defn} Let $(A, \ff J)$ be a pair of a normal ring and an ideal. 
Let $\Delta'$ be a $\bb Q$-divisor on $\Spec {{A}}$.
We say that a $\bb Q$-divisor $\Delta$ on $\Spec {{A}}$ {\bf approximates} $\Delta' \mod \ff J^n$
if we can write $\Delta = \sum a_iD_i$ and $\Delta' = \sum a_iD'_i$ where
$D_i= D'_i \mod \ff J^n$.
\end{defn}

\begin{lemma}
\label{approxlc}
Let $(A, \ff J)$ be a pair of a ring and an ideal, 
where $A$ is excellent.
Suppose 
that $(\Spec {{A}}, \Delta)$ is a klt (resp. lc) pair.

Then there exists a positive integer $n_0$ such that if $n \geq n_0$ and 
if $\Delta'$ is an approximation of $\Delta \mod \ff J^n$ such that $K_{\Spec {{A}}}+\Delta'$ is $\bb Q$-Cartier
then $(\Spec {{A}}, \Delta')$ is klt (resp. lc).
\end{lemma}
\begin{proof}
Let $\pi\colon Y \rightarrow \Spec {{A}}$ be a log resolution
of $(\Spec {{A}}, \Delta)$ and let $E = \pi^{-1}(V(\ff J))$ whose existence
is guaranteed by \cite[Theorem 1.1.9 and Theorem 1.1.13]{temkin18}.
Note that $\pi$ is projective. Perhaps passing to a higher resolution we may assume that $E$ is a divisor.

Let $\widehat{D}$ be a component of $\Delta$. We may write 
\[
\pi^{-1}I_{\widehat{D}} \cdot \cal O_{Y} = \cal O_Y(-\widehat{D}'-\sum a_iE_i)
\]
where $\widehat{D}'$
is the strict transform of $\widehat{D}$ and $E_i$ are $\pi$-exceptional and $a_i \geq 0$.
Choose $r$ larger than $a_i$ for all $i$ and for all the components $\widehat{D}$ of $\Delta$.

Next, pick $n_0$ so that
$\ff J^{n_0} \subset \pi_*\cal O_Y(-2rE)$.
Note that in this case 
$\pi^{-1} \ff J^{n_0} \subset \cal O_Y(-2rE)$
and so if $s = t \mod \ff J^{n_0}$
then 
\[
\pi^*s = \pi^*t \mod \cal O_Y(-2rE).
\]

This choice of $r, n_0$ guarantees that if $\widehat{D}$ is a component of 
$\supp\Delta$ and $D$ is an approximation of $\widehat{D}$ mod $\ff J^n$ where $n \geq n_0$
that if we write
$$\pi^{-1}I_{\widehat{D}} \cdot \cal O_{Y} = \cal O_Y(-\widehat{D}'-\sum a_iE_i)$$ 
and
$$\pi^{-1}I_D \cdot \cal O_{Y} = \cal O_Y(-D'-\sum b_iE_i)$$ 
then $a_i = b_i$ and so $D' = \widehat{D}'$ in some infinitessimal neighborhood of $E$.
Thus, if $\Delta'$ is an approximation of 
$\Delta \mod \ff J^n$ for $n \geq n_0$
then $\pi$ is also a log resolution
of $(\Spec {{A}}, \Delta')$.

Furthermore, if we let $\pi_*^{-1}\Delta$ (respectively $\pi_*^{-1}\Delta')$ be the strict transform of $\Delta$ (respectively $\Delta'$)
we see that $\pi_*^{-1}\Delta\vert_E = \pi_*^{-1}\Delta'\vert_E$ which implies that
the two are $\pi$-numerically equivalent and hence the discrepancies
of $(\Spec {{A}}, \Delta)$ 
and $(\Spec {{A}}, \Delta')$ are the same
and the result follows.
\end{proof}



\section{Approximating formal separatrices}
\label{s_approx2}

In this section we work in the following set up:

Let $X$ be a $\bb Q$-factorial and klt quasi-projective threefold. Let $\cal F$ be a co-rank one foliation on $X$ with non-dicritical singularities and let $(\cal F, \Delta)$ be a F-dlt foliated pair.
Let $f\colon X \rightarrow Z$
be a birational contraction and
$p \in f(\exc f)$ be a closed point. We assume that $D \coloneqq f^{-1}(p)$ is tangent to $\cal F$.
Suppose moreover that $Z$ is klt away from finitely many points.  As we will see, this will always be the case for birational contractions
arising in the course of the MMP.

Let $\hat{f}\colon  \widehat{X} \rightarrow \widehat{Z}$ be the completion of $f$ along the fibre
$f^{-1}(p)$, and let $\widehat{\cal F}$ be the formal foliation.

\begin{remark} \label{r_sep}
Observe that by Lemma \ref{basicpropertyfdlt} and Lemma \ref{l_fdlt-terminal} 
if $D$ is a curve and $D_i$ is component of $D$ then 
$D_i$ is contained in a (possibly formal) $\widehat{\cal F}$-invariant divisor.
\end{remark}

\begin{lemma}\label{l_approx_etale}
Let $\widehat{S} \subset \widehat{X}$ be any irreducible formal divisor not contained in $\exc f$.
Fix an integer $n>0$.

Then there exists an \'etale morphism $\sigma\colon Z' \rightarrow Z$
and a divisor $S'$ on $X' = X \times_Z Z'$
such that if $\tau\colon X' \rightarrow X$ is the projection, $f'\colon X' \rightarrow Z'$ the induced morphism,
$D' = f'^{-1}(\sigma^{-1}(p))$ and $\widehat{\tau}\colon \widehat {X'}\to \widehat X$ is the completion of $\tau$ along $D'$, 
then $\widehat{S'} = \widehat{\tau}^*\widehat{S}$ mod $I_{D'}^n$, where $\widehat{S'}$ is the restriction of $S'$ to $\widehat{X'}$.
\end{lemma}
\begin{proof} We may assume that $Z=\Spec B$ is affine. 
Let $(A, \ff m)$ be the henselisation of $B$ at $p$, and let 
$\widehat{A}$ be the formal completion of $B$ at $p$.
Let $\widetilde{X} = X \times_{\Spec B} \Spec {\widehat{A}}$.
By the Grothendieck existence theorem there exists a divisor $\widetilde{S}$
on $\widetilde{X}$ such that $\widetilde{S}\vert_{\widehat{X}} = \widehat{S}$.
%

Let $\widetilde f\colon \widetilde X\to \Spec {\widehat A}$ be the induced morphism. By the proper mapping theorem, we have that 
$\overline{V}:= \widetilde f_*\widetilde{S}$ is a divisor on $\Spec {\widehat{A}}$.
We claim that $\overline{V}$ is a $\bb Q$-Cartier divisor on $\Spec {\widehat A} \setminus p$.
Indeed, by assumption $Z$ is klt away from finitely many points.  By \cite[Proposition 9.4]{GKKP11} a
klt variety has quotient singularities in codimension two, and so it follows that, away from finitely many points, $Z$
has quotient singularities.  In particular,  it follows that $\Spec {\widehat{A}}\setminus \{p\}$
is $\bb Q$-factorial and so $\overline{V}$ is $\bb Q$-Cartier.

Let $\widetilde D=\widetilde f^{-1}(p)$. 
Pick a positive integer $k$ large enough so that $\ff m^k\subset \widetilde{f}_*I_{\widetilde D}^{n+n'}$
where $n'$ is a sufficiently large positive integer so that $\cal O(-\overline{V})$ is not contained in $\widetilde f_*I_{\widetilde{D}}^{n'}$.
By Corollary \ref{sepapprox}, there exists a divisor $V$ on $\Spec A$ 
which agrees with $\overline{V}$ mod $\ff m^k$.
Recall
\[
A = \varinjlim_{(\Spec {B'}, q) \rightarrow (\Spec B, p)} B'
\] 
where we run over \'etale morphisms $(\Spec {B'}, q) \rightarrow (\Spec B, p)$
sending $q$ to $p$.
Thus, we see that there exists some \'etale cover $\Spec {B'} \rightarrow \Spec B$ and a
divisor $V'$ on $\Spec {B'}$ which agrees with $V$ when pulled back to $\Spec A$.

Let $S'$ be the strict transform of $V'$ on $X'$ 
and let $\widehat{S'}$ be the restriction of $S'$ to $\widehat{X'}$. Then
we have that $\widehat{S'} = \hat{\tau}^*\widehat{S} \mod I_{D'}^n$, as required.
\end{proof}

\begin{lemma}
Notation as above.
Let $\widehat{S}_1, ..., \widehat{S}_k$ be any collection of distinct irreducible
$\widehat{\cal F}$-invariant divisors and suppose that $\widehat{S}_1, ..., \widehat{S}_k$ are $\bb Q$-Cartier
and such that $\widehat{S}_i$ is not contained in $\exc f$ for any $i$.
Let $X'$ and $S'_1, ..., S'_k$ be as in Lemma \ref{l_approx_etale}, where $S'_i$ is an approximation of  
$\widehat{S}_i$ mod $I_{D'}^n$.  

Then for $n$ large enough and perhaps shrinking $X'$ to a smaller neighborhood of $D'$ we have $(X', \sum S'_i)$ is log canonical.
\end{lemma}
\begin{proof}
We first claim that $S'_i$ is $\bb Q$-Cartier. Indeed, by construction we have an isomorphism
$\hat{f}_*\cal O_{\widehat{X}}(\widehat{S}_i) \cong \hat{f}_*\cal O_{\widehat{X}}(S'_i)$
of reflexive sheaves on $\widehat{Z}$.  Since $\cal O_{\widehat{X}}(\widehat{S}_i)$ and 
$\cal O_{\widehat{X}}(S'_i)$ are reflexive and $\widehat{S}_i$ and $S'_i$ do not contain
any $f$-exceptional divisors in their support it follows that this isomorphism on $Z$
gives an isomorphism $\cal O_{\widehat{X}}(\widehat{S}_i) \cong \cal O_{\widehat{X}}(S'_i)$ on $\widehat{X}$.  
In particular, $S'_i$ is $\bb Q$-Cartier
because $\widehat{S}_i$ is.

The result then follows by combining
Lemma \ref{approxlc} and Lemma \ref{singcomparison}.
\end{proof}


\section{Constructing the flip}
\label{s_flip}

\subsection{Set up}
\label{s_flip_construction}
Through out this section, we assume that
$X$ is a $\bb Q$-factorial klt quasi-projective threefold,
$\cal F$ is a co-rank one foliation on $X$ with non-dicritical singularities 
and
$(\cal F, \Delta)$ is F-dlt foliated pair. 

\medskip

The goal of this section is to show that 
if $f\colon  X \rightarrow Z$ is a   flipping contraction associated to  a
$(K_{\cal F}+\Delta)$-negative
extremal ray $R$, then the $(K_{\cal F}+\Delta)$-flip exists (cf. Section \ref{s_ample_model}).
The basic idea is to reduce the $(K_{\cal F}+\Delta)$-flip to a $(K_X+\widetilde{\Delta})$-flip
for some klt pair $(X, \widetilde{\Delta})$. 

Recall that in \cite[Lemma 8.21]{Spicer17} it was proven that the flipping contraction
exists in the category of algebraic spaces.  Observe that the proof given there only requires
that $(\cal F, \Delta)$ has non-dicritical and log canonical singularities and so works for F-dlt foliated pairs
with non-dicritical singularities.

Our goal here is to show that 
\begin{enumerate}
\item[(a)] if $X$ is projective then also $Z$ is projective and $\rho(X/Y)=1$. 
\item[(b)] the flip exists. 
\end{enumerate}
By Remark \ref{r_fdlt=lc} and Lemma \ref{l_curvesinrayaretangent}, it follows that  $\exc f$ is tangent to the foliation.

\medskip 

The following  result will be needed to address (a) and (b) above:

\begin{lemma}\label{l_rounddown}
 Without loss of generality, we may assume that
$\lfloor \Delta \rfloor = 0$.
\end{lemma}
\begin{proof}
The $(K_{\cal F}+\Delta)$-flip is the $(K_{\cal F}+(1-\epsilon)\Delta)$-flip
for $\epsilon>0$ sufficiently small.
\end{proof}

The following  result will be needed to address  (b) above:

\begin{lemma}
\label{etaleflip}
Let $f\colon  X \rightarrow Y$ be a small contraction between algebraic spaces.
Let $D$ be a $\bb Q$-Cartier divisor on $X$ such that $-D$ is $f$-ample.  Let $\{U_i \rightarrow Y\}_{i\in I}$ be an \'{e}tale
cover of $Y$.
Suppose that, for each $X_i = X \times_Y U_i \rightarrow U_i$, the $D|_{X_i}$-flip exists.

Then the $D$-flip exists.
\end{lemma}
\begin{proof}
The existence of the flip is equivalent to the $\cal O_Y$-algebra
$$\bigoplus_{m \geq 0} f_*\cal O_X(mD)$$ being finitely generated, where $m$ is taken to be sufficiently divisble.
Finite generation of an algebra can be checked \'etale locally and the result follows.
\end{proof}

\medskip 

We briefly sketch the remainder of the goals for this subsection.  First, we  show that, after possibly replacing $X$ by an appropriate small modification, 
 we may find a formal divisor $\widehat{F}$ on $\widehat{X}$, the formal completion of $X$ along $\exc f$,
such that $K_{\widehat{X}}+\widehat{F}$ is numerically equivalent to $K_{\widehat{\cal F}}$ and such that $(\widehat{X}, \widehat{F})$
has log canonical singularities.  
Indeed, we will show that if $\widehat{S}$ is the sum of all the $\cal F$-invariant divisors
meeting $\exc f$ then $K_{\widehat{\cal F}}-(K_{\widehat{X}}+\widehat{S})$ is $f$-nef.  Since $\exc f$ is a curve, 
it is then easy to find an effective divisor $D$ on some analytic neighborhood of $\exc f$ such that $D$ restricted to $\widehat{X}$ is numerically equivalent
to $K_{\widehat{\cal F}}-(K_{\widehat{X}}+\widehat{S})$.  Then $\widehat{F} = \widehat{S}+D$ will be our desired divisor.
Second, we will apply our version of Elkik approximation theorem to show 
that we may approximate $\widehat{F}$ on some \'etale neighborhood $X' \rightarrow X$ of $\exc f$
by a divisor $F'$.  We will then, in the next subsection, construct the foliated flip as a $(K_{X'}+F')$-flip.

\medskip

We now proceed with the construction outlined above. Let $S_1,\dots,S_{\ell}$ be 
the collection of all the $\cal F$-invariant divisors (formal or otherwise) on $X$ meeting at least one of the 
curves contracted by $f$, whose existence is guaranteed 
by Remark \ref{r_sep}.  We remark that by Lemma \ref{l_formalseparatrix} we may extend 
each $S_i$ to a formal divisor on $\widehat{X}$.
We emphasise that each $f$-exceptional curve is contained in one such divisor.
By Lemma \ref{l_approx_etale}, we may find a diagram

\begin{center}
\begin{tikzcd}
X' \arrow{d}{f'} \arrow{r}{\tau} & X \arrow{d}{f} \\
Z' \arrow{r}{\sigma}& Z \\
\end{tikzcd}
\end{center}
where $\sigma\colon Z' \rightarrow Z$ is \'etale
and divisors $S'_k$ on $X'$ which 
approximate the $S_k$ to some arbitrarily high (but fixed) order. Let $\mathcal F'$ be the transformed foliation on $X'$. 

Note that $X'$ is klt since $X$ is. Let $g\colon Y \rightarrow X'$ be a small 
$\bb Q$-factorialisation of $X'$ and let $\cal F_Y$ the pulled back foliation on $Y$.  Then $g^*K_{\cal F'} = K_{\cal F_Y}$ and
$g^*K_{X'} =K_Y$.
Let $h = f'\circ g\colon Y \rightarrow Z'$ be the composition and write
$\exc h = C = \cup C_i$.

Let $U$ be a small analytic
neighbourhood around $C$. We claim that, for each $i$, we can find $\bb Q$-divisors $D_{i1}, \dots, D_{im_i}\ge 0$ on $U$ 
such that $D_{ik}\cdot C_j = \delta_{ij}$ for each $k=1,\dots,m_i$, where $m_i \geq 0$ is an integer whose precise value will be determined later.
Indeed, let $A$ be a general sufficiently ample Cartier divisor on $Y$ so that
$A \cap C = \{P_{ij}\}$ is a finite collection of points where $P_{ij} \in C_i$ and $P_{ij} \notin C_k$ for $k \neq i$.
Then, on a sufficiently small analytic neighborhood $U$ of $C$ 
we may write $A\cap U = \sum A_{ij}$ where $A_{ij}\cap C = P_{ij}$
and $A_{ij} \cap A_{i'j'} = \emptyset$ for $(i, j) \neq (i', j')$.
Taking $D_{i1}, ..., D_{im_i}$ to be some of the $A_{ij}$ as $A$ varies across ample divisors will then have our desired properties

We remark that if we write $g^*(K_{\cal F'}+\Delta') = K_{\cal F_Y}+\Delta_Y$
then by Lemma \ref{lem_fdlt_small_mod}, we have that $(\cal F_Y, \Delta_Y)$ is F-dlt.
Moreover, since $\cal F'$ has non-dicritical singularities it follows
that each $g$-exceptional curve is tangent to $\cal F_Y$.

Let $\widetilde S_k$ be the strict transform of $S'_k$ on $Y$ and note that $\widetilde{S}_k$ is a $\bb Q$-Cartier divisor.
Let $\widehat{Y}$ be the formal completion of $Y$ along $\exc h$, and notice that we have a morphism
$\hat{g}\colon \widehat{Y} \rightarrow \widehat{X}$. Let $T_k$ be the strict transform of $S_k$ on $\widehat{Y}$.
We claim that $T_k$ is a $\bb Q$-Cartier divisor. Observe that we have an isomorphism of sheaves
$\cal O_{\widehat{Y}}(T_k) \cong
\cal O_{\widehat{Y}}(\widetilde{S}_k)$:  indeed, by item (1) of Corollary \ref{sepapprox} we know that
if $\hat{h}$ is the restriction of $h$ to $\widehat{Y}$ then
$\hat{h}_*\cal O_{\widehat{Y}}(T_k) \cong
\hat{h}_*\cal O_{\widehat{Y}}(\widetilde{S}_k)$ and since $\exc h$ is of codimension $\geq 2$ this gives an
isomorphism of reflexive sheaves on $\widehat{Y}$.
Since $\widetilde{S}_k$ is $\bb Q$-Cartier this implies that $T_k$ is $\bb Q$-Cartier as well.


For any $T_j$ which is convergent and contains one of the $C_i$ by Corollary \ref{basicpropertyfdlt2}  we may write
$(K_{\widehat{Y}}+\sum T_k)\vert_{T_{j}}= K_{T_{j}}+\Theta_{j}$
and $K_{\cal F_Y}\vert_{T_{j}} = K_{T_{j}} +\Delta_{j}$ 
where $\Theta_{j}\geq \Delta_{j}$.

Now fix $i$ and consider $C_i$.  We consider two possibilities as in Lemma \ref{l_fdlt-terminal}.
First, suppose that $\cal F_Y$ is terminal at the generic point of $C_i$.
In this case there exists a single $k_i$ such that $C_i \subset T_{k_i}$ and moreover $T_{k_i}$ is convergent.
Again, by Corollary \ref{basicpropertyfdlt2} we know that $\Theta_{k_i}$ and
$\Delta_{k_i}$ have the same coefficient along $C_i$.

Now suppose that $\cal F_Y$ is canonical (but not terminal) at the generic point of $C_i$.
In particular, Lemma \ref{l_fdlt-logsmooth} implies that $\cal F_Y$ has simple singularities at a general point of $C_i$.
Let $q \in C_i$ be a general point and let $S$ be the strong separatrix of $\cal F_Y$ at $q$, recall that $S$ is convergent.
Then by Lemma \ref{l_formalseparatrix} this separatrix may be extended to a convergent divisor containing $C_i$,
which is therefore one of the $T_k$.  Let $T_{k_i}$ be this divisor.
Again, we may apply Corollary \ref{basicpropertyfdlt2} to see that $\Theta_{k_i}$ and
$\Delta_{k_i}$ have the same coefficient along $C_i$.

It follows that
\[
(K_{\widehat{Y}}+\sum T_k)\cdot C_i \leq K_{\cal F_Y}\cdot C_i,
\]
and so 
\[
(K_Y+\sum\widetilde{S}_k)\cdot C_i \leq K_{\cal F_Y}\cdot C_i.
\]

By Lemma \ref{l_approx_etale}, we may find a diagram

\begin{center}
\begin{tikzcd}
Y' \arrow{d}{h'} \arrow{r}{\tau'} & Y \arrow{d}{h} \\
Z'' \arrow{r}{\sigma'} & Z' \\
\end{tikzcd}
\end{center}
where $\sigma'\colon Z'' \rightarrow Z'$ is \'etale
and divisors $D'_{ik}$ on $Y'$ such that $D'_{ik}$
approximate  $D_{ik}$ to some arbitrarily high (but fixed) order.  Notice that, as in the proof of the fact that $T_k$ is $\bb Q$-Cartier, we have that 
 $D'_{ik}$ is  $\bb Q$-Cartier.
Observe also that if $C_i' \subset Y'$ is an $h'$-exceptional curve such that $\tau'(C_i') = C_i$ then we
still have $C_i'\cdot D'_{jk} = \delta_{ij}$.

Notice that if $\widetilde S_k'$ is a surface on $Y'$ such that $\tau'(\widetilde S'_k)=\widetilde S_k$,
 then we still have the inequality
\[
(K_{Y'}+\sum \widetilde S'_k)\cdot C' \leq K_{\cal F_{Y'}}\cdot C'
\]
for any $h'$-exceptional curve $C'$.
Thus, if we take 
\[
a_i = (K_{\cal F_{Y'}} - (K_{Y'}+\sum \widetilde S'_k))\cdot C'_i \geq 0
\]
 where $C'_i$ is a curve such that $\tau'(C'_i) = C_i$,  then we have 
that 
\[
K_{Y'}+\sum\widetilde S'_k+\sum \frac{a_i}{m_i}D'_{ik}\equiv_{h'}K_{\cal F_{Y'}}.
\]

Let $\Delta' = \tau'^*g^*\tau^*\Delta$.
Notice that $(\cal F_{Y'}, \Delta')$ is F-dlt and the $\widetilde S'_k$ are smooth in codimension one, and so 
Lemma \ref{singcomparison} and Lemma \ref{approxlc} imply that (perhaps replacing $Z''$ be an open neighborhood of $h'((\tau')^{-1}(C))$)
$(Y', \Delta'+\sum S'_k)$ is lc and $(Y', \Delta'+(1-\epsilon)\sum S'_k)$ is klt for $0<\epsilon\le 1$. So
by choosing the $D_{ik}$ general enough and $m_i$ large enough so that $a_i \leq m_i$
we may assume that
if $A = \sum \widetilde S'_k +\sum \frac{a_i}{m_i}D'_{ik}$,
then $(Y', \Delta'+A)$ is lc and $(Y', \Delta+(1-\epsilon)A)$ is klt.

Since $X \rightarrow Z$ contracts only a single extremal ray
and $-(K_{\cal F}+\Delta)$ is relatively ample, there exists  $\lambda \in \bb Q$ such that 
$\lambda (K_{\cal F}+\Delta)\equiv_f (K_X+\Delta)$.
Since $g$ is small and $\tau$ and $\tau'$ are \'etale, it follows that 
$\lambda (K_{\cal F_{Y'}}+\Delta') \equiv_{h'} (K_{Y'}+\Delta')$, where $\cal F_{Y'}$ is the transformed foliation on $Y'$.
Since $A \equiv_{h'} (K_{\cal F_{Y'}}+\Delta') - (K_{Y'}+\Delta')$, we have
\[
A \equiv_{h'} \mu (K_{\cal F_{Y'}}+\Delta')
\] 
for $\mu = 1-\lambda$.

\subsection{Existence of the flip}
Below, we use the same notation as in the previous subsection.

\begin{lemma}
\label{baseofflipisprojective}
Set up as above.

Then the following hold:
\begin{enumerate}
\item$\rho(X/Z) = 1$; and

\item if, in addition,  $X$ is projective,  then also $Z$ is projective.
\end{enumerate}
\end{lemma}
\begin{proof} 
Let $D$ be a $\bb Q$-Cartier divisor on $X$ which is $f$-numerically trivial.
In order to show that $\rho(X/Z)=1$, it is enough to show that there exists a $\bb Q$-Cartier divisor $M$ on $Z$ such that $f^*M = D$. Indeed, since $f\colon X\to Z$ is a flipping contraction associated to an extremal ray $R$, the claim implies that  the sequence
\[
0\longrightarrow \text{Pic}(Z)\otimes \bb Q\stackrel{f^*}{\longrightarrow}\text{Pic}(X)\otimes \bb Q\longrightarrow \bb Q\longrightarrow 0
\]
is exact, where the second to last arrow is given by $D \mapsto D\cdot \xi$ where $\xi$ is a fixed curve such that $[\xi] \in R$.

Let $D' = \tau'^*g^*\tau^*D$.
Since the descent problem above is \'etale local,
it suffices to show that $D' = (h')^*M'$ for some $\bb Q$-Cartier divisor on $Z'$.
The existence of $M'$ follows by applying the classical relative base point free theorem
to the pair $(Y', \Delta'+(1-\epsilon)A)$, for some $0 < \epsilon \ll 1$. Indeed, recall from Section \ref{s_flip_construction} we have that $(Y', \Delta'+(1-\epsilon)A)$ is klt.
Thus, since 
\[
D' - (K_{Y'}+\Delta'+(1-\epsilon)A) \equiv_{h'} -(1-\epsilon\mu)(K_{\cal F_{Y'}}+\Delta')
\]
is $h'$-big and nef for small $\epsilon$, we have that $D'$ is $h'$-semi-ample.  
Thus, by definition, there is
some $n \gg 0$ and a Cartier divisor $L$ on $Z'$ such that $(h')^*L = nD'$.  
Thus, we may choose $M' = \frac{1}{n}L$ and (1) follows. 

Assume now that $X$ is projective and that 
$D=H_R$ is a nef $\mathbb Q$-Cartier divisor on $X$ which defines a supporting hyperplane for $R$ in $\overline {NE}(X)$ and let $M$ be the induced $\bb Q$-Cartier divisor on $Z$.  The existence of $H_R$ is a direct consequence of the Cone theorem, \cite{Spicer17} or Theorem \ref{t_cone-dlt}. 
Projectivity then follows by  noting that, for any subvariety $V$ of $Z$ we have $M^{\dim V}\cdot V >0$. Indeed, $M$ is ample by the Nakai-Moishezon criterion for
algebraic spaces,
\cite[Theorem 3.11]{Kollar90},
and so $Z$ is projective. Thus, (2) follows.
\end{proof}


\begin{theorem}
\label{flipsexist}
Let $X$ be a  $\bb Q$-factorial projective threefold and let $\cal F$ be a co-rank one foliation
on $X$ with non-dicritical singularities.
Suppose that $(\cal F, \Delta)$ is a F-dlt foliated pair.
 Let $f\colon X\to Z$ be a $(K_{\cal F}+\Delta)$-negative flipping contraction.

Then the  $(K_{\cal F}+\Delta)$-flip $X\dashrightarrow X^+$ exists.  
Moreover,
\begin{enumerate}
\item $X^+$ is projective and $\bb Q$-factorial,

\item $(\cal F^+, \Delta^+)$ has F-dlt singularities,

\item If, in addition, the foliated pair $(\cal F, \Delta)$ admits terminal (resp. canonical, resp. log terminal, resp. log canonical) singularities then so does $(\cal F^+, \Delta^+)$,

\item $\cal F^+$ has non-dicritical singularities and

\item $X^+$ is klt.
\end{enumerate}
\end{theorem}
\begin{proof}
By Lemma \ref{etaleflip} it suffices to construct the flip \'etale locally on the base.
Thus, taking $Z''$ as in Section \ref{s_flip_construction}, we see that to construct the flip it suffices to produce
an ample model for $K_{\cal F_{Y'}}+\Delta'$ over $Z''$, see Section \ref{s_ample_model}.

However, we know that for $\epsilon>0$ sufficiently small $K_{Y'}+\Delta'+A$,
$K_{Y'}+\Delta+(1-\epsilon)A$ and $K_{\cal F_{Y'}}+\Delta'$ all have the same ample
model over $Z''$. As above, we have that $(Y', \Delta'+(1-\epsilon)A)$ is klt
and so the ample model over $Z''$ exists by \cite[Theorem 1.2]{BCHM06}.  Call this model
\[
c\colon Y'\dashrightarrow Y'^+.
\]

Since $h'\colon Y' \rightarrow Z''$ is small, we know that $c$ is small and
so $(Y'^+, c_*(\Delta'+(1-\epsilon)A))$ is in fact klt.
Projectivity and $\bb Q$-factoriality follow easily. 

Our claims on the singularities of $(\cal F^+, \Delta^+)$ are a direct consequence of Lemma \ref{l_negativity}.

Non-dicriticality of $\cal F^+$ follows from the fact that $X$ is klt by Lemma \ref{singcomparison} and Lemma \ref{l_image_cont_non-dicrti}.
\end{proof}

\begin{remark}
If one is so inclined this can all be done in the analytic topology around the
flipping curves.  The relevant classical log MMP is known to exist by \cite{Nakayama87}.
\end{remark}

\begin{remark}\label{r_qpflip}
Since the construction of the flip is local on the base, if $f\colon X \rightarrow Z$ is a flipping contraction between
quasi-projective varieties then the arguments above show that the flip $\phi\colon X \dashrightarrow X^+$
still exists and enjoys all the properties of the flip listed in Theorem
\ref{flipsexist}, with the exception of projectivity of $X^+$.
\end{remark}

\subsection{$(K_{\cal F}+\Delta)$-negative contractions are extremal}

\begin{theorem}
\label{contractionsareextremal}
Let $X$ be a normal $\bb Q$-factorial projective threefold and 
let $\mathcal F$ be a co-rank one foliation on   $X$ with non-dicritical singularities. 
Let  $(\mathcal F,\Delta)$ be a F-dlt pair and let 
$R$ be a $(K_{\cal F}+\Delta)$-negative extremal ray.

Then the contraction associated to $R$ 
\[
\phi_R\colon X \rightarrow Y
\]
exists in the category
of projective varieties and $\rho(X/Y) = 1$.  In particular, $\phi_R$ is extremal. Moreover, if $\phi_R$ is a divisorial contraction then $Y$ is $\mathbb Q$-factorial and klt. 
\end{theorem}
\begin{proof}
First, observe that, by Lemma \ref{singcomparison}, $X$ is klt.

If $\loc R = X$, then, as in \cite[Theorem 8.13]{Spicer17}, it follows that $R$ is in fact $K_X$-negative and so the contraction exists.

If $\loc R = D$ a divisor and $D$ is transverse to the foliation, then as in \cite[Lemma 8.15]{Spicer17}, it follows that
$R$ is $(K_X+\Delta)$-negative.
If $D$ is  invariant then we claim that 
that $R$ is $(K_X+\Delta+D)$-negative.  
Indeed, $D$ is covered by curves $C_t$ which span $R$ and if  $D^{\nu}\to D$ is the normalisation of $D$ and
 we write 
 \[
 (K_{\cal F}+\Delta)\vert_{D^{\nu}}= K_{D^{\nu}}+\Theta \qquad\text{and}\qquad  (K_X+\Delta+D)\vert_{D^\nu} = K_{D^{\nu}}+\Theta'
 \]
then Corollary \ref{basicpropertyfdlt2} implies that
$(\Theta-\Theta')\cdot C_t \geq 0$ and so since $(K_{D^{\nu}}+\Theta)\cdot C_t<0$ we see that $(K_{D^{\nu}}+\Theta')\cdot C_t<0$
as required.
In either case, the contraction exists by 
Lemma \ref{singcomparison} and the existence of $(K_X+D)$-negative divisorial contractions.

If $\loc R = C$ is a curve then the result follows by Lemma \ref{baseofflipisprojective}.
\end{proof}

\section{Special termination}
The goal of this section is to show 
the following:

\begin{theorem}[Special Termination]\label{t_spter}
Let $X$ be a $\bb Q$-factorial quasi-projective threefold and let $\cal F$ be a co-rank one foliation with non-dicritical singularities on $X$.
Let $(\cal F, \Delta)$ be an F-dlt pair.  
Let 
\[
X=X_0\dashrightarrow X_1\dashrightarrow X_2\dashrightarrow \dots
\]
be an infinite sequence of $(K_{\cal F}+\Delta)$-flips and let $(\cal F_i,\Delta_i)$ be the transformed foliated pair on $X_i$.

Then after finitely many flips, the flipping and flipped locus are disjoint from any lc centres
of $(\cal F_i, \Delta_i)$.
\end{theorem}

Note that the result and the some of the proofs below were inspired by Shokurov's special termination in the classical setting \cite{Shokurov03} (see also \cite[\S 4.2]{Corti05}).
 
We begin with the following:
\begin{lemma}
Let $X$ be a $\bb Q$-factorial quasi-projective threefold and let $\cal F$ be a co-rank one foliation with non-dicritical singularities on $X$.
Let $(\cal F, \Delta)$ be an F-dlt pair.  
Suppose that there exist infinitely many $\cal F$-invariant divisors. 

Then any sequence of $(K_{\cal F}+\Delta)$-flips terminates. 
\end{lemma}
\begin{proof}
Since the intersection of two invariant divisors is contained in $\sing \cal F$ and since $\mathcal F$ has non-dicritical singularities, it follows that there exist infinitely many pairwise disjoint $\cal F$-invariant divisors. 
By
\cite[Theorem 2]{Pereira06}, 
 there exists a morphism $f\colon X\to C$ onto a smooth curve $C$ such that $\mathcal F$ is induced by $f$ and, in particular, 
\[
K_{\mathcal F}=K_{X/C}+\sum (1-\ell_D)D
\]
where the sum is taken over all the vertical irreducible divisors and $\ell_D$ denotes the multiplicity of the fibre $f^{-1}(f(D))$ along $D$. 
Thus 
\[
K_{\mathcal F}\sim_{\bb Q,f} K_X + \Gamma
\]
where $\Gamma$ is the sum of all the vertical prime divisors which are contained in a non-reduced fibre. 
Since $(\mathcal F,\Delta)$ is F-dlt and since any component of $\Gamma$ is $\mathcal F$-invariant, Lemma \ref{singcomparison} implies that $(X,\Delta+\Gamma)$ is log canonical.

Note that if $X\dashrightarrow X'$ is a $(K_{\cal F}+\Delta)$-flip and the flipping curve $\xi$ is horizontal, then $F\cdot \xi>0$ for any general fibre $F$ of $f$ and, in particular, the strict transform $F'$ of $F$ on $X'$ contains the flipped curve $\xi'$, contradicting the fact that the transformed foliation $\cal F'$ on $X'$ has non-dicritical singularities
(cf. Lemma \ref{l_image_cont_non-dicrti}). Thus, we may assume that any sequence of $(K_{\cal F}+\Delta)$-flips preserves the fibration onto the curve $C$. 
 In particular, any sequence of $(K_{\mathcal F}+\Delta)$-flips is also a sequence of $(K_X+\Delta+\Gamma)$-flips. Thus, termination follows from termination of three-dimensional log canonical flips (cf. \cite[Theorem 2.3]{Shokurov03}). 
\end{proof}

Thus, from now on, we assume that $\mathcal F$ admits at most finitely many invariant divisors.  The proof proceeds in two steps.  We first consider the case of lc
centres transverse to the foliation.  We then handle the case of lc centres
tangent to the foliation by induction on dimension:
supposing the statement is true for all 
$(d-1)$-dimensional lc centres, we then prove it for all $d$-dimensional lc centres.

%
%
%
%

\subsection{Log canonical centres transverse to the foliation}

\begin{proposition}
\label{p_disjointtranscentre} Let $X$ be a $\bb Q$-factorial quasi-projective threefold 
and let $\cal F$ be a co-rank one foliation with non-dicritical singularities  on $X$.
Suppose that $(\cal F, \Delta)$ is F-dlt and let $W \subset X$ be an lc centre transverse
to $\cal F$. Let $\phi\colon X\dashrightarrow X^+$ be a flip and let
 $\Delta^+ = \phi_*\Delta$.  Suppose that $W$ is not contained in the flipping locus $Z$.
Let $W^+$ be the strict transform of $W$ and   
$Z^+$ be the flipped locus.

 Then $W^+ \cap Z^+ \subset W^+$ is not a divisor.

\end{proposition}
\begin{proof}
Suppose for sake of contradiction that $W^+ \cap Z^+ =: D \subset W^+$ is a divisor.

By Remark \ref{r_finmanylccenters} it follows that 
$W$ is contained in the support of $\lfloor \Delta \rfloor$, and in fact is a stratum of $\lfloor \Delta \rfloor$. 
Hence $W^+$ is  a stratum of $\lfloor \Delta^+\rfloor$. 

 Let $\cal F^+$ be the transformed foliation on $X^+$ and let $\cal G^+$ be the foliation restricted to $W^{+, \nu}$ where $\nu\colon W^{+, \nu} \rightarrow W^+$ is the 
normalisation. 
Thus, since $W^+$ is a stratum of $\lfloor \Delta^+\rfloor$,  we may apply   Lemma \ref{adjunction} and by induction on the codimension of $W^+$, we may write 
\[
(K_{\cal F^+}+\Delta^+)\vert_{W^{+, \nu}} = K_{\cal G^+}+\Theta^+
\]
where $(\cal G^+, \Theta^+)$ is log canonical.
Likewise we may write $(K_{\cal F}+\Delta)\vert_{W^\nu} = K_{\cal G}+\Theta$.

We now claim that $D$ is $\cal G^+$ invariant.
 
  If $\text{dim}(W) = 2$ then $D$ is a curve and it must be tangent to $\cal F^+$ by Lemma \ref{l_image_cont_non-dicrti}, and hence it is invariant
 by $\cal G^+$.  Indeed, by Lemma \ref{l_negativity} we know that $\cal F^+$ is terminal at the generic point of $D$ and so by Lemma \ref{l_fdlt-terminal}
 there exists a unique germ $S$ of an invariant surface containing $D$.  It follows that $\nu^{-1}(S \cap W^+)$ is $\cal G^+$-invariant and so
 $D$ is also $\cal G^+$-invariant.  
 
 If $\text{dim}(W) = 1$ then $D$ is a finite collection of points and $\cal G^+$ is the foliation by points on $W^+$ and so $D$ is automatically
 $\cal G^+$-invariant.

 By Lemma \ref{l_negativity}, applied to the map $W^\nu\dashrightarrow W^{+,\nu}$, we know that $a(D, \cal G, \Theta) < a(D, \cal G^+, \Theta^+)$.
 However, since $(\cal G^+, \Theta^+)$ is log canonical and since $D$ is invariant it follows that $a(D, \cal G^+, \Theta^+) = 0$, i.e., the coefficient
 of $D$ in $\Theta^+$ is zero,  see Remark \ref{remark_inv_not_lc}.
 On the other hand, since $(\cal G, \Theta)$ is log canonical, we have $a(D, \cal G, \Theta) \geq -\epsilon(D) = 0$.
 This, however, is an impossibility.
\end{proof}

\begin{corollary}
\label{c_termdisjointtranscentre}
Let $X$ be a $\bb Q$-factorial quasi-projective threefold and let $\cal F$ be a co-rank one foliation with non-dicritical singularities on $X$.
Suppose that $(\cal F, \Delta)$ is F-dlt. 

Then, after finitely many flips, the flipping locus
is disjoint from all the lc centres transverse to the foliation.  
\end{corollary}
\begin{proof} Remark \ref{r_finmanylccenters} implies that there are only
finitely many lc centres transverse to $\cal F$ and these are strata of $\lfloor \Delta \rfloor$. Thus, by Lemma \ref{l_negativity}, we may assume that  no lc centre transverse
to the foliation is contained in the flipping locus.  
Let $\phi\colon X\dashrightarrow X^+$ be a flip.
Suppose that the flipping locus meets some lc centre of $(\cal F, \Delta)$ transverse to the foliation.
Then, since these are strata of $\lfloor \Delta \rfloor$, it follows that the flipping locus  meets some divisorial component $W$ of $\lfloor \Delta \rfloor$.
Thus, to prove our result it suffices to show that, for any component
$W$ of $\lfloor \Delta \rfloor$, after finitely many flips the flipping locus is disjoint from $W$.

So suppose that $W$ meets the flipping locus.
Note that by Lemma \ref{singcomparison} we see that $(X, \Delta)$ is dlt and so by \cite[Corollary 5.52]{KM98} it follows that $W$ is normal.
Let $\psi\colon W \dashrightarrow W^+$ be the induced map.
Let $\cal G$ denote the restricted foliation
and write
\[
(K_{\cal F}+\Delta)\vert_W = K_{\cal G}+\Theta
\]
where, by Lemma \ref{adjunction}, $(\cal G, \Theta)$ is F-dlt. In particular, $W$ is klt.

By Proposition \ref{p_disjointtranscentre}, none of the curves in  the flipped locus is  contained in $W^+$. Thus, $\psi$ is a birational contraction. If $\psi$ does not contract any divisors, then there exists a curve $Z$ contained in the flipping locus, such that $Z\cap W\neq \emptyset$ but $Z$ is not contained in $W$.  Then $Z\cdot W>0$ and so if
$Z^+$ is a flipped curve we must have $Z^+\cdot W^+ <0$ implying that $Z^+\subset W^+$, a contradiction.  

Thus, $\psi$ contracts a divisor at each flip. In particular, each flip reduces $\rho(W)$ by $1$ and we can only have finitely many such flips. 
%
%
%
%
%
\end{proof}

By Corollary \ref{c_termdisjointtranscentre}, it suffices to show that the flipping
locus is eventually disjoint from lc centres which are tangent to the foliation.

\begin{lemma}
\label{noncontained}
After a finite sequence of flips, 
if $Z$ is an lc centre and $C$ is a flipping curve
then $Z$ is not contained in $C$.
\end{lemma}
\begin{proof} By Corollary \ref{c_termdisjointtranscentre}, 
after finitely many flips we may assume that the flipping locus is disjoint
from all lc centres transverse to the foliation, in particular,
it is disjoint from $\lfloor \Delta \rfloor$.

By Proposition \ref{p_finmanylccenters}, there are only finitely many lc centres of $(\cal F, \Delta)$ not contained in
$\lfloor \Delta \rfloor$  and so the claim follows from Lemma \ref{l_negativity}.
\end{proof}

\begin{defn}[Hyperstandard set]\label{d_hyperstandard}
Let $I\subset [0,1]$ be a subset. We define:
\[
I_+=\{ \sum_{j=1}^m a_j i_j\mid i_j\in I, a_j\in \mathbb N \text{ for } j=1,\dots,m\}\cap [0,1],
\]
and
\[
D(I)=\{ \frac{m-1+f}{m}\mid m\in \mathbb N, f\in I_+\}\cap [0,1].
\]
\end{defn}
Note that $D(D(I))=D(I)$ \cite[Lemma 4.4]{MP04}.
Moreover, if $I\subset [0,1]$ is a finite set, then the only accumulation point of $D(I)$ is $1$ and, in particular, it satisfies DCC. 

\begin{lemma}
\label{l_adjunction}
Let $(X,\Delta)$ be a log canonical pair such that the coefficients of $\Delta$ belong to a subset $I\subset [0,1]$, and let $S$ be an irreducible component of $\lfloor \Delta \rfloor$. Let $\nu\colon S^{\nu}\to S$ be the normalisation and let $\Theta$ be the divisor on $S^{\nu}$ defined by adjunction:
\[
(K_X+\Delta)|_{S^{\nu}}=K_{S^{\nu}}+\Theta.
\]

Then,  the coefficients of $\Theta$ belong to $D(I)$.
\end{lemma}
\begin{proof}
Let $\pi\colon X' \rightarrow X$ be a dlt blow up $(X, \Delta)$, see \cite[Theorem 10.4]{Fujino09},
and write $K_{X'}+\Delta' = \pi^*(K_X+\Delta)$.  Note that, 
after possibly replacing $I$ by $I \cup \{1\}$,
 the coefficients of $\Delta' $ belong to  $I$ 
and that every irreducible component of $\lfloor \Delta \rfloor$ is normal.
We may freely replace $(X, \Delta)$ by $(X', \Delta')$ and so we may assume without loss of generality
that $S$ is normal.  The result is then a direct consequence of \cite[Lemma 4.3]{MP04}.
\end{proof}

\medskip 

Given a F-dlt pair $(\cal F,\Delta)$ on a $\bb Q$-factorial quasi-projective threefold $X$, we denote by $d$ the minimal dimension of an lc centre of $(\cal F,\Delta)$ which is tangent to $\cal F$ and intersects the flipping locus of a $(K_{\cal F}+\Delta)$-flip. Our goal is to show that there can be only finitely many flips with $d=0,1$ or $2$. 

\subsection{Special termination in dimension $d = 0$}
This follows directly from Lemma \ref{noncontained}

\subsection{Special termination in dimension $d = 1$}

\begin{lemma}
Set up as above. Then
after finitely many flips the flipping locus is disjoint from all
one-dimensional lc centres.
\end{lemma}
\begin{proof} Let $I$ be the set of coefficients of $\Delta$.  Recall that $D(I)$ satisfies DCC.
Let $C$ be a one-dimensional lc centre which is tangent to $\cal F$ and let $\nu\colon C^{\nu}\to C$ be its normalisation. Let $\Theta\ge 0$ 
be the $\mathbb Q$-divisor on $C^{\nu}$ whose existence is guaranteed by Lemma \ref{l_tangentlccentre}. 
Since $D(D(I))=D(I)$, it follows by Lemma \ref{l_adjunction} and the proof of Lemma  \ref{l_tangentlccentre} that 
 the coefficients of $\{\Theta\}$ take values in $D(I)$.
It follows by
Lemma \ref{l_negativity}
 that, after a flip, $\{\Theta\}$ strictly decreases. However, by Lemma \ref{l_tangentlccentre} and since we are assuming that there are no zero-dimensional lc centre intersecting the flipping locus, it follows  that the flip is an isomorphism near $\lfloor \Theta \rfloor$
and the result follows.
\end{proof}

\subsection{Special termination in dimension $d = 2$}

\begin{lemma}
Set up as above. Then after finitely many flips the flipping locus
is disjoint from all two-dimensional lc centres.  
\end{lemma}
\begin{proof}
Let $I\subset [0,1]$ be a finite set containing the coefficients of $\Delta$. Let $S$ be a two-dimensional  lc centre intersecting the flipping locus. By Corollary \ref{c_termdisjointtranscentre}, we may assume that $S$ is $\mathcal F$-invariant and, by Lemma \ref{adjunction},  we may  write $(K_{\cal F}+\Delta)\vert_W = K_W+\Theta$ for some $\bb Q$-divisor $\Theta\ge 0$ where $W \rightarrow S$ is the normalisation and such that
$(W,\Theta':=\lfloor \Theta \rfloor_{red}+\{\Theta\})$ is lc
and $(W, (1-\epsilon)\Theta')$ is klt for $0 < \epsilon < 1$. 
Note that the coefficients of $\Theta'$ belong to $D(I)$. 

We define 
\[
d_I(W,\Theta)=\sum_{a\in D(I)}\#\{E\mid a(E,W,\Theta)<-a \text{ and }c_X(E)\not\subset \lfloor \Theta' \rfloor \}.
\]
Then $d_I(W,\Theta)<\infty$. 

Let $\phi\colon  X \dashrightarrow X^+$ be a flip,
and let $\psi\colon S \dashrightarrow S^+$ be the induced birational map.  We denote by $(\mathcal F^+,\Delta^+)$ be the transformed foliated pair  and we write
$(K_{\cal F^+}+\Delta^+)\vert_{W^+} = K_{W^+}+\Theta^+$ where $W^+ \rightarrow S^+$ is the normalisation.
Note that $d_I(W,\Theta)\ge d_I(W^+,\Theta^+)$.

Suppose first that $\psi^{-1}$ contracts a divisor $D\subset S^+$. 
Let $Z \subset S$ be the centre of $D$ on $S$.  By induction we know that $Z$ is not contained
in $\lfloor \Theta' \rfloor$. It follows that $d_I(W^+,\Theta^+)<d_I(W,\Theta)$. 

Thus, after finitely many flips,  we may assume that $\psi$ is a birational contraction. 
As in the proof of Corollary \ref{c_termdisjointtranscentre}, the claim follows. 
%
\end{proof}

The Lemma above concludes the proof of Theorem \ref{t_spter}.

\begin{corollary}
\label{specialcaseMMP}
Let $X$ be a $\bb Q$-factorial quasi-projective threefold and let $\cal F$ be a co-rank 1 foliation with non-dicritical singularities on $X$.
Let $\pi\colon X\to Z$ be a birational morphism. 
Let $(\cal F, \Delta)$ be an F-dlt pair on $X$.  Assume that every component of $\text{exc}(\pi)$ is an lc centre for $(\cal F,\Delta)$. 

Then any sequence of $(K_{\cal F}+\Delta)$-flips over $Z$
\[
X=X_0\dashrightarrow X_1\dashrightarrow X_2\dashrightarrow \dots
\]
terminates.
\end{corollary}
\begin{proof}
By Theorem \ref{t_spter}, any sequence of flips is eventually disjoint from the lc centres of 
$(\cal F, \Delta)$ and so is eventually disjoint from $\text{exc}(\pi)$, in which case the MMP
terminates.
\end{proof}


\section{Existence of F-dlt modifications}

We now show the existence of an F-dlt modification as in Definition \ref{d_fdltm}. The result is a consequence of the existence of flips and special termination and it will be used to prove the base point free theorem in Section \ref{s_bpft}.

\begin{theorem}[Existence of F-dlt modifications]
\label{t_existencefdlt}
Let $\cal F$ be a co-rank one foliation on a 
normal projective variety $X$ of dimension at most three.
Let $(\cal F, \Delta)$ be a foliated pair.

Then  $(\cal F, \Delta)$ admits a F-dlt modification $\pi\colon Y\to X$ (cf. Definition \ref{d_fdltm}) such that
if $\cal G$ is the pulled back foliation on $Y$ then 
\begin{enumerate}

\item $Y$ is $\bb Q$-factorial,

\item $Y$ is klt and

\item $\cal G$ is non-dicritical.
\end{enumerate}

Suppose in addition that $(\cal F,\Delta)$ is lc 
 and let $\Gamma\coloneqq \pi_*^{-1} \Delta +\sum \epsilon(E_i)E_i$ where the sum is over all the $\pi$-exceptional divisors. Then we may choose $\pi\colon Y\to X$ so that 
 if $Z$ is an lc centre of 
$(\cal G, \Gamma)$ then $Z$ is contained in a codimension one lc centre of $(\cal G, \Gamma)$.
\end{theorem}
\begin{proof}
We assume that $\dim X=3$. The case where $\dim X=2$ is similar. 
Let $\phi\colon W\to X$ be a  foliated log resolution. Let $\cal H$ be the pulled back foliation on $W$.

We may write
\[
K_\cal H + \sum b_iG_i +\sum a_jF_j+\phi_*^{-1} \Delta= \phi^*(K_{\cal F}+\Delta)
\]
where $G_i, F_j$ are $\phi$-exceptional prime divisors such that  $b_i \geq \epsilon(G_i)$ and $a_j < \epsilon(F_j)$.
Let 
\[
\widetilde{\Delta} = \sum b_iG_i +\sum a_jF_j+\phi_*^{-1} \Delta
\]
and
\[
\Theta = \sum \epsilon(G_i)G_i +\sum \epsilon(F_j)F_j+\phi_*^{-1} \Delta.
\]
Note that $(\cal H, \Theta)$ is F-dlt.  By Proposition \ref{p_finmanylccenters}, after possibly replacing $W$ by a sufficiently high log resolution, we may assume that every lc centre
is contained in a codimension one lc centre. 

We run a 
$(K_{\cal H}+\Theta)$-MMP over $X$. Recall that, by Remark \ref{r_simple=ndc}, simple singularities are non-dicritical. 
By 
Theorem \ref{contractionsareextremal}
all the required divisorial
contractions exists and
by Theorem \ref{flipsexist} all the flips exist.  By construction
each $G_i, F_j$ is an lc centre of $(\mathcal H, \Theta)$ and so Corollary
\ref{specialcaseMMP} implies that this MMP terminates. 
Call this MMP $f\colon W \dashrightarrow Y$ and let $\cal G$ be the transformed foliation on $Y$.
Note that Lemma \ref{l_mmp-fdlt} implies that  $(\cal G, f_*\Theta)$ is F-dlt.

The MMP preserves $\bb Q$-factoriality, 
klt singularities (cf. Theorem \ref{flipsexist} and Theorem \ref{contractionsareextremal}) and non-dicriticality (cf. Lemma \ref{l_image_cont_non-dicrti}), and so we have that $Y$ is $\bb Q$-factorial and klt and $\cal G$ is non-dicritical.

Denote by $\pi\colon Y\to X$ the induced morphism.
We have $K_{\cal G}+f_*\widetilde{\Delta} = \pi^*(K_{\cal F}+\Delta)$
and so $D:= f_*\Theta - f_*\widetilde{\Delta}$ is $\pi$-nef and $\pi$-exceptional.
The negativity lemma then implies that $f_*\widetilde{\Delta} - f_*\Theta=-D \geq 0$.
Thus, setting $F =-D$ and noting that $f_*\Theta = \pi_*^{-1}\Delta+\sum \epsilon(E_i)E_i$
where we sum over the $\pi$-exceptional divisors, we have
\[
K_{\cal G}+\pi_*^{-1}\Delta+\sum \epsilon(E_i)E_i+F = \pi^*(K_{\cal F}+\Delta).\]

We now show the last claim. To this end,  
we may freely replace $(\cal F, \Delta)$ by $(\cal G, f_*\Theta)$ and so we may assume
that $(\cal F, \Delta)$ is F-dlt. Arguing as above we see that it suffices to show that the $(K_{\cal H}+ \Theta)$-MMP does not contract
any component in the support of $\sum G_i$.
By assumption, $b_i = \epsilon(G_i)$ and so
\[(K_{\cal H}+\Theta) - (K_{\cal H}+\widetilde{\Delta})  = \sum (\epsilon(F_j)-a_j)F_j \geq 0.\]
Since $K_{\cal H}+\widetilde{\Delta}$ is trivial over $X$, each step of the $(K_\cal H+\Theta)$-MMP is 
$\sum (\epsilon(F_j)-a_j)F_j$-negative and so only components in the support
of $\sum (\epsilon(F_j)-a_j)F_j$ are contracted.  In particular, no component in the support of $\sum E_i$ is 
contracted by the MMP and our result follows.
\end{proof}

\begin{theorem}[Cone theorem for lc pairs]
\label{t_cone-lc}
Let $X$ be a normal projective threefold and let $\mathcal F$ be a co-rank one foliation on $X$ with non-dicritical singularities. Suppose that $X$ is potentially klt.
Let  $(\mathcal F,\Delta)$ be an lc pair where $\Delta\ge 0$ and let $H$ be an ample $\bb Q$-divisor. 

Then there exist countable many curves $\xi_1,\xi_2,\dots$ such that 
\[
\overline {NE(X)}=\overline{NE(X)}_{K_{\mathcal F}+\Delta\ge 0}+ \sum\bb R_+[\xi_i].
\]

Furthermore, for each $i$,  $\xi_i$ is a rational curve tangent to $\cal F$ such that $(K_{\mathcal F}+\Delta)\cdot \xi\ge -6$,
and if $C \subset X$ is a curve such that $[C] \in \bb R_+[\xi_i]$ and $\loc {\bb R_+[\xi_i]} \neq X$ then $C$ is tangent to $\cal F$.


In particular, there exist only finitely many  $(K_{\mathcal F}+ \Delta+H)$-negative extremal rays. 
\end{theorem}
\begin{proof}
By Theorem \ref{t_existencefdlt}, there exists a 
F-dlt modification 
 $\pi\colon Y\to X$ for the foliated pair $(\cal F, \Delta)$.
We may write $K_{\cal G}+\Gamma = \pi^*(K_{\cal F}+\Delta)$ where $\cal G$ is the transformed foliation on $Y$ and $\Gamma\ge 0$.

Observe that if $R \subset \overline{NE}(X)$ is an extremal ray then there exists an extremal ray $R' \subset \overline{NE}(Y)$
such that $\pi_*R' = R$. 
Moreover, if $R$ is $(K_{\cal F}+\Delta)$-negative then $R'$ is $(K_{\cal G}+\Gamma)$-negative and so by Theorem \ref{t_cone-dlt}
$R'$ is spanned by a rational curve $\xi$ tangent to $\cal G$ with $(K_{\cal G}+\Gamma)\cdot \xi \geq -6$.
Then $\pi(\xi)$ spans $R$ and has all the desired properties.

If $C \subset X$ and $[C] \in \bb R_+[\xi_i]$ for some $i$ then we may apply Lemma \ref{l_curvesinrayaretangent} to conclude
that $C$ is tangent to $\cal F$.
%
%
\end{proof}

\subsection{Contraction in the non-$\bb Q$-factorial case}\label{s_nonqfactorial}

Through out this subsection, we assume that $(X, \Gamma)$ is klt for some $\Gamma \geq 0$ but that $X$
is not necessarily $\bb Q$-factorial.

\begin{lemma}\label{l_contraction_as}
Let $X$ be a normal projective threefold and let $\cal F$ be a co-rank one foliation on $X$ with non-dicritical singularities. Suppose that $X$ is potentially klt. 
Let $\Delta\ge 0$ be a $\bb Q$-divisor such that $(\mathcal F,\Delta)$ is a log canonical pair,
and suppose there exists a small $\bb Q$-factorialisation $\pi\colon \overline{X} \rightarrow X$ (cf. Definition \ref{defn_small_q_fact})
such 
that if we write $K_{\overline{\cal F}}+\overline{\Delta} = \pi^*(K_{\cal F}+\Delta)$, 
where $\overline{\cal F}$ is the pulled back foliation on $\overline{X}$, 
then for any choice of $\epsilon>0$ we may find $\Theta$ such that  
$(1-\epsilon)\overline{\Delta} \leq \Theta \leq \overline{\Delta}$ and $(\overline{\cal F}, \Theta)$
is F-dlt.
Let $R$ be a $(K_{\cal F}+\Delta)$-negative extremal ray. Assume that $\loc R \neq X$.

Then there exists a contraction
\[\phi_R\colon X \rightarrow \cal Y\] 
of $R$ in the category of algebraic spaces.
\end{lemma}
\begin{proof}
By Theorem \ref{t_cone-lc}, there exists a nef $\mathbb Q$-Cartier divisor $H_R$ on $X$ which defines a supporting hyperplane for $R$ in $\overline {NE}(X)$.
Let $\pi\colon \overline{X} \rightarrow X$ be a small $\bb Q$-factorialisation
of $X$ as in the hypotheses of the Lemma.

First, suppose that $\loc R = D$ is a divisor.  By Theorem \ref{t_cone-lc}, we may find  $\epsilon>0$ sufficiently small and a $\mathbb Q$-divisor
$\Theta$ as in the hypotheses of the Lemma so that $K_{\overline{\cal F}}+\Theta$ is negative on any extremal ray $R'$
in $\overline{NE}(\overline{X})$ such that $\pi_*R' = R$.

Suppose that  $D$ contains an irreducible component $D_0$ transverse to the foliation.
Let $\overline{D_0}$ be the strict transform of $D_0$ under $\pi$ 
and let $\nu\colon \overline{D_0}^{\nu}\to \overline{D_0}$ the normalisation map. We may write
\[
\nu^*(K_{\overline{\cal F}}+\overline{\Delta}+\overline{D_0}) = K_{\cal G}+\Gamma
\]
 where $\cal G$ is the restricted foliation and $\Gamma\ge 0$. It follows that $\overline{D_0}^{\nu}$ is covered by $(K_{\cal G}+\Gamma)$-negative curves and, therefore, $K_{\cal G} +\Gamma$ is not pseudo-effective. Thus, 
\cite[Corollary 2.28]{Spicer17} implies that $\overline {D_0}^\nu$ is covered by rational curves which are tangent
to $\cal G$ and, in particular, $\cal G$ is algebraically integrable. Thus, since ${\overline{\cal F}}$ is non-dicritical, there exists a morphism $\overline{D_0}^\nu\to B$ onto a curve $B$ such that the general fibre is $\mathbb P^1$. 
In particular, if 
 $F$ is a general fibre in this $\bb P^1$-fibration structure,
then we claim that 
\[
(K_{\overline{\cal F}}+\overline{\Delta}+\overline{D_0})\cdot F = (K_{\overline X}+\overline{\Delta}+\overline{D_0})\cdot F.
\]
Indeed, if we write  $\nu^*(K_X+\overline{\Delta}+\overline{D_0}) = K_{D^\nu_0}+\Gamma'$ then by Lemma \ref{lem_trans_diff_comp} we see that the coefficient in $\Gamma$ and $\Gamma'$ of a curve $\Sigma$
which is not a fibre of $D^{\nu}_0\rightarrow B$ is the same.  In particular, it follows that $\Gamma\cdot F = \Gamma' \cdot F$.
Finally, observe that $K_{\overline{D_0}^{\nu}}\cdot F = K_{\cal G}\cdot F = -2$ and our stated equality follows.

We run a $(K_{\overline{\cal F}}+\Theta)$-MMP contracting/flipping only $\pi^*H_R$-trivial extremal rays $R'$
such that $\loc {R'}$ meets the strict transform of $D$.
Theorem \ref{flipsexist} implies that  all the required flips exist.

By our above observation we know that if the strict transform of $D$ contains a non-invariant component we may choose 
this extremal ray to also be $(K_{\overline{X}}+\Theta)$-negative.
Thus, by classical termination of log flips we see that there can only be finitely many flips before the strict transform of each non-invariant
component of $D$ is contracted. 
By Special Termination (cf. Theorem \ref{t_spter}), there are only finitely many such flips before the strict transform of each invariant component of $D$ is contracted.

Denote by $\overline{Y}$ the step in this MMP after the last component of $D$ is contracted,
and let $f:\overline{X} \dashrightarrow \overline{Y}$ denote the induced rational map. 
Observe that each step of this MMP is $\pi^*H_R$-trivial
and so $\pi^*H_R$ descends to a $\bb Q$-Cartier divisor $M$ on $\overline{Y}$

We know that $\overline{Y}$ is $\bb Q$-factorial and, as in the proof of
\cite[Lemma 8.20]{Spicer17},
we know that if $S$ is a divisor on $\overline {Y}$ then $M^2\cdot S>0$.  Moreover, if $B$ is a curve we see
that 
$M\cdot B = 0$ if and only if $B$ is the strict transform of a $\pi$-exceptional curve, $B$ is the strict transform
of a flipped curve or $B$ is the strict transform of a curve $C$ with $[C] \in R$.  Notice that 
there are finitely many such curves and let $\Sigma$ be the union of all such curves.
By \cite[Lemma 8.21]{Spicer17} there exists a contraction of $\Sigma$ in the category of algebraic spaces, call
it $c \colon \overline{Y} \rightarrow \cal Y$.  Since $c$ contracts every $\pi$-exceptional curve and every 
flipped curve it gives a contraction
\[\phi_R \colon X \rightarrow \cal Y.\]

Now suppose that $\loc R$ is a curve.
In this case we see that if $S$ is a divisor on $\overline{X}$ then, 
as in the proof of \cite[Lemma 8.20]{Spicer17}, we have  that $(\pi^*H_R)^2 \cdot S >0$
and that $(\pi^*H_R)\cdot C = 0$ if and only if $C \subset \pi^{-1}(\loc R) \cup \exc\pi$.  As above, we apply \cite[Lemma 8.21]{Spicer17}
to produce a contraction of $\pi^{-1}(\loc R)$ in the category of algebraic spaces, which factors through $\pi$.
\end{proof}

\begin{theorem}
\label{t_nonqfactcontract}
Let $(X,\Gamma)$ be a  projective three-dimensional klt pair and let $\mathcal F$ be a co-rank one foliation on $X$ with  non-dicritical singularities. 
Let $\Delta\ge 0$ be a $\bb Q$-divisor such that $(\mathcal F,\Delta)$ is a log canonical pair
and suppose there exists a small $\bb Q$-factorialisation $\pi\colon \overline{X} \rightarrow X$ such 
that if we write $K_{\overline{\cal F}}+\overline{\Delta} = \pi^*(K_{\cal F}+\Delta)$, 
where $\overline{\cal F}$ is the pulled back foliation on $\overline{X}$, 
then for any choice of $\epsilon>0$ we may find $\Theta$ such that  
$(1-\epsilon)\overline{\Delta} \leq \Theta \leq \overline{\Delta}$ and $(\overline{\cal F}, \Theta)$
is F-dlt.
Let $R$ be a $(K_{\cal F}+\Delta)$-negative extremal ray. Assume that $\loc R \neq X$.

Then the contraction associated to $R$ 
\[
\phi_R\colon X \rightarrow Y
\]
exists in the category
of projective varieties and $\rho(X/Y) = 1$.  In particular,  $\phi_R$ is extremal.
\end{theorem}
Note that, by Lemma \ref{lem_fdlt_small_mod}, the assumption of the Theorem are verified if $(\cal F,\Delta)$ is F-dlt.

\begin{proof}


Let $\phi_R\colon X \rightarrow \cal Y$ be the contraction onto an algebraic space $\cal Y$, whose existence is guaranteed by Lemma \ref{l_contraction_as}.
We first show that if $M$ is a $\bb Q$-Cartier divisor with $M\cdot R = 0$ then $M = \phi_R^*N$ for some $\bb Q$-Cartier
divisor on $\cal Y$.
Observe that this problem is \'etale local on $\cal Y$, so we may freely replace $\cal Y$ by a sufficiently
small \'etale neighborhood of some point $y \in \cal Y$.  

Let $\pi\colon \overline{X} \rightarrow X$ be a small $\bb Q$-factorialisation
of $X$ as in the hypotheses of the Lemma
and let $g\colon \overline{X} \rightarrow \cal Y$ be the composition of $\pi$ with $\phi_R$.
We may write 
\[
K_{\overline{\cal F}} +\overline{\Delta} = \pi^*(K_{\cal F}+\Delta)
\quad\text{and}\quad
K_{\overline{X}}+\overline{\Gamma} = \pi^*(K_X+\Gamma).
\]

By Lemma \ref{l_approx_etale}, we may replace $\cal Y$ by another \'etale neighborhood of $y \in \cal Y$, so that we may approximate
every invariant divisor (formal or otherwise) of $\overline{\cal F}$ meeting $g^{-1}(y)$ 
by global divisors on $\overline{X}$.
Let $S_k$ be the collection of all such divisors. 
As in Section \ref {s_flip_construction} we see that
\[(K_{\overline{\cal F}}+\overline{\Delta}) - (K_{\overline{X}}+\overline{\Delta}+\sum S_k)\]
is $g$-nef.

For some $\epsilon>0$ sufficiently small, we may run a 
$(K_{\overline{X}}+\overline{\Gamma}+\epsilon (\overline{\Delta}+ \sum S_k))$-MMP 
over $X$, and we obtain a map $\overline X\dashrightarrow \overline X'$.  
Let $\pi' \colon \overline{X}' \rightarrow X$ be the induced morphism.  Each step of this MMP
is $(K_{\overline{X}}+\overline{\Gamma})$-trivial and so if we let $T_k$ be the strict transform of $S_k$ and $\overline{\Delta}'$ be the strict transform of $\overline{\Delta}$ on $\overline X'$, 
we see that $\overline{\Delta}'+\sum T_k$ is nef over $X$.  Observe that $\sum T_k$ still approximates the  divisors which are invariant with respect to 
the transformed foliation $\overline{\cal F}'$ on $\overline X'$.

Thus, replacing $\overline{X}$ by $\overline{X}'$ we may freely assume that $(\Delta+\sum S_k) \cdot C \geq 0$
for any $\pi$-exceptional curve $C$.
Since $K_{\cal F}+\Delta$ is strictly negative on any $\phi_R$-exceptional curve, we see that 
for $0< \delta \ll 1$ we have 
\[-(K_{\overline{X}}+(1-\delta)(\overline{\Delta}+\sum S_k))\] is nef over $\cal Y$.
By Lemma \ref{singcomparison} and the fact that
$(\overline{\cal F}, \overline{\Delta})$ is log canonical we see that 
$(\overline{X}, \overline{\Delta}+\sum S_k)$ is log canonical and since $\overline{X}$ is klt 
we have that  $(\overline{X}, (1-\delta)(\overline{\Delta}+\sum S_k))$ is klt for $\delta>0$ and
so we may apply the base point free theorem to $\pi^*M$ to conclude that there exists a $\bb Q$-Cartier divisor
$N$ on $\cal Y$ with $\phi_R^*N = M$.

Let $H_R$ be 
 a nef $\mathbb Q$-Cartier divisor on $X$ which defines a supporting hyperplane for $R$ in $\overline {NE}(X)$. Taking $M = H_R$ we see that by applying the Nakai-Moishezon criterion for algebraic spaces (cf. \cite[Theorem 3.11]{Kollar90}) to $N$
that $N$ is ample and hence $\cal Y = Y$ is projective.
\end{proof}

Observe that, by Lemma \ref{lem_fdlt_small_mod},  the hypotheses of Theorem \ref{t_nonqfactcontract}
are satisfied if we suppose that $(\cal F, \Delta)$ is F-dlt.

\subsection{Potentially klt varieties}

\begin{lemma}
\label{l_lockltglobklt}
Let $X$ be a normal projective variety.  Suppose that $X$ is \'etale locally potentially klt. Then $X$ is potentially klt.

In particular, let $\phi_R\colon X \rightarrow Y$ be the contraction associated to an extremal ray as in Theorem
\ref{t_nonqfactcontract}.
Then $Y$ is potentially klt.
\end{lemma}
\begin{proof}
Choose a finite \'etale cover 
\[\{g_i \colon U_i \rightarrow X\}_{i = 1, ..., N}\] 
such that 
$U_i$ is affine and there exists $\Delta_i \geq 0$ such
that $(U_i, \Delta_i)$ is klt.

Without loss of generality we may assume each $g_i$ is Galois,
with Galois group $G_i$. Perhaps replacing $\Delta_i$ by $\frac{1}{\# G_i}\sum_{g \in G_i} g\cdot \Delta_i$
we may assume that there exists a Zariski open set $V_i \subset X$ and a $\bb Q$-divisor $\Theta_i \geq 0$
on $V_i$ such that $g_i$ factors through $V_i$,  $K_{V_i}+ \Theta_i$ is $\bb Q$-Cartier  and $g_i^*(K_{V_i}+\Theta_i) =K_{U_i}+\Delta_i$.  Observe that $(V_i, \Theta_i)$ is klt and so we may freely assume that
$g_i \colon U_i \rightarrow X$ is an open immersion.

There exists $m>0$ such that $m\Delta_i \in  \lvert -mK_{U_i} \rvert$ for all $i$. Let $H$ be a divisor on $X$
such that $\cal O(-mK_X+mH)$ is globally generated. We may assume that, 
for all $i$, there exists
$D_i \in \lvert -mK_X+mH \rvert$ such that $(U_i, \frac{1}{m}D_i|_{U_i})$ is klt.
It follows that, for a general element $D \in \lvert -mK_X+mH \rvert$, we have that $(X, \frac{1}{m}D)$ is klt.
Thus, $X$ is potentially klt. 

To prove our final claim it suffices to check that $Y$ is 
\'etale locally potentially klt.  So let $y \in Y$ and let $U$ be a sufficiently small 
\'etale neighborhood of $y$, and let $X_U = X \times_Y U$.
By the construction given in proof of Theorem \ref{t_nonqfactcontract}, there exists a small morphism
$\pi \colon \overline{X_U} \rightarrow X_U$ and a divisor
 $D \geq 0$ such that $(\overline{X_U}, D)$ is klt and  $-(K_{\overline{X_U}}+D)$ is $\phi_R$-nef.
By the basepoint free theorem, we may find a $0 \leq A \sim_{\bb Q} -(K_{\overline{X_U}}+D)$ such
that $(\overline{X_U}, D+A)$ is klt and $K_{\overline{X_U}}+D+A$ is $\phi_R$-trivial.
Thus, $(U, (\phi_R)_*(D+A))$ is klt and so $X$ is \'etale locally potentially klt.
\end{proof}

\section{Base Point Free Theorem}\label{s_bpft}

The goal of this section is to prove the base point free theorem. We begin with the following version of the canonical bundle formula:

\begin{lemma}
\label{canonicalbundleformula}
Let $X$ be a normal projective variety of dimension at most three 
and
let $(\cal F, \Delta)$ be an lc pair on $X$ such that $\Delta\ge 0$.
Suppose that there is a fibration $f\colon X \rightarrow Y$ onto a variety $Y$ with klt singularities and 
such that the general fibre of $f$ is tangent to $\cal F$. Assume that $K_{\cal F}+\Delta \sim_{\bb Q, f} 0$.

Then there is a foliation $\cal G$ on $Y$ such
that $f^{-1}\cal G = \cal F$, and a $\bb Q$-divisor $\Theta \geq 0$ and a semi-ample divisor $D$
such that $K_{\cal F}+\Delta \sim_{\bb Q} f^*(K_{\cal G}+\Theta+D)$
and $(\cal G, \Theta)$ is lc.
\end{lemma}
\begin{proof}
First, notice that since the fibres of $f$ are tangent to $\cal F$ there exists
a foliation $\cal G$ on $Y$ so that $\cal F = f^{-1}\cal G$.
We also have that there exists $M$ on $Y$ so that 
$K_{\cal F}+\Delta \sim_{\bb Q} f^*M$.

Consider a commutative diagram as follows

\begin{center}
\begin{tikzcd}
X' \arrow{d}{g} \arrow{r}{\nu} & X \arrow{d}{f} \\
Y' \arrow{r}{\mu} & Y \\
\end{tikzcd}
\end{center}
where $\mu, \nu$ are resolutions of singularities.
Let $\cal F'$ and $\cal G'$ be the transformed foliations
on $X'$ and $Y'$ respectively.

Write
\[
K_{\cal F'} +\Delta' = \nu^*(K_{\cal F}+\Delta).
\]
 and so we have $K_{\cal F'}+\Delta' \sim_{\bb Q} g^*(\mu^*M)$.

We claim that there exists an open subset $U\subset Y'$ such that $Y'\setminus U$ has codimension at least two and  such that, over $U$ we have that $K_{\cal F'/\cal G'} = K_{X'/Y'} - R$ where $R$ 
is $\cal F'$-invariant and $K_{\cal F'/G'}\coloneqq K_{\cal F'}-g^*K_{\cal G'}$. 
Indeed, $R$ is supported in the zero locus of a $1$-form obtained as the pull-back of a $1$-form on $Y'$ which defines $\cal G'$. 
Thus, the claim follows. 
By \cite[Theorem 8.3.7]{Kollar07} we may find
a nef $\bb Q$-divisor $J$ and an effective $\bb Q$-divisor $B$
such that $\mu^*M \sim_{\bb Q} K_{\cal G'}+B+J$.  Furthermore,
by \cite[Theorem 0.1]{Ambro04} in the case $\text{dim}(Y) = 1$ and \cite[Theorem 8.1]{PS09} when $\text{dim}(Y) = 2$ 
we know that $J$ is in fact semi-ample.

If $B=\sum a_iB_i$, then
\[
a_i=1-\sup \{t \mid (X', \Delta'-R+tg^*B_i) \text{ is lc above the generic point of } B_i\}.
\]
An explicit calculation
shows that 
\[
a_i=\epsilon(B_i) - \sup \{t \mid (\cal F', \Delta'+tg^*B_i) \text{ is lc above the generic point of } B_i\}.
\]

Thus, since $(\cal F', \Delta')$ is lc, it follows that $a_i \leq \epsilon(B_i)$
and $\mu_*B \geq 0$.

Since $Y$ is klt of dimension at most two, it follows that $Y$ is $\mathbb Q$-factorial.
Since $J$ is semi-ample the base locus of $\mu_*J$ consists of isolated points. 
 Thus, Fujita Theorem \cite[Theorem 1.10]{Fujita83} implies that $\mu_*J$ is semi-ample.
Letting $\Theta = \mu_*B$ and $D = \mu_*J$ gives our result.
\end{proof}

\begin{lemma}\label{l_lambda}
Let $X$ be a normal projective threefold and let $\mathcal F$ be a co-rank one foliation with non-dicritical singularities. Suppose that $(X, D)$ is klt for some $D \geq 0$.
Let $\Delta=A+B$ be a $\bb Q$-divisor such that $(\mathcal F,\Delta)$ is  an lc pair, $A\ge 0$ is an ample $\bb Q$-divisor and $B\ge 0$. Assume that $K_{\mathcal F}+\Delta$ is not nef, but there exists a $\bb Q$-divisor $H$ such that $K_{\mathcal F}+\Delta+H$ is nef. 
Let 
\[
\lambda=\inf\{t>0\mid K_{\mathcal F}+\Delta+tH \text{ is nef }\}.
\] 

Then there exists a $(K_{\mathcal F}+\Delta)$-negative extremal ray $R$ such that $(K_{\mathcal F}+\Delta+\lambda H)\cdot \xi=0$, for any $\xi\in R$. 
\end{lemma}

\begin{proof}
By Theorem \ref{t_cone-lc}, there exist only finitely many curves $\xi_1,\dots,\xi_m$ such that
if $R_i=\bb R_+ [\xi_i]$ for $i=1,\dots,m$, then $R_1,\dots,R_m$ are $(K_{\mathcal F}+\Delta)$-negative extremal rays. 

Let $C\coloneqq K_{\cal F}+\Delta+H$ and let 
\[
\mu=\min_i \frac{C\cdot \xi_i}{-(K_{\mathcal F}+\Delta)\cdot \xi_i}.
\]
It follows easily that $\mu=\frac{1-\lambda}{\lambda}$. By construction, there exists $j$ such that 
\[
\frac 1 \lambda (K_{\mathcal F}+\Delta + \lambda H )\cdot \xi_j=(\mu(K_{\mathcal F}+\Delta)+C)\cdot \xi_j=0.
\]
Thus, the claim follows by taking $R=R_j$. 
\end{proof}

\begin{lemma}\label{l_bpft}
Let $X$ be a normal projective threefold and let $\mathcal F$ be a co-rank one foliation with non-dicritical singularities. 
Suppose that $X$ is potentially klt. 
Let $\Delta\ge 0$ be a $\bb Q$-divisor such that $(\mathcal F,\Delta)$ is a log canonical pair,
and suppose there exists a small $\bb Q$-factorialisation $\pi\colon \overline{X} \rightarrow X$ such 
that if we write $K_{\overline{\cal F}}+\overline{\Delta} = \pi^*(K_{\cal F}+\Delta)$, where $\overline{\cal F}$ is the pulled back foliation on $\overline{X}$, 
then for any choice of $\epsilon>0$ we may find $\Theta$ such that  
$(1-\epsilon)\overline{\Delta} \leq \Theta \leq \overline{\Delta}$ and $(\overline{\cal F}, \Theta)$
is F-dlt.
Let $A\ge 0$ and $B\ge 0$ be $\bb Q$-divisors such that $\Delta= A+B$ and $A$ is ample. Assume that $K_{\mathcal F}+\Delta$ is nef.

Then $K_{\mathcal F}+\Delta$ is semi-ample.
\end{lemma}
\begin{proof}
Let $\Gamma=\frac 1 2 A+ B$ and let 
\[
\lambda = \min\{t\ge 0\mid K_{\mathcal F}+\Gamma+tA \text{ is nef }\}.
\]
If $\lambda <1/2$ then $K_{\mathcal F}+\Delta$ is ample and there is nothing to prove. Thus, we may assume that $\lambda=1/2$. By Lemma \ref{l_lambda}, 
there exists a $(K_{\mathcal F}+\Gamma)$-negative extremal ray $R$ such that $(K_{\mathcal F}+\Delta)\cdot \xi =0$ for all $\xi\in R$.

Suppose that $\loc R \neq X$. By Theorem \ref{t_nonqfactcontract}, 
there exists a morphism $f\colon X\to X'$ which contracts exactly all the curves in $R$. 
Let $\mathcal F'$ be the transformed foliation on $X'$ and let 
$A'$ be an ample $\bb Q$-divisor on $X'$ such that $A-f^*A'$ is also ample. Then there exists a $\bb Q$-divisor $A''\ge 0$ on $X'$  and a $\bb Q$-divisor $B\ge 0$ on $X$ such that  $A'\sim_{\bb Q}A''$ and if  $\Delta':=f^*A''+B'$, then $(\cal F,\Delta')$ is log canonical.  Let $\Delta''$ be the image of $\Delta'$ in $X'$.
Then  $\Delta'' = A''+B''$ where $B'' \geq 0$ and
$(\cal F', \Delta'')$ is lc.
Note that $\rho(X')<\rho(X)$. Lemma \ref{l_lockltglobklt} implies
that $X'$ is potentially klt.

If $f$ is a flipping contraction then the existence of a small $\bb Q$-factorialisation $\pi\colon \overline{X'} \rightarrow X'$ satisfying the hypotheses of the lemma is an immediate consequence of the existence by such a small $\bb Q$-factorialisation for $X$.  By Lemma \ref{l_image_cont_non-dicrti}, it follows that  $\cal F'$ has non-dicritical singularities.
Thus we may replace
$(\cal F, \Delta)$ by $(\cal F', \Delta'')$ and continue.

Now suppose $f$ is a divisorial contraction.  Consider a diagram as in the proof of 
Lemma \ref{l_contraction_as}
\begin{center}
\begin{tikzcd}
\overline{X} \arrow{d}{\pi} \arrow[dashed]{r}{\overline{f}} & \overline{X'} \arrow{d}{\pi'} \\
X \arrow{r}{f} & X'\\
\end{tikzcd}
\end{center}
where $\pi\colon\overline{X} \rightarrow X$ is a small $\bb Q$-factorialisation satisfying the hypotheses of the lemma,
$\overline{f}$ is a $(K_{\overline{\cal F}}+\overline{\Gamma})$-MMP 
where $K_{\overline{\cal F}}+\overline{\Gamma} = \pi^*(K_{\cal F}+\Gamma)$
and $\pi'$ is the induced morphism.
We claim that $\pi'\colon \overline{X'} \rightarrow X'$ satisfies the hypotheses of the lemma.  It is immediate that $\overline{X'}$ is projective
and $\bb Q$-factorial since $\overline{X}$ is and $\pi'$ is small.
We may choose $\epsilon, \delta >0$ sufficiently small and  $\Theta$ on $\overline{X}$ such that $(\overline{\cal F}, \Theta)$ is F-dlt 
as guaranteed by our hypotheses
and such that $\overline{f}$ is $(K_{\overline{\cal F}}+\Theta-\delta\pi^*A)$-MMP.
Thus, if we let $\Theta' = \overline{f}_*(\Theta-\epsilon\pi^*A)$ we see that $(\overline{\cal F'}, \Theta')$ is F-dlt.
Observe again, that Lemma \ref{l_image_cont_non-dicrti} implies that  $\cal F'$ has non-dicritical singularities.
We may therefore replace $(\cal F, \Delta)$ by $(\cal F', \Delta'')$ and continue.
After finitely many steps we obtain the claim.

Now assume that $\loc R = X$.  
Let $H_R = K_{\cal F}+\Delta+A$ be a supporting hyperplane to $R$ where $A$ is an ample divisor
and let $\nu = \nu(H_R) <3$. Observe that $H_R^3 = 0$ and $(K_{\cal F}+\Delta)\cdot D_1\cdot D_2 <0$
where $D_i = H_R$ for $1 \leq i \leq \nu$ and $D_i=A$ otherwise.  Then we may apply \cite[Corollary 2.28]{Spicer17}
to see that $X$ is covered by rational curves tangent to $\cal F$ and spanning $R$.

Let $C$ be a general curve spanning $R$ tangent to $\cal F$
and let $C'$ be the strict transform of $C$ on $\overline{X}$.  Notice that $K_{\overline{\cal F}}\cdot C<0$ and
$K_{\overline{X}}\cdot C <0$.  
We may run a $K_{\overline{X}}$-MMP where each step of the 
MMP is $K_{\overline{\cal F}}+\overline{\Delta}$-trivial.  Call this MMP 
$\psi\colon \overline{X} \dashrightarrow W$ and set $\cal H = \psi_*\overline{\cal F}$
and $\Gamma =\psi_*\overline{\Delta}$.  Notice that $(\cal H, \Gamma)$ 
has canonical singularities.  Observe that it suffices to show that
$K_{\cal H}+\Gamma$ is semi-ample.  

Notice that our MMP will terminate with a  Mori fibre space $f\colon W \rightarrow Y$
whose fibres are tangent to $\cal H$.
Notice that, as in the proof of \cite[Theorem 8.13]{Spicer17}, it follows that $f$ is $K_X$-negative. Thus, $Y$ is klt.
By Lemma \ref{canonicalbundleformula}, there is a foliation
$\cal G$ on $Y$, a semi-ample divisor $D$ and $\Theta \geq 0$
such that $(\cal G, \Theta)$ is lc and
\[
K_{\cal H}+\Gamma \sim_{\bb Q} f^*(K_{\cal G}+\Theta+D).
\]

Let $G$ be an ample divisor on $Y$ and choose $0< \delta \ll 1$ such that 
$\Delta - \delta f^*G \sim_{\bb Q} \Delta' \geq 0$
and $(\cal F, \Delta')$ is lc.
Since $K_{\cal F}+\Delta' \sim_{\bb Q, f} 0$, we may apply Lemma \ref{canonicalbundleformula} again
to find $D'$ and $\Theta' \geq 0$ such that $\Theta+D \sim_{\bb Q} \Theta'+D+D'+\delta G \sim_{\bb Q} \Theta'' \geq 0$
where $(\cal G, \Theta'')$ is lc.  
Replacing $X$ and $(\cal F, \Delta)$ by
$Y$ and $(\cal G, \Theta'')$ respectively and proceeding as above, we obtain the claim.
\end{proof}

\begin{theorem}\label{t_bpf} Let $X$ be a normal projective threefold and let $\mathcal F$ be a co-rank one foliation with non-dicritical singularities. Suppose that 
$X$ is potentially klt. 
Let $\Delta$ be a $\bb Q$-divisor such that $(\mathcal F,\Delta)$ is a F-dlt pair.
Let $A\ge 0$ and $B\ge 0$ be $\bb Q$-divisors such that $\Delta= A+B$ and $A$ is ample. Assume that $K_{\mathcal F}+\Delta$ is nef. 

Then $K_{\mathcal F}+\Delta$ is semi-ample.  
\end{theorem}

\begin{proof}
The Theorem follows immediately  from Lemma \ref{lem_fdlt_small_mod} and Lemma \ref{l_bpft}.
\end{proof}


\section{Minimal Model Program with scaling}

\label{S_mmp_with_scaling}

The goal of this section is to show the existence of a minimal model for a F-dlt pair $(\cal F,A+B)$ where $A
\ge 0$ is an ample $\bb Q$-divisor, $B\ge 0$ and such that $K_{\cal F}+A+B$ has non-negative Kodaira dimension.  To this end, we are not able to show termination of flips in general, but we can show that a special sequence of flips terminates. This process is called MMP with scaling. Below, we adopt many of the techniques used in \cite{BCHM06}. 

\medskip 

Let $f\colon X\dashrightarrow Y$ be a proper birational map of normal varieties and let $D$ be a  $\bb Q$-divisor on $X$ such that both $D$ and $D':=f_*D$ are $\bb Q$-Cartier. We say that $f$ is $D$-{\bf non-positive} if for any resolution of indeterminacy $p\colon W\to X$ and $q\colon W\to Y$, we may write
\[
p^*D=q^*D'+E
\]
where $E\ge 0$ is $q$-exceptional.

In particular, if  $(\mathcal F,\Delta)$ is a F-dlt foliation on a normal projective variety $X$, then a sequence of $(K_{\mathcal F}+\Delta)$-flips and divisorial contractions is a $(K_{\mathcal F}+\Delta)$-non-positive birational map. 
A {\bf minimal model} of $(\mathcal F,\Delta)$ is a $(K_{\mathcal F}+\Delta)$-non-positive birational map $f\colon X\dashrightarrow X'$ such that if $\mathcal F'$ is the transformed foliation on $X'$ 
and $\Delta'=f_*\Delta$, then 
\begin{enumerate}
\item $X'$ is $\mathbb Q$-factorial and klt and $\cal F'$ is non-dicritical, 
\item   $(\mathcal F',\Delta')$ is F-dlt and $K_{\mathcal F'}+\Delta'$ is nef, and 
\item if $E$ is a $f^{-1}$-exceptional divisor on $X'$ then $E$ is $\cal F'$-invariant and $a(E,\mathcal F)=0$. 
\end{enumerate}

%
%

\begin{lemma}\label{l_mmpws}
Let $\mathcal F$ be a co-rank one foliation with non-dicritical singularities on a normal $\bb Q$-factorial projective threefold $X$. 
Let $\Delta$ be a $\bb Q$-divisor such that $(\mathcal F,\Delta)$ is a F-dlt pair. Let $A\ge 0$ and $B\ge 0$ be $\bb Q$-divisors such that $\Delta\sim_{\bb Q} A+B$, $A$ is ample and $(\mathcal F,A+B)$ is F-dlt. Assume that $H\ge 0$ is a $\bb Q$-divisor such that $K_\mathcal F+\Delta+H$ is nef and  
\[
K_{\mathcal F}+\Delta\sim_{\mathbb R}D+\alpha H.
\]
where $\alpha \ge 0$, and $D\ge 0$ is a $\mathbb R$-divisor whose support is a union of lc centres of $(\mathcal F,\Delta)$. 

Then there exists a birational contraction $f\colon X\dashrightarrow Y$ which is a minimal model for $(\mathcal F,\Delta)$. 
\end{lemma}
\begin{proof}
Note that, by Lemma \ref{singcomparison}, $X$ is klt. 
Let 
\[
\lambda =\inf \{ t>0\mid K_{\mathcal F}+ \Delta+tH \text { is nef}\}. 
\]
If $\lambda =0$, then $K_{\mathcal F}+\Delta$ is nef and there is nothing to prove. Otherwise, by 
Lemma \ref{l_lambda},  there exists a curve $\xi$ in $X$ such that $R=\bb R_+[\xi]$ is an extremal ray of $\overline {NE(X)}$ satisfying: 
\[
(K_{\mathcal F}+\Delta)\cdot \xi<0 \qquad\text{and}\qquad (K_{\mathcal F}+\Delta+\lambda H)\cdot \xi=0. 
\]
Note that, since $(D+\alpha H)\cdot \xi<0$, $\alpha\ge 0$ and $H\cdot \xi>0$, it follows that $\xi$ is contained in the support of $D$ and, in particular, $\xi$ intersects an lc centre of $(\mathcal F,\Delta)$. 

By Theorem \ref{flipsexist} and 
Theorem \ref{contractionsareextremal}, 
$R$ defines a  divisorial contraction or a flip $\phi\colon X\dashrightarrow X'$ and $X'$ is klt and $\mathbb Q$-factorial. Let $\mathcal F'$ be the transformed foliation on $X'$ and let $\Delta'$, $H'$  and $D'$ be the image in $X'$ of $\Delta$, $H$ and $D$ respectively. It follows that $K_{\mathcal F'}+\Delta'+\lambda H'$ is nef. 
By Lemma \ref{l_mmp-fdlt}, $(\mathcal F',\Delta')$ is F-dlt and by Lemma \ref{l_image_cont_non-dicrti}, $\cal F'$ is non-dicritical. 

By Lemma \ref{l_A+B}, there exist $\bb Q$-divisors $A'\ge 0$ and $B'\ge 0$ such that $\Delta'\sim_{\bb Q}A'+B'$, $A'$ is ample and $(\mathcal F', A'+B')$ is F-dlt.  

Thus, we may replace $X, \Delta,\cal F, D, H$ and $\alpha$ by $X', \cal F', \Delta', D', \lambda H'$ and $\alpha/ \lambda$ respectively and we proceed as above. 
Theorem \ref{t_spter} implies that, after finitely many steps, we obtain a minimal model of $(\mathcal F,\Delta)$. 
\end{proof}

\begin{lemma}\label{l_existence-of-minmod}
Let $\mathcal F$ be a non-dicritical co-rank one foliation on a 
smooth projective threefold $X$. 
Let $\Delta=A+B$ be a $\bb Q$-divisor such that 
$(\mathcal F,\Delta)$ is a F-dlt pair, 
$A\ge 0$ is an ample $\bb Q$-divisor and $B\ge 0$. Assume that there exists a $\bb Q$-divisor $D\ge 0$ such that 
\begin{enumerate}
\item $K_{\mathcal F}+\Delta\sim_{\mathbb Q}D$,
\item \label{i_logsmooth} $(\mathcal F,\Delta+D)$ is log smooth, and 
\item \label{i_component} any component of $D$ is either semi-ample or it is contained in the stable base locus of $D$. 
\end{enumerate}

Then there exists a birational contraction $f\colon X\dashrightarrow Y$  which is a minimal model for $(\mathcal F,\Delta)$. 
\end{lemma}
\begin{proof}
We may write $D=D_1+D_2$ where $D_1,D_2\ge 0$ and  the components  of $D_1$ are exactly the components of $D$ which are lc centres of $(\mathcal F,\Delta)$. Note that, in particular, $D_1$ contains all the components of $D$ which are $\mathcal F$-invariant. 
Let $k$ be the number of components of $D_2$. We proceed by induction on $k$. 

If $k=0$, then $D_2=0$ and the support of $D$ is a union of lc centres of $(\mathcal F,\Delta)$. Let $H$ be a sufficiently ample $\bb Q$-divisor such that
 $K_{\mathcal F}+\Delta+H$ is ample. Then Lemma \ref{l_mmpws} implies that 
 there exists a birational contraction $f\colon X\dashrightarrow Y$ which is a minimal model for $(\mathcal F,\Delta)$. 
 
We now assume that $k>0$. Let 
\[
\lambda = \sup\{t\ge 0\mid (\mathcal F,\Delta+tD_2) \text{ is } F\text{-dlt}\}
\]
By Item (2), it follows that  $\lambda>0$ and $(\cal F,\Delta+\lambda D_2)$ is F-dlt.
Note that $\lambda\in \mathbb Q$. 
 Moreover, we have
\[
K_{\mathcal F}+\Delta+\lambda D_2\sim_{\mathbb Q} D+\lambda D_2
\]
By induction, it follows that $(\mathcal F,\Delta+\lambda D_2)$ admits a minimal model $X\dashrightarrow X'$, which is a birational contraction. Let $\mathcal F'$ be the transformed foliation on $X'$ and let $\Delta', D', D'_1$ and $D'_2$ be the image of $\Delta, D, D_1$ and $D_2$ on $X'$ respectively. Let $H'=\lambda D'_2$. Then $K_{\mathcal F'}+\Delta'+ H'$ is nef and 
\[
K_{\mathcal F'}+\Delta'\sim_{\mathbb Q} D'_1+ \frac 1 \lambda H'.
\]
Thus, Lemma \ref{l_mmpws} implies that 
there exists a birational contraction $X'\dashrightarrow Y$ which is a minimal model for $(\mathcal F',\Delta')$.

Let $f\colon X\dashrightarrow Y$ be the induced map. Note that $f$ is a birational contraction. 
In order to show that $f\colon X\dashrightarrow Y$ is a minimal model for $(\mathcal F,\Delta)$, it is enough to show that $f$ is $(K_{\mathcal F}+\Delta)$-non-positive. Let $\mathcal G$ be the transformed foliation on $Y$ and let $\Gamma=f_*\Delta$. By Lemma \ref{l_A+B}, there exists $\bb Q$-divisors $A'\ge 0$ and $B'\ge 0$ such that $\Gamma\sim_{\bb Q}A'+B'$, $A'$ is ample and $(\mathcal F',A'+B')$ is F-dlt. 
Thus,  Theorem \ref{t_bpf} implies that $K_{\mathcal G}+\Gamma$ is semi-ample. 

Let  $p\colon W\to X$ and $q\colon W\to Y$ be a  resolution of indeterminacy of $f$. Then, we may write 
\[
p^*(K_{\mathcal F}+\Delta)+F=q^*(K_{\mathcal G}+\Gamma)+E
\]
where $E,F\ge 0$ are $q$-exceptional $\bb Q$-divisors  without any common component. Since $K_{\mathcal G}+\Gamma$ is semi-ample, it follows that the stable base locus of $q^*(K_{\mathcal G}+\Gamma)+E$ coincides with the support of $E$. Let us assume that $F\neq 0$. Then, we claim that there exists a component $S$ of $F$ which is contained in the stable base locus of $p^*D+F$. 
Indeed either there exists a component $S$ of $F$ which is $p$-exceptional and the claim follows immediately or the image $T$ of a component $S$ of $F$  in $X$ is $f$-exceptional. In particular, $T$ is contained in the support of $D$ and, by Item \ref{i_component}, $T$ is contained in the stable base locus of $D$. 
It follows that $S$ is contained in the stable base locus of $p^*D+F$. Thus, $S$ is a component of $E$, a contradiction. It follows that $F=0$ and, in particular, $f$ is $(K_{\mathcal F}+\Delta)$-non-positive. Thus, $f\colon X\dashrightarrow Y$ is a minimal model for $(\mathcal F,\Delta)$. 
\end{proof}

\begin{theorem}\label{t_existence-minmod}
Let $\mathcal F$ be a co-rank one foliation with non-dicritical singularities on a  $\bb Q$-factorial projective threefold $X$. 
Let $\Delta=A+B$ be a $\bb Q$-divisor such that $(\mathcal F,\Delta)$ is a F-dlt pair, 
$A\ge 0$ is an ample $\bb Q$-divisor and $B\ge 0$. Assume that there exists a $\bb Q$-divisor $D\ge 0$ such that 
$K_{\mathcal F}+\Delta\sim_{\mathbb Q}D$. 

Then $(\mathcal F,\Delta)$ admits a minimal model. 
\end{theorem}

\begin{proof}
By Lemma \ref{l_exc}, after possibly replacing $A$ by a $\bb Q$-equivalent divisor, we may assume that
 $\lfloor \Delta \rfloor = 0$ and that, for any exceptional divisor $E$ over $X$, if $a(E,\cal F,\Delta)=-\epsilon(E)$ then $E$ is invariant and $a(E,\cal F)=a(E,\cal F,\Delta)=0$. 
 
By \cite[Proposition 3.5.4]{BCHM06}, we may find a positive integer $m$ and $\bb Q$-divisors $P\ge 0$ and $N\ge 0$ such that $P+N\sim_{\bb Q}D$ and any component of $N$ is  contained in the stable base locus of $P+N$, whilst every component $\Sigma$ of $P$   is  such that $m\Sigma$ is mobile. Let $\pi\colon Z\to X$ be a foliated log resolution of $(\cal F,\Delta+P+N)$ which also resolves the base locus of $|m\Sigma|$ for any component $\Sigma$ of $P$.
Let $\mathcal G$ be the transformed foliation on $Z$. We may write 
\[
K_{\mathcal G}+\Delta_Z=\pi^*(K_{\mathcal F}+\Delta)+F
\]
for some $\bb Q$-divisors $\Delta_Z,F\ge 0$ without common components.
Let $C\ge 0$ be a $\pi$-exceptional $\bb Q$-divisor on $Z$ such that $\pi^*A-C$ is ample. Notice that $\pi^*A-tC$ is ample for any $0<t<1$. 

Thus, there exist $\delta,\epsilon >0$  and a $\bb Q$-divisor 
 $\Gamma\sim_{\bb Q} \Delta_Z -\delta C + \epsilon\sum  E_i$ where the sum is taken over all the non-invariant $\pi$-exceptional divisors and such that 
\begin{enumerate}
\item $\Gamma=A'+B'$ where $A'\ge 0$ is an ample $\mathbb Q$-divsor and $B'\ge 0$, 
\item $(\mathcal G,\Gamma)$ is F-dlt, 
\item we may write 
\[
K_{\mathcal G}+\Gamma\sim_{\mathbb Q}\pi^*(K_{\mathcal F}+\Delta)+F'
\]
where $F'\ge 0$ is a $\pi$-exceptional $\bb Q$-divisor, whose support contains every exceptional divisor $E$ of $\pi$ such that $a(E,\mathcal F)> -\epsilon(E)$, 
\item there exists an effective divisor $D'\sim_{\bb Q} K_{\mathcal G}+\Gamma$ such that any component of $D'$ is either semi-ample or it is contained in the stable base locus of $D'$, and 
\item $(\mathcal G,\Gamma+D')$ is a log smooth foliation.
\end{enumerate}

Lemma \ref{l_existence-of-minmod} implies that $(\mathcal G,\Gamma)$ admits a minimal model $g\colon Z\dashrightarrow Y$, which is a birational contraction. We want to show that the induced map $f\colon X\dashrightarrow Y$ is a minimal model of $(\mathcal F,\Delta)$. 
Let $p\colon W\to Z$ and $q\colon W\to Y$ be proper birational morphisms that resolve the indeterminacy locus of $g$. Let $r\colon W\to X$ be the induced morphism. Since $g$ is $(K_{\mathcal G}+\Gamma)$-non-positive, we may write 
\[
p^*(K_{\mathcal G}+\Gamma) = q^*(K_{\mathcal F'}+\Gamma')+G
\]
where $\mathcal F'$ is the transformed foliation on $Y$, $\Gamma'=g_*\Gamma$ and $G\ge 0$ is $q$-exceptional. On the other hand, we also have 
\[
p^*(K_{\mathcal G}+\Gamma)\sim_{\mathbb Q}r^*(K_{\mathcal F}+\Delta)+p^*F'.
\]

Since $K_{\mathcal F'}+\Gamma'$ is nef, the negativity lemma implies that $G\ge p^*F'$.
In particular, the support of $G$ contains 
every exceptional divisor $E$ of $\pi$ such that $a(E,\mathcal F)> -\epsilon(E)$.
 Thus, if $E'$ is a $f^{-1}$-exceptional divisor on $Y$ then $E'$ is invariant and $a(E',\mathcal F)=-\epsilon(E)=0$.
 Moreover, it follows that $f$ is $(K_{\mathcal F}+\Delta)$-non-positive. Thus, the claim follows. 
\end{proof}

\section{Existence of F-terminalisations}

\begin{theorem}[Existence of F-terminalisations]
\label{t_existenceFterminalization}
Let $\cal F$ be a co-rank one 
foliation on a normal variety $X$ of dimension $\leq 3$.
Let $(\cal F, \Delta)$ be a foliated pair with $\Delta\ge 0$.

Then there exists a birational morphism $\pi\colon Y \rightarrow X$
such that 
\begin{enumerate}
\item  if $\cal G$ is the transformed foliation and $\Delta_Y = \pi_*^{-1}\Delta$, then $(\cal G, \Delta_Y)$
is F-dlt, canonical and terminal along $\sing Y$,

\item $Y$ is klt and $\bb Q$-factorial and

\item $K_{\cal G}+\Delta_Y+E = \pi^*(K_{\cal F}+\Delta)$ where $E \geq 0$.
\end{enumerate}
\end{theorem}
\begin{proof}
We assume that $\dim X = 3$.  The case where $\dim X = 2$ is similar.
Let $\mu\colon W\to X$ be a foliated log resolution of $\cal F$, let $\cal H$ be the pulled back foliation on $W$ and let 
$\Delta' =\mu_*^{-1}\Delta$. Let $A$ be an
ample Cartier divisor on $X$ and let $C\ge 0$ be a $\mu$-exceptional $\bb Q$-divisor on $W$  such that $\mu^*A-C$ is ample. 
Let $0<\delta\ll 1$ such that if 
$\Gamma= \mu^*A-\delta C$
then 
\[
K_{\cal H}+\Delta'+\Gamma+G_1=\mu^*(K_\cal F+\Delta+A) +G_2
\]
where $G_1,G_2\ge 0$ are $\mu$-exceptional $\bb Q$-divisors without any common component and the support of $G_2$ contains all the $\mu$-exceptional divisors with discrepancy greater than zero with respect to $(\cal F,\Delta)$.

For $0<\epsilon \ll 1$ we have that $\mu^*A-\delta C+\epsilon \lfloor \Delta'\rfloor$ is ample. Thus,  by Lemma \ref{l_perturbation} and Remark \ref{r_simple=ndc} we may find $0\le A'\sim_{\bb Q}\Gamma+\epsilon \lfloor \Delta'\rfloor$
such that $(\cal H, \Delta''+A')$ is F-dlt where $\Delta'' \coloneqq \Delta' -\epsilon \lfloor \Delta'\rfloor$.
We claim that if $n\ge 6$ then any $(K_\cal H+\Delta''+A'+n\mu^*A)$-negative extremal ray is generated by a curve which is contracted by $\mu$. 
Indeed let $C$ be a curve spanning a $(K_\cal H+\Delta''+A')$-negative extremal ray and suppose
that $\mu_*C \neq 0$.  On one hand, Theorem \ref{t_cone-dlt} implies that $-(K_\cal H+\Delta''+A')\cdot C \leq 6$, on the other
hand $n\mu^*A\cdot C \geq 6$ and so $(K_\cal H+\Delta''+A'+n\mu^*A)\cdot C \geq 0$, proving our claim.

By choosing $n$ sufficiently large, we may also assume that there exists a $\bb Q$-divisor $D\ge 0$ on $W$ 
such that $D\sim_{\mathbb Q}K_{\cal H}+\Delta''+A'+n\mu^*A$. 
By Lemma \ref{l_perturbation}, Remark \ref{r_simple=ndc} and Corollary \ref{c_canonical_perturb}
we may find $0\le A''\sim_{\bb Q}A' + n\mu^*A$  such that  $(\cal H, \Delta''+A'')$ 
is F-dlt and canonical. 
By Theorem \ref{t_existence-minmod}, $K_{\cal H}+\Delta''+A''$ admits a minimal model $f\colon W\dashrightarrow Y$.

We claim that each step in this MMP will be an MMP over $X$.  Indeed, as observed above, the first step
of this MMP must contract only $\mu$-exceptional curves, and so this first step is a step over $X$.
Let $ W' \dashrightarrow W''$ be an intermediate step of this MMP.  By induction, we have a morphism $\mu'\colon W' \rightarrow X$
and, as observed earlier, Theorem \ref{t_cone-dlt} implies that if $f'\colon W\dashrightarrow W'$ is the induced map then 
every $(K_{f'_*\cal H}+f'_*\Delta''+f'_*A'')$-negative extremal ray
is spanned by a curve $C$ such that $(\mu')^*A\cdot C = 0$, and so this step of the MMP will again be a step over $X$.
Let $\pi \colon Y \rightarrow X$ be the induced morphism.



Note that,  Lemma \ref{l_negativity} implies that if  $\cal G$ is the transformed foliation and $\Delta_Y = f_*\Delta' = \pi_*^{-1}\Delta$, 
then  $\cal G$ is  F-dlt and canonical
and, moreover it is terminal along $\sing Y$.
 Moreover, by definition of minimal model, we have that
$Y$ is klt and $\bb Q$-factorial. Finally, $f_*(G_2-G_1)$
is nef over $X$ and $\pi$-exceptional
and so the negativity lemma applies to show that $f_*G_2=0$.
Thus, if $E:= \pi^*(K_{\cal F}+\Delta)-(K_{\cal G}+\Delta_Y)$ then $E \geq 0$.
\end{proof}

\begin{defn}
We call a modification $\pi\colon Y \rightarrow X$ as in Theorem \ref{t_existenceFterminalization}
an {\bf F-terminalisation} for the foliated pair $(\cal F, \Delta)$. 
\end{defn}

\begin{theorem}
\label{t_canimpliesnondicritical}
Let $(\cal F, \Delta)$ be a foliated pair on a projective threefold $X$.
Assume that $\cal F$ is a co-rank one foliation and either
\begin{enumerate}
\item $(\cal F, \Delta)$ is F-dlt or
\item $(\cal F, \Delta)$ is canonical.
\end{enumerate}

Then $\cal F$ has non-dicritical singularities.

Furthermore, if $(\cal F, \Delta)$ is F-dlt and $K_X$ is $\bb Q$-Cartier then $X$ is klt.
\end{theorem}
\begin{proof}
We will only prove the case where $(\cal F, \Delta)$ is F-dlt. The other one may be handled in a similar manner.
Let $\mu\colon W\to X$ be a foliated log resolution of $(\cal F, \Delta)$ 
which only extracts divisors $E$ of discrepancy $-\epsilon(E)$ and let $\cal H$ be the pulled back foliation on $W$.
Our result follows if we can show $\mu^{-1}(P)$ is tangent to $\cal H$ for all $P \in X$.
So suppose for sake of contradiction that there is some $P$ such 
that $\mu^{-1}(P)$ is not tangent to $\cal H$ and let $C\subset \mu^{-1}(P)$ be a general curve transverse
to the foliation.

Write $K_{\cal H}+\Gamma = \mu^*(K_{\cal F}+\Delta)+E$ where $E \geq 0$ is $\mu$-exceptional, $\Gamma \geq 0$ and $\mu_*\Gamma = \Delta$, 
so that $E$ and $\Gamma$ do not have any common component. 
Observe that $\lfloor \Gamma -\mu_*^{-1}\Delta\rfloor = 0$.
Let $A\ge 0$ be an ample divisor on $X$ and let $G\ge 0$ be a $\mu$-exceptional $\bb Q$-divisor on $W$ such that $\mu^*A-G$ is ample. 
Let $F$ be the sum of all the $\mu$-exceptional non-invariant divisors. 
There exist sufficiently small $\epsilon, \delta>0$ such that if $\Theta=\mu^*A-\delta G+\Gamma +\epsilon F$, then we may write
\[
K_\cal H + \Theta+E_1=\mu^*(K_\cal F+\Delta+A)+E_2
\]
where $E_1, E_2\ge 0$ are $\mu$-exceptional $\bb Q$-divisors without common components and such that the  support of $E_2$ contains all the $\mu$-exceptional non-invariant divisors. 

As in the proof of Theorem \ref{t_existenceFterminalization}, by Theorem \ref{t_existence-minmod}, we may run a $(K_{\cal H}+\Theta +n\mu^*A)$-MMP $\phi\colon W\dashrightarrow Y$, where $n$ is sufficiently large so that the induced map $\nu\colon Y\to X$ is a proper morphism. Let $\cal G$ be the transformed foliation on $Y$. Notice that, the negativity lemma  implies that 
$\phi_*E_2=0$ and, in particular, $\phi$ contracts all the non-invariant $\mu$-exceptional divisors. 
Moreover, we have
\[
E_2-E_1=E-\delta G+\epsilon F.
\]
Thus, if $\delta$ is sufficiently small, then the support of $E$ is contained in the support of $E_2$ and therefore $\phi_*E=0$. 
It follows that
\[
K_{\cal G}+\phi_*\Gamma = K_{\cal G}+\nu_*^{-1}\Delta  = \nu^*(K_{\cal F}+\Delta)
\]
and that every $\nu$-exceptional divisor is $\cal G$-invariant.
Since $C$ is transverse to the foliation we have $0 \neq \phi_*C =: C'$ is also transverse to the foliation and so 
is not contained in any $\nu$-exceptional divisor.  Let $A_1$ and $A_2$ be two distinct effective Cartier divisors containing 
$\nu(C') = P$ and write $\nu^*A_i = \nu_*^{-1}A_i + B_i$ where $B_i \geq 0$ is $\nu$-exceptional.
On one hand we know that $B_i\cdot C' >0$, on the other hand we know that $\nu^*A_i\cdot C' = 0$ and so $\nu_*^{-1}A_i \cdot C' <0$.
Let $D = \nu_*^{-1}A_1+\nu_*^{-1}A_2$
and let
\[
\lambda = \sup \{ t \geq 0 | (\cal G, \nu_*^{-1}\Delta+tD) \text{ is log canonical at the generic point of } C' \}.
\]
We claim that $\lambda >0$.  Indeed, suppose that $\lambda = 0$.    In this case $C'$ is an lc centre
of $(\cal G, \nu_*^{-1}\Delta)$, which in turn implies that $\nu(C') = P$ is an lc centre
of $(\cal F, \Delta)$. By Lemma \ref{l_fdlt-logsmooth}, it follows that $\cal F$ has simple singularities near $P$, and so, by  Remark \ref{r_simple=ndc}, we have that $\cal F$
is non-dicritical near $P$, a contradiction.

Thus, by the definition of $\lambda$, it follows that  $C'$ is an lc centre of $(\cal G, \nu_*^{-1}\Delta+\lambda D)$. Moreover, we have 
$(K_{\cal G}+\nu_*^{-1}\Delta+\lambda D)\cdot C' <0$.  
This however is a contradiction of foliation subadjunction, \cite[Theorem 4.5]{Spicer17} which implies that 
$(K_{\cal G}+\nu_*^{-1}\Delta+\lambda D)\cdot C' \geq 0$. 

To see our final claim, since $(\cal F, \Delta)$ is F-dlt, we may find a log resolution $\mu\colon W \rightarrow X$ only extracting divisors $E$ of discrepancy $>-\epsilon(E)$.  We run a 
$(K_{\cal H}+\mu_*^{-1}\Delta+F)$-MMP over $X$, where $F$ is the sum of all the $\mu$-exceptional non-invariant divisors. Note that this MMP
terminates by Corollary \ref{specialcaseMMP}. 
Let $\phi\colon W\dashrightarrow Y$ be the output
of this MMP, with induced morphism $\nu\colon Y\rightarrow X$. Observe that $Y$ is klt. By the negativity lemma, we know that $\nu$
is small and so $K_Y = \nu^*K_X$ which implies
that $X$ is klt.
\end{proof}

\begin{remark}
Theorem \ref{t_canimpliesnondicritical}
shows that the hypothesis of non-dicriticality in the cone theorem (and in the above results) is superfluous.
When $X$ is smooth this result follows from \cite[Proposition 3.11]{LPT11}.
\end{remark}

\begin{proof}[Proof of Theorems \ref{T_flips_exist} and \ref{T_min_model1}] 

Note that Theorem \ref{T_flips_exist} (resp. Theorem \ref{T_min_model1}) 
follows directly 
from Theorem \ref{t_canimpliesnondicritical} and Theorem \ref{flipsexist} (resp. Theorem \ref{t_existence-minmod}).
\end{proof}


\section{Abundance for $c_1(K_{\cal F}+\Delta) = 0$}
\label{s_abundance}

The goal of this section is to prove the following:

\begin{theorem}
\label{threefoldlogabundancezero}
Let $X$ be a projective threefold and $\cal F$ be a co-rank one foliation.
Let $(\cal F, \Delta)$ be a foliated pair with log canonical 
singularities where $\Delta \geq 0$ is a $\bb Q$-divisor  such that  $c_1(K_{\cal F}+\Delta) = 0$.

Then $\kappa(K_{\cal F}+\Delta) = 0$.
\end{theorem}

As we mentioned in the Introduction, the result above is a consequence of \cite[Theorem 2]{LPT11} in the case of foliations with canonical singularities defined on a smooth projective variety. See also \cite[Theorem 1.3]{Druel21} for results in this direction on singular varieties,
as well as the recent preprint \cite{DruelOu19}, which achieves the main result of Section \ref{ss_d0Fcan}
in all dimensions.

\subsection{$\Delta \neq 0$ or $\cal F$ is log canonical
but not canonical}

\begin{lemma}
\label{logabundancenumtrivial}
Suppose $X$ is a klt projective surface and $\cal F$ is a rank one foliation on $X$.
Suppose that $(\cal F, \Delta)$ is lc where $\Delta \geq 0$ is a $\bb Q$-divisor and $c_1(K_{\cal F}+\Delta)= 0$.

Then $\kappa(K_{\cal F}+\Delta) = 0$.
\end{lemma}
\begin{proof}
Without loss of generality we may replace $(\cal F, \Delta)$ by an F-dlt modification.
In particular, we may assume that $\cal F$ has canonical singularities.

If $\Delta = 0$, then our claim follows from \cite[Lemma IV.3.1]{McQuillan08}.

If $\Delta \neq 0$, then $K_{\cal F}$ is not pseudo-effective.
Running an MMP for $K_{\cal F}$ with scaling of some ample divisor, and replacing $\cal F$
by this output we may assume that we have a $\bb P^1$-fibration
$f\colon X \rightarrow C$ such that $\cal F$ is induced by the fibration.

By Lemma \ref{canonicalbundleformula} we see that $K_{\cal F}+\Delta \sim_{\bb Q} f^*\Theta$
where $\Theta \geq 0$ and our result is proven.
\end{proof}

\begin{lemma}
\label{algintcase}
Let $X$ be a normal projective threefold and $\cal F$ be a co-rank one foliation on $X$.
Suppose that $\cal F$ is algebraically integrable, 
$(\cal F, \Delta)$ is log canonical where $\Delta \geq 0$ is a $\bb Q$-divisor and $c_1(K_{\cal F}+\Delta) = 0$.

Then $\kappa(K_{\cal F}+\Delta) = 0$.
\end{lemma}
\begin{proof}
By Theorem \ref{t_existencefdlt} we may replace $(\cal F, \Delta)$ by an F-dlt modification and so we may assume
without loss of generality that $X$ is $\bb Q$-factorial and klt.

By assumption $\cal F$ admits a meromorphic first integral $f\colon X \dashrightarrow C$ where $C$ is a smooth proper
curve.
Let $\mu \colon X' \rightarrow X$ be a resolution of indeterminacies of $f$ and let $f'\colon X' \rightarrow C$
be the resolved map. Observe that $f'$ is a holomorphic first integral of $\cal F'$, the pull back
of $\cal F$ on $X'$.

As $(\cal F, \Delta)$ is F-dlt 
it follows that $\cal F$ has non-dicritical singularities by Theorem \ref{t_canimpliesnondicritical}, and so
if $p \in X$ then $\mu^{-1}(p)$ is tangent to $\cal F'$. Since $f'$ is a 
holomorphic first integral of $\cal F'$ this implies that $f'(\mu^{-1}(p))$ is a single point
and so $f'$ contracts every fibre of $\mu$.  The rigidity lemma then implies
that in fact $f \colon X \rightarrow C$ is a morphism.

We apply Lemma \ref{canonicalbundleformula} to write $K_{\cal F}+\Delta \sim_{\bb Q} f^*\Theta$
for some $\bb Q$-divisor $\Theta \geq 0$ and we can conclude.
\end{proof}

\begin{proposition}
\label{KFlogcanonical}
Let $X$ be a normal projective threefold and $\cal F$ be a co-rank one foliation.
Let $(\cal F, \Delta)$ be a foliated pair with log canonical 
singularities and $\Delta \geq 0$ is a $\bb Q$-divisor.
Suppose that $c_1(K_{\cal F}+\Delta) = 0$
and that either $\Delta \neq 0$ or $\cal F$ is log canonical
but not canonical.

Then $\kappa(K_{\cal F}+\Delta) = 0$.
\end{proposition}
\begin{proof}
 By Theorem \ref{t_existenceFterminalization},
we may replace $(\cal F, \Delta)$ with a F-terminalisation,
so we may assume without loss of generality 
that $(\cal F, \Delta)$ is F-dlt and canonical 
with $\Delta \neq 0$
and that $X$ is $\bb Q$-factorial and klt.

In this case, $K_{\cal F}$ is not pseudo-effective and hence it is uniruled by Lemma \ref{lem_KF_not_psef_uniruled}.  So there exists a diagram
\begin{center}
\begin{tikzcd}
W \arrow{r}{p} \arrow{d}{q} & X \\
B & \\ 
\end{tikzcd}
\end{center}
where $q \colon W \rightarrow B$ parametrizes a dominant family of rational curves tangent to $\cal F$.

We may obviously assume that $\text{dim}(B) = 2$. 
Let $G$ be an ample divisor on $B$.
Let $b \in B$ be general and let $W_b$ be the fibre over $b$.  There are two cases, either 
\begin{enumerate}
\item \label{i_inf} $p(W_b) \cap p(W_{b'}) \neq \emptyset$
for infinitely many $b'$ or
\item \label{i_empty} $p(W_b) \cap p(W_{b'}) = \emptyset$ for $b \neq b'$.
\end{enumerate}

Suppose that we are in case \eqref{i_inf}, and let $Z_b \subset B$ be the closed subset of $B$ parametrizing those points $b'$ with 
$p(W_{b'}) \cap p(W_{b}) \neq \emptyset$.   Set $S = p(q^{-1}(Z_b))$.  
Observe that $\cal F$ has non-dicritical singularities by Theorem \ref{t_canimpliesnondicritical}.
We claim that $S$ is $\cal F$-invariant.  Supposing the claim we see that a general leaf of $\cal F$ is algebraic and we may conclude by Lemma \ref{algintcase}.
To see the claim, without loss of generality we may assume that $p(W_b)$ and $p(W_{b'})$ all pass through the same point $x \in X$.
Let $m\colon X' \rightarrow X$ be a resolution of singularities of $X$ such that $m^{-1}(x)$ is a divisor, let $S' = m_*^{-1}S$ 
and let $\Sigma_b$ and $\Sigma_{b'}$ be the strict
transforms of $p(W_b)$ and $p(W_{b'})$ respectively.  
Observe that $\Sigma_{b'} \cap m^{-1}(x) \subset \sing \cal F'$ where $\cal F'$ is the pulled back foliation.
Thus, $S'$ is covered by curves tangent to $\cal F'$ which pass through $\sing \cal F'$ which by 
\cite[Th\'eor\`eme 4]{CM92} and the fact that $\cal F'$ is non-dicritical implies that $S'$ must be $\cal F'$-invariant, as required.

So suppose we are in case \eqref{i_empty}.  Note that the rational map $q\circ p^{-1}\colon X \dashrightarrow B$ is almost proper, 
and in particular, we see that 
\[
p(W_b)\cdot p_*q^*G = 0
\]
where $G$ is an ample Cartier divsior on $B$.
Let $A$ be an ample divisor on $X$. 
Given a sufficiently large positive integer $m$, we may run a $K_X$-MMP with scaling of $A+p_*q^*(mG)$
Denote this MMP by $\phi\colon X \dashrightarrow X'$, and let $(\cal F',\Delta')$ be the transformed foliated pair on $X'$.
Each step of this MMP is $(K_{\cal F}+\Delta)$-trivial and so we see that $(\cal F', \Delta')$ is lc, 
$c_1(K_{\cal F'}+\Delta') = 0$ and that $\kappa(K_{\cal F'}+\Delta') = 0$ implies that 
$\kappa(K_{\cal F}+\Delta) = 0$.  

Observe that $K_X\cdot p(W_b) <0$.  If we choose $m$ larger than $6m'$ where $m'$ is the Cartier index of $p_*q^*G$ 
then, arguing as in the proof of Theorem \ref{t_existenceFterminalization}, each step of this MMP will be $p_*q^*mG$ trivial and so 
this MMP must terminate in a Mori fibre space $f\colon X' \rightarrow S$ such that $S$ is a klt surface and $f$ contracts the strict transform of $p(W_b)$. Thus, 
 the fibration $f$ is tangent to $\cal F'$.
By Lemma \ref{canonicalbundleformula},
there is a foliation $\cal G$ on $S$ so
that $\cal F' = f^{-1}\cal G$, 
a pseudo-effective divisor $D$ and divisor
$\Theta \geq 0$ such that $K_{\cal F'}+\Delta' \sim_{\bb Q} f^*(K_{\cal G}+\Theta+D)$
and $(\cal G, \Theta)$ is lc.  

Suppose for sake of contradiction that $D \neq 0$.  Since $c_1(K_{\cal G}+\Theta) = -D$ it follows that 
$\cal G$ is uniruled by Lemma \ref{lem_KF_not_psef_uniruled}, in particular, $\cal G$ is algebraically integrable which implies the same holds
for $\cal F$.
Since $f$ is $K_{X'}$-negative,
we see that, in addition, $S$ is klt.
Thus, we can apply Lemma \ref{logabundancenumtrivial} to conclude
that $K_{\cal G}+\Theta$, and hence $K_{\cal F'}+\Delta'$, is torsion.
\end{proof}

\subsection{$\Delta = 0$ and $\cal F$ is canonical}
\label{ss_d0Fcan}
In this section $\cal F$ is a co-rank one foliation on a normal projective threefold with
$c_1(K_{\cal F}) = 0$.  Suppose that $\cal F$ has canonical singularities.

\begin{lemma}
\label{l_replacebyFterminalization}
We may freely replace $\cal F$ by an F-terminalisation.  Thus we may 
assume that $X$ is klt and $\bb Q$-factorial and $\cal F$ is terminal along $\sing X$. In particular, $\sing X$
is tangent to $\cal F$.
\end{lemma}
\begin{proof}
Let $\pi\colon Y\to X$  be an F-terminalisation, whose existence is guaranteed by Theorem \ref{t_existenceFterminalization} and let $\cal G$ be the pulled back foliation on $Y$.
By definition, $K_{\cal G}+F = \pi^*K_{\cal F}$ where $F \geq 0$.
On the other hand, since $\cal F$ is canonical, 
$K_{\cal G} = \pi^*K_{\cal F}+E$ where $E \geq 0$.
Thus $E = F = 0$ and so $c_1(K_{\cal G})=0$.
Furthermore, if $K_{\cal G}$ is torsion then so is $K_{\cal F}$. The last claim follows from \cite[Lemma 8.6]{Spicer17} together with Theorem \ref{t_canimpliesnondicritical}.
\end{proof}

\begin{lemma}
\label{KXMMPreduction}
Let $\phi\colon X \dashrightarrow X'$ be a sequence of steps of a  $K_X$-MMP such that $\phi$ is birational and let $\cal F'$ be the
transformed foliation on $X'$.

Then
\begin{enumerate}
\item $c_1(K_{\cal F'}) = 0$;

\item $X'$ has klt and $\bb Q$-factorial singularities,

\item $\cal F'$ has canonical singularities, and

\item \label{i_tang} $\sing X'$ is tangent to $\cal F'$.
\end{enumerate}

Moreover, if $K_{\cal F'}$ is torsion then so is $K_{\cal F}$.
\end{lemma}
\begin{proof}
Each step of the MMP is $K_{\cal F}$-trivial so we see that
$\cal F'$ has canonical singularities and $c_1(K_{\cal F'}) = 0$.
Furthermore, we see that $K_{\cal F} =\phi^*K_{\cal F'}$ and so if $K_{\cal F'}$
is torsion then so is $K_{\cal F}$.

Since $X$ has klt and $\bb Q$-factorial singularities it follows that
$X'$ has klt and $\bb Q$-factorial singularities.

By Theorem \ref{t_canimpliesnondicritical} at each step of the MMP the transformed foliation has non-dicritical
singularities and so we see that only curves tangent to $\cal F$ are contracted by the MMP.
In particular, we see that the flipping and flipped loci are all tangent to the foliation.

To prove Item (\ref{i_tang}), consider a step in the $K_X$-MMP, call it $f \colon Y \dashrightarrow W$ and let $\cal F_Y$ and $\cal F_W$ be the transformed 
foliations on $Y$ and $W$ respectively.
We claim that if $f$ is a divisorial contraction, then $\text{exc}(f)$ is foliation invariant.  Indeed suppose not.
By our above observation, $f$ contracts a divisor $E$ transverse to the foliation
to a curve $C$ and such that the foliation restricted to $E$ must be tangent to the
fibration $E \rightarrow C$.  Let $F$ be a general fibre of $E \rightarrow C$. By Lemma \ref{singcomparison}
we know that \[0= K_{\cal F_Y}\cdot F = (K_Y+E)\cdot F\]
a contradiction of the fact that $f$ is a $K_Y$-negative contraction.

Thus, all divisorial contractions in the MMP only contract invariant divisors and so by Lemma \ref{l_replacebyFterminalization}
we may conclude that
$\sing X'$ is indeed tangent to $\cal F'$.
\end{proof}

\begin{lemma}
\label{notuniruled}
$\cal F$ is not uniruled.
\end{lemma}
\begin{proof}
The proof of \cite[Theorem 3.7]{LPT11} works equally well in the case where $X$ is singular.
\end{proof}

\begin{lemma}
\label{tosurfacefibrestangent}
Suppose we have a morphism $f\colon X \rightarrow S$ where $S$ is a surface with klt singularities.
Suppose furthermore that $K_{\cal F} \sim_{\bb Q, f} 0$ and the fibres of $f$ are tangent to $\cal F$.

Then $\kappa(K_{\cal F}) = 0$.
\end{lemma}
\begin{proof}
By Lemma \ref{canonicalbundleformula},
there is a foliation $\cal G$ on $S$ so
that $\cal F = f^{-1}\cal G$, a pseudo-effective divisor $D$ and divisor
$\Theta \geq 0$ such that $K_{\cal F} \sim_{\bb Q} f^*(K_{\cal G}+\Theta+D)$
and $(\cal G, \Theta)$ is lc.  
It suffices to show that $K_{\cal G}+\Theta$ is torsion.

Suppose for sake of contradiction that $D \neq 0$.  Since $c_1(K_{\cal G}+\Theta) = -D$ it follows that 
$\cal G$ is uniruled, by Lemma \ref{lem_KF_not_psef_uniruled}. In particular, $\cal G$ is algebraically integrable, which implies the same holds
for $\cal F$.

We may freely replace $(\cal G, \Theta)$ by an F-dlt modification and so we may assume without loss of
generality that $S$ is klt. Thus,  so we may 
 apply Lemma \ref{logabundancenumtrivial} to conclude
that $K_{\cal G}+\Theta$ is torsion.
\end{proof}

\begin{lemma}
\label{maptosurface}
Suppose we have a morphism $f\colon X \rightarrow S$ where $S$ is a surface with klt singularities
and $\kappa(S) \geq 0$.  Suppose moreover that the fibres of $f$ are generically transverse to $\cal F$
and that for all $s \in S$, the fibre $f^{-1}(s)$ does not contain a divisor.

Then $\kappa(K_{\cal F}) = 0$.
\end{lemma}
\begin{proof}
Let $S^\circ = S \setminus \sing S$ and let $X^\circ = f^{-1}(S^\circ)$.

The pull back of a pluri-canonical form 
on $S^\circ$ restricts to a non-zero form on the leaves
of $\cal F\vert_{X^\circ}$, i.e., we have a non-zero map 
\[
H^0(S^\circ, \cal O_S(mK_S)) \rightarrow H^0(X^\circ, \cal O_{X^\circ}( mK_{\cal F}))
\]
for all $m\ge 0$.
Since the complement of $X^\circ$ in $X$ has codimension at least two, 
we see that $H^0(X^\circ, \cal O_{X^\circ}( mK_{\cal F})) \cong H^0(X, \cal O_X( mK_{\cal F}))$.
By assumption, we have that  $H^0(S, \cal O_S(mK_S))\neq 0$ for some $m$ sufficiently divisible, and our result follows.
\end{proof}

We will need the following definition, \cite[Definition 8.1]{Touzet16}.

\begin{defn}
Let $X$ be a projective manifold and let $\cal F$ be a co-rank one foliation on $X$.
Let $H_1, ..., H_p$ be $\cal F$-invariant hypersurfaces such that $\sum H_i$ is a normal crossings divisor.
We say that $\cal F$ is of {\bf KLT type}
with respect to $H_1, ..., H_p$
if there exist rational numbers $0 \leq a_i <1$ such that
\[
N^*_{\cal F}+\sum a_iH_i
\]
is pseudo-effective.
\end{defn}

We will also need the following theorem, due to Touzet.

\begin{theorem}
\label{touzetstheorem}
Let $X$ be a projective manifold and let $\cal F$ be a co-rank one foliation on $X$.
Let $H_1, ..., H_p$ be $\cal F$-invariant hypersurfaces such that $\sum H_i$ is a normal crossings divisor.
Suppose that $\cal F$ is of KLT type with respect to $H_1, ..., H_p$.

Then either
\begin{enumerate}
\item $\kappa(N^*_{\cal F}+\sum H_i) = \nu(N^*_{\cal F}+\sum H_i) \geq 0$ or

\item $\kappa(N^*_{\cal F}+\sum H_i) = -\infty$, $\nu(N^*_{\cal F}+\sum H_i) =1$ and 
there exists a holomorphic map $\Psi\colon X \rightarrow \ff h$ where $\ff h = \bb D^n/\Gamma$ is a quotient
of a polydisc by an irreducible lattice $\Gamma \subset ({\rm Aut~} \bb D)^n$
and $\cal F = \Psi^{-1}\cal H$ where $\cal H$ is one of the tautological foliations on $\ff h$.
\end{enumerate}
\end{theorem}
\begin{proof}
This is \cite[Theorem 6]{Touzet16} and \cite[Theorem 9.7]{Touzet16}.
\end{proof}

We will also need the following classification theorem on surface foliations:

\begin{theorem}
\label{t_triv_surface_foliation} Let $X$ be a normal projective surface and let $\cal L$ be a rank one foliation on $X$  with canonical foliation singularities.  Suppose $c_1(K_{\cal L}) = 0$.

Then there exists a finite cover $\tau\colon \widetilde{X} \rightarrow X$ and a birational morphism $\mu \colon \widetilde X\rightarrow Y$ such that if $\widetilde{\cal L}$ and $\cal G$ are the transformed foliations on $\tilde X$ and $Y$ respectively, then $\tau$ is 
ramified along $\widetilde{\cal L}$-invariant divisors and $\mu$ 
contracts $K_{\widetilde{\cal L}}$-trivial rational curves tangent to $\widetilde{\cal L}$.

Moreover,  one of the following holds:
\begin{enumerate}
\item \label{i_sesqui} $X = C \times E/G$ where $g(E) =1$, $C$ is a smooth projective curve, $G$ is a finite group
acting on $C \times E$ and $\cal G$ is the foliation induced by the $G$-invariant fibration $C\times E \rightarrow C$;

\item \label{i_kron} $\cal G$ is a linear foliation on the abelian surface $Y$;

\item \label{i_ricatti} $Y$ is a $\bb P^1$-bundle over an elliptic curve and $\cal G$ is transverse
to the bundle structure and leaves at least one section invariant; or

\item \label{i_qtoric} $Y \cong \bb P^2, \bb P^1 \times \bb P^1$ or $\bb F_n$ (the $n$-th Hirzebruch surface)
and $\cal G$ admits at least 3 invariant rational curves.  If $Y \cong \bb P^1 \times \bb P^1$ or $\bb F_n$ 
then at least 2 of these invariant curves must be fibres of a $\bb P^1$-bundle structure on $Y$.
\end{enumerate}
\end{theorem}
\begin{proof}
This follows directly from \cite[Theorem IV.3.6]{McQuillan08}.
\end{proof}

\begin{lemma}
\label{KXpsefcase}
Suppose that $\kappa(X) \geq 0$.  Then $\kappa(K_{\cal F}) = 0$.
\end{lemma}
\begin{proof}
We have an equality of divisors $K_{\cal F}+N^*_{\cal F} = K_X$,
in particular we see that $c_1(N^*_{\cal F}) = c_1(K_X)$.

If $\kappa(X) = 3$, then $\kappa(N^*_{\cal F}) = 3$, a contradiction
of the Bogomolov-Castelnuovo-De Franchis inequality (see \cite[Theorem 7.2]{GKKP11} for the proof of this
statement in the singular setting). Thus, we may assume that $\kappa(X) \leq 2$.
Moreover, by Lemma \ref{KXMMPreduction} we may assume that $K_X$ is nef.

We distinguish two cases. We first assume that $\kappa(N^*_{\cal F}) = -\infty$.
Let $\mu\colon Y \rightarrow X$ be a log resolution of $X$ and let
$E$ be the reduced divisor whose support coincides with the $\mu$-exceptional divisor and $\cal G$ be the pulled back foliation.  
By Lemma \ref{KXMMPreduction} Item (\ref{i_tang}), $\sing X$ is tangent to $\cal F$ and so
$E$ is $\cal G$-invariant.
Notice that $\cal O_Y(N^*_{\cal G}+E)$ is the saturation
of the image of $\mu^*\cal O_X(N^*_{\cal F})$ in $\Omega^1_Y(\log E)$.
Since $\mu_*E = 0$ and $\mu_*N^*_{\cal G} = N^*_{\cal F}$, we have that
$\kappa(N^*_{\cal G}+E) = -\infty$.

Write $N^*_{\cal G}+E_1 \sim_{\bb Q} \mu^*N^*_{\cal F}+E_2$ where $E_i \geq 0$ and is $\mu$-exceptional. 
We claim
that $\lfloor E_1 \rfloor = 0$. Indeed, this is a local problem on both $X$ and $Y$, 
so perhaps shrinking both we may assume that $E_1 = aE$ for some rational number $a>0$ so that $E_1$ consists of a single component
and  $N^*_{\cal G} \sim_{\bb Q} \mu^*N^*_{\cal F}-aE$.
Consider the following commutative diagram
\begin{center}
\begin{tikzcd}
Y' \arrow{d}{\nu} \arrow{r}{\tau} & Y \arrow{d}{\mu}\\
X' \arrow{r}{\sigma} & X \\
\end{tikzcd}
\end{center}
where $\sigma$ is the index one cover associated to $N^*_{\cal F}$.
Let $\cal F'$ be the pull back of $\cal F$ along $\sigma$ and let $\cal G'$ be the pull
back of $\cal G$ along $\tau$.
Assume that the ramification index along $E$ is $r$.
Since $X'$ is klt, by \cite{GKKP11} we have a morphism $\nu^*(\Omega^{[1]}_{X'}) \rightarrow \Omega^1_{Y'}$
and so 
\[N^*_{\cal G'}= \nu^*N^*_{\cal F'}+cE'\]
where $c \geq 0$.
On the other hand, since $E'$ is invariant
\[N^*_{\cal G'} = \tau^{*}(N^*_{\cal G})+(r-1)E' \sim_{\bb Q, \nu} (-ra+(r-1))E'.\]
Thus, $-ra+(r-1) = c \geq 0$
which implies that $a\leq \frac{r-1}{r} <1$, as claimed.

Since $N^*_{\cal G}+E_1$ is pseudo-effective and $\lfloor E_1 \rfloor = 0$, it follows that 
$\cal G$ is 
a KLT type foliation.

Since $\kappa(N^*_{\cal G}+E) = -\infty$, Theorem \ref{touzetstheorem} implies 
that $\nu(N^*_{\cal G}+E) = 1$.  Since $\cal F$ has canonical singularities, we may write
$K_{\cal G} \equiv F$ where $F \geq0$ is $\mu$-exceptional. Since $\lfloor E_1 \rfloor = 0$, it follows that 
\[
K_Y+E = N^*_{\cal G}+E+K_{\cal G} \equiv N^*_{\cal G}+E+F
\]
has numerical dimension, and hence Kodaira dimension, equal to one,
which in turn implies that $\kappa(K_X) = 1$.

Let $f\colon X \rightarrow C$ be the Iitaka fibration associated to $K_X$. If $\cal F$ is the foliation induced by $f$ then we may
conclude by Lemma \ref{algintcase}.  So suppose for sake of contadiction that 
the general fibre of $f$ is transverse to $\cal F$.
We therefore get an exact sequence
\[0 \rightarrow \cal L \rightarrow \cal F \rightarrow f^*T_C \otimes I_Z \rightarrow 0 \]
where $\cal L$ is the foliation by curves coming from the intersection of $\cal F$ and $T_{X/C}$
and $Z$ is a subscheme supported on $\sing X\cup\sing \cal F$ and on the locus
where $\cal F$ is tangent to the fibres of $X \rightarrow C$.

Let $X_p$ be the fibre over $p\in C$,
and let $\cal F_p$ be the foliation restricted to $X_p$.  For general $p$ we see that $c_1(K_{\cal F_p}) = 0$ and that $X_p$
is not uniruled. In particular, $\cal F_p$ is canonical. We claim that  $N^*_{\cal F_p}$ is torsion for general $p$.  Indeed, $K_{\cal F_p}\sim_{\bb Q}0$ by Lemma \ref{logabundancenumtrivial} and $K_{X_p}\sim_{\bb Q}0$ and so $N^*_{\cal F_p}\sim_{\bb Q}0$.
Thus, we may find a $\bb Q$-divisor $A$ on $X$ whose support is contained in fibres of $f$, but contains no fibre of $f$
and a divisor $B$ on $C$ such that $f^*B \sim_{\bb Q} N^*_{\cal F}+A$.
However, $K_X \sim_{\bb Q} f^*H$ where $H$ is an ample $\bb Q$-divisor on $C$ and $K_X \equiv N^*_{\cal F}$.  Thus, $A = 0$ and $B \equiv H$. In particular, $B$ is ample and $\kappa(N^*_{\cal F}) = 1$. It follows, by Lemma \ref{l_bcd} below, that  $\cal F$ is algebraically integrable and so we may conclude
by Lemma \ref{algintcase}.

\medskip

We now assume that $\kappa(N^*_{\cal F}) \geq 0$.

If $\kappa(X) = 0$, then since $K_X$ is nef, it follows
that both $N^*_{\cal F}$ and $K_X$, hence $K_{\cal F}$ are torsion, and we are done.

Suppose $\kappa(X) \geq 1$ and let $f\colon X \rightarrow B$ be the canonical map.
By assumption, there exists $D\geq 0$ such that  $N^*_{\cal F} \sim_{\bb Q} D$ and $D$ is numerically equivalent
to $f^*H$ where $H \geq 0$ is ample.  However, this implies that $D$ is actually supported on fibres of
$f$, and is in fact equal to a sum of fibres (with the appropriate multiplicities), and so $D = f^*B$ for some ample divisor $B$.
But this implies that $\kappa(N^*_{\cal F}) \geq 1$.
By Lemma \ref{l_bcd} below, it follows
that $\cal F$ is in fact algebraically integrable and so we conclude by Lemma \ref{algintcase}.
\end{proof}

\begin{lemma}\label{l_bcd}
Let $X$ be a klt projective variety and let $\cal F$ be a co-rank one foliation on $X$ with non-dicritical singularities.
Suppose that $\sing X$ is tangent to $\cal F$, $N^*_{\cal F}$ is $\bb Q$-Cartier and $\kappa(N^*_{\cal F}) = 1$.  

Then $\cal F$ is algebraically integrable.
\end{lemma}
We thank the referee for suggesting the following proof of Lemma \ref{l_bcd}.
\begin{proof}
Let $\pi\colon X' \rightarrow X$ be a log resolution of $X$ and let $\cal F'$ be the pulled back foliation.
Let $E$ be the sum of all the $\pi$-exceptional divisors and observe that $E$ is $\cal F'$-invariant.

By \cite[Lemma 4.6]{Druel21} 
we may find effective $\pi$-exceptional divsiors $E_1$ and $E_2$ such that $E_2$ is $\cal F'$-invariant and 
$\lfloor E_2 \rfloor =0$
such that $\pi^*N^*_{\cal F}+E_1 \sim_{\bb Q} N_{\cal F'}+E_2$.  In particular, $\kappa(N^*_{\cal F'}(E)) =1$.
We may then apply \cite[Lemma 9.2]{Touzet16} to conclude.
\end{proof}

\begin{lemma}
\label{l_rank1_abun}
Let $X$ be a $\bb Q$-factorial projective threefold and
let $\pi\colon X \rightarrow B$ be a fibration over a curve.  
Let $\cal L$ be a rank one foliation on $X$ tangent to the fibres of $X \rightarrow B$.
Suppose that 
\begin{enumerate}
\item $c_1(K_{\cal L}) = 0$;
\item $\cal L$ is not uniruled; and
\item the general fibre of $\pi\colon X \rightarrow B$ does not admit a quasi-\'etale cover by an abelian surface.
\end{enumerate} 

Then $\kappa(K_{\cal L}) = 0$.
\end{lemma}
\begin{proof}
Observe first of all that we may assume that $\cal L$ has canonical singularities, otherwise $\cal L$ would be uniruled,
a contradiction. 

Let $\cal L_p$ be the foliation restricted to $X_p$. 
For a general $p$ we have $K_{\cal L}\vert_{X_p} = K_{\cal L_p}$.

By Theorem \ref{t_triv_surface_foliation}, for a general fibre $X_p$, we may find a finite morphism
$\tau_p\colon \widetilde X_p\rightarrow X_p$,  ramified along foliation invariant divisors such that the pulled back foliation on $\widetilde X_p$ is generated by a global vector field. In particular,  $K_{{\cal L}}\vert_{X_p} \sim_{\bb Q} 0$. This implies
that $K_{{\cal L}} \sim_{\bb Q} \sum a_iF_i$ where $a_i$ are rational numbers and $F_i$ are divisors supported on fibres
of $f$, and are therefore in fact entire fibres (counted with multiplicity) since $c_1(K_{\cal L}) =0$.
 Thus, we may find a finite morphism
$\sigma\colon  \widetilde{X}\rightarrow X$ ramified only along divisors supported on fibres and such that if $\widetilde {\cal L}$ be the pulled back foliation on $\widetilde X$ then $K_{\widetilde {\cal L}}$ is Cartier.
Since all the fibres of $X \rightarrow B$ are $\cal L$-invariant we have that $K_{\widetilde{\cal L}} = \sigma^*K_{\cal L}$.
Thus we may freely
replace $X$ and $\cal L$ by $\widetilde{X}$ and ${\widetilde{\cal L}}$ respectively and we may assume that $K_{\cal L}$
is Cartier.

In particular, around all points of $X$, we have that $\cal L$ is generated by a vector field $v$.  Thus by \cite{BM97}
we may find a resolution of singularities $\mu \colon Y \rightarrow X$ by only blowing up in $\cal L$-invariant centres.
Since $\cal L$ has canonical singularities this implies that $K_{\cal L_Y} = \mu^*K_{\cal L}$ where $\cal L_Y$ is the transformed
foliation.  Thus, we may replace $X$ and $\cal L$ by $Y$ and $\cal L_Y$ respectively and  we may assume that $X$ is smooth.

We first handle the case where $\cal L_p$ is algebraically integrable for general $p$, and hence $\cal L$ is algebraically integrable.
In this case we get a rational map
$g\colon X \dashrightarrow S$ such that the fibres of $g$ induce $\cal L$. 
Again, by \cite{BM97} we may find a resolution of singularities of $g$, $\mu \colon X' \rightarrow X$, 
by blowing up in $\cal L$-invariant centres and let $f'\colon X' \rightarrow S$ be the induced map. Observe
that $\mu^*K_{\cal L} = K_{\cal L'}$ where $\cal L'$ is the transformed foliation and
we may argue as in Lemma \ref{algintcase}
to show that $\kappa(K_{\cal L'}) = 0$ and so $\kappa(K_{\cal L}) = 0$.

By Theorem \ref{t_triv_surface_foliation} we may find a cover $\tau_p\colon \widetilde{X}_p\rightarrow X_p$ ramified along
invariant divisors and a birational contraction $\mu_p\colon \widetilde{X}_p \rightarrow Y_p$ such that
$Y_p$ falls into one the types listed in Theorem \ref{t_triv_surface_foliation}.
We now argue based on which case $Y_p$ in falls into. Note that, since the surfaces appearing in Cases (2), (3) and (4) of  Theorem \ref{t_triv_surface_foliation} are not deformation equivalent, we may consider the four cases separately. 
Observe also that in Case (2) that $\tau_p$ is an isomorphism, unless $\cal L_p$ is algebraically integrable. In Case (2) we also have that 
$\mu_p$ is an isomorphism since the foliation on $Y_p$ is smooth, hence terminal, 
and $c_1(K_{\tau_p^{-1}\cal L_p}) = 0$.

In Case (\ref{i_sesqui}), we see that $\cal L$ is algebraically integrable and so we are done.

Case (\ref{i_kron}) does not occur by assumption.

In Case (\ref{i_ricatti}), we run a $K_X$-MMP over $B$, call it $\phi\colon X  \dashrightarrow X'$ and let $\cal L'$ be the transformed foliation on $X'$.
Each step of this MMP is $\cal L$-trivial. Thus, it suffices to check that $\kappa(K_{\cal L'}) = 0$.
A general fibre is uniruled, but not rationally connected, and so 
this MMP terminates in a Mori fibre space $g\colon X' \rightarrow S$ over $B$ with $\text{dim}(S) = 2$. Let $h\colon S\to B$ be the induced morphism. As in the proof of Lemma \ref{KXMMPreduction} we see that $\phi$ only contracts curves tangent to $\cal L$. 
Note that since $K_{\cal L}$ is Cartier we also have that $K_{\cal L'}$ is Cartier, and so $K_{\cal L'} = g^*M$ where $M$ is a Cartier divisor on $S$
with $c_1(M)=0$.

By assumption, we know that the fibres of $g$ are generically not tangent
to $\cal L'$ and so we have a non-zero morphism 
\[
dg:(g_*\cal L')^{**} \cong M \rightarrow T_{S/B}.
\]
Since $c_1(M) = 0$ this immediately implies
that for a general $p \in B$ that $M\vert_{S_p} \cong T_{S/B}\vert_{S_p}$,
where $S_p$ is the fiber over $p$.

Since $\rho(X'/S) = 1$ and $K_{\cal L'}$ is $g$-trivial,
we see that $(g_*\cal L')^{**}$ is a rank one reflexive sheaf and
$c_1((g_*\cal L')^{**}) =0$.
Thus, we have $(g_*\cal L')^{**} = T_{S/B}(-\sum a_iF_i)$
where $F_i$ are supported on fibres of $S \rightarrow B$.

Since $c_1(T_{S/B}(-\sum a_iF_i)) = 0$ and noting
that $T_{S/B}^* = K_{S/B} - R$ where $R$ is the ramification divisor of $h$,
we have 
\[ 
c_1(K_{S/B} +\sum a_iF_i - R) = 0.
\]
Thus,
\[
K_{S/B} +\sum a_iF_i - R\sim_{\bb Q} h^*M,
\]
for some $\bb Q$-divisor $M$ on $B$ such that $c_1(M)=0$. 
We apply \cite[Theorem 3.5]{Ambro05} to conclude that $M\sim_{\bb Q}0$
and hence
$K_{\cal L}$ is torsion.

In Case (\ref{i_qtoric}), we proceed in a similar fashion to the above case with a few modifications.
We first claim that we may assume without loss of generality that $\rho(X_p) \geq 2$.
Since $\cal L_p$ is singular for a general $p$ it follows
that there exists a component $\Sigma \subset \sing {\cal L}$
such that $\Sigma$ dominates $B$.  Let $b\colon \widetilde{X} \rightarrow X$ be a blow up centred in $\Sigma$
followed by a resolution of singularities which is a sequence of blow ups in foliation invariant centres,
again which exists by \cite{BM97}.  Let $\widetilde{\cal L}$ be the pulled back foliation
and observe that $K_{\widetilde{\cal L}} = b^*K_{\cal L}$.  By construction we have $\rho(\widetilde{X}_p) \geq 2$,
$c_1(K_{\widetilde{\cal L}}) = 0$ and $\widetilde{\cal L}$ has canonical singularities.   So we may freely replace $X$ by
$\widetilde{X}$ as required.

For general $p$ let $D_{0, p}$ and $D_{\infty, p}$ denote two $\cal L_p$-invariant divisors
which are fibres in a $\bb P^1$-fibration structure on $X_p$. 

First, assume that there exist two divisors $D_0$ and $D_\infty$ on $X$ such that $D_0\cap X_p = D_{0, p}$
and $D_\infty \cap X_p = D_{\infty, p}$ for general $p$ and that
if we run a $K_X$-MMP $\phi\colon X \dashrightarrow X'$ over $B$, we terminate in
a Mori fibre space $g\colon X' \rightarrow S$ where $S$ is a surface, $D_0$ and $D_\infty$ are not contracted by $\phi$ 
and $g$ contracts the strict transforms
of $D_0$ and $D_\infty$, call them $D'_0$ and $D'_\infty$.
Arguing as above we see that $c_1((g_*\cal L)^{**}) = T_{S/B}(-\Sigma_0 - \Sigma_\infty-\sum a_iF_i)$
where $\Sigma_0 = g(D'_0)$ and $\Sigma_\infty=g(D'_\infty)$ are reduced divisors dominating $B$  
and $F_i$ are supported on fibres of $S \rightarrow B$.  

Since $c_1(T_{S/B}(-\Sigma_0 - \Sigma_\infty-\sum a_iF_i)) = 0$ and noting
that $T_{S/B}^* = K_{S/B} - R$ where $R$ is the ramification divisor of $h$
we have 
\[c_1(K_{S/B} +\Sigma_0+\Sigma_\infty+\sum a_iF_i - R) = 0.\]
Thus,
\[K_{S/B} +\Sigma_0+\Sigma_\infty+\sum a_iF_i - R\sim_{\bb Q} h^*M.\]
We again apply \cite[Theorem 3.5]{Ambro05} (see also \cite[Theorem 1.3]{Floris14})
 to conclude that $M\sim_{\bb Q}0$
and hence
$K_{\cal L}$ is torsion.

Thus, to conclude it suffices to arrange the existence of $D_0$ and $D_\infty$ and such 
a Mori fibre space structure. Let $p \in B$ be a general point.
Observe that $\cal L$ is singular and so by assumption $\rho(X_p) \geq 2$.
Then we may find a sufficiently small \'etale neighborhood $U$ of $p$ such that $X \times_B U$ admits divisors $D_0$ and $D_\infty$ as
required and 
an MMP over $U$ terminating
in the desired Mori fibre space structure.  Thus, we may find a (possibly ramified) cover $\overline{B} \rightarrow B$
such that $\overline{X} = X \times_B \overline{B}$ admits such an MMP over $\overline{B}$.  
Let $\overline{\cal L}$ be the pulled back foliation. 
Observe
that $\sigma\colon \overline{X} \rightarrow X$ is ramified only along $\overline{\cal L}$-invariant divisors
and so $K_{\overline{\cal L}} = \sigma^*K_{\cal L}$ and thus we may freely replace $X$ and $\cal L$ by $\overline{X}$ and  $\overline{\cal L})$ respectively
and our result follows.
\end{proof}

\begin{corollary}
\label{c_fibration}
Suppose we have a morphism $f\colon X \rightarrow C$ where $C$ is a smooth curve of positive genus
and suppose that the general fibre does not admit a quasi-\'etale cover by an abelian surface.

Then $\kappa(K_{\cal F}) = 0$.
\end{corollary}
\begin{proof}
By Lemma \ref{algintcase},
we may assume that $\cal F$ is generically transverse to $f$.
Let $\cal L$ be the foliation in curves tangent to both $\cal F$ and the fibration $f\colon X \rightarrow C$.
We have an exact sequence
\[
0 \rightarrow \cal L \rightarrow \cal F \rightarrow (f^*T_C)\otimes I_Z \rightarrow 0
\]
where $Z$ is supported on the components of fibres
which are $\cal F$-invariant and on subvarieties of codimension at least 2. Thus, we have
$K_{\cal F} = K_{\cal L}+f^*K_C+D$ where $D \geq 0$ is the codimension one part of $Z$.

By Lemma \ref{notuniruled}, $\cal F$ is not uniruled and so we know that $K_{\cal L}$ is pseudo-effective.
By assumption, $f^*K_C$ is nef, and since $K_{\cal F}$ is numerically trivial we must
have $C$ is genus one and $D = 0$.  In particular, since $D = 0$ we see that no component of a fibre
of $f$ can be invariant under $\cal F$.

So $K_{\cal L}\sim_{\bb Q} K_{\cal F}$ and,
in particular, it suffices to prove that $K_{\cal L}$ is torsion.
Observe that $\cal L$ has canonical singularities above the generic point of $C$.
We may now apply Lemma \ref{l_rank1_abun} to conclude.
%
\end{proof}

\begin{proposition}
\label{KFcanonical}
Suppose $\cal F$ has canonical singularities and $c_1(K_{\cal F})= 0$.  

Then
$\kappa(K_{\cal F}) = 0$.  
\end{proposition}
\begin{proof}
First, assume that $K_X$ is pseudo-effective.
Then $\kappa(K_X) \geq 0$ and we apply Lemma \ref{KXpsefcase} to conclude.

Now, assume that $K_X$ is not pseudo-effective. By Lemma \ref{KXMMPreduction}, we may assume that $X$ admits a Mori fibre space
$f\colon X \rightarrow B$.  If $B$ is a surface with $\kappa(B) \geq 0$ or $f$ is tangent to $\cal F$,
we may apply Lemmas \ref{maptosurface} and \ref{tosurfacefibrestangent}
to conclude (note that since $f\colon X \rightarrow B$ is a Mori fibre space it follows that $f^{-1}(b)$ does not
contain a divisor for any $b$ and that $B$ is klt since $X$ is).  If $B$ is a rationally connected surface then $X$ is rationally connected and therefore it
is immediate that $K_{\cal F}$ is torsion.

Otherwise, we may find a map $X \rightarrow C$ where $C$ is a curve.  If $C$ is of positive genus then 
observe that if $F$ is a general fibre then $-K_F \neq 0$ and so it does not admit a quasi-\'etale cover by an abelian surface and
we conclude by Corollary \ref{c_fibration}. Otherwise $X$ is rationally connected and again it is
immediate that $K_{\cal F}$ is torsion.
\end{proof}

\subsection{Proof of Theorem \ref{threefoldlogabundancezero}}
\begin{proof}
If $\Delta \neq 0$ or $\cal F$ is not canonical we apply
Proposition \ref{KFlogcanonical}.

If $\Delta = 0$ and $\cal F$ is canonical we apply
Proposition \ref{KFcanonical}.
\end{proof}

\bibliography{math}
\bibliographystyle{alpha}
\end{document}